\documentclass[12pt,twoside]{article}
\usepackage{amsmath,amsfonts,amssymb,amsthm,array,mathdots}
\usepackage{a4,enumitem,indentfirst}
\usepackage{etoolbox} 
\usepackage{url} 

\usepackage{stmaryrd} 

\usepackage{xr}

\usepackage{epsfig, color} 
\usepackage{tikz}
\usepackage{pgfplots}\pgfplotsset{compat=1.13}
\usetikzlibrary{arrows,positioning,calc,patterns}

\graphicspath{{figures/}}

\theoremstyle{plain}

\newtheorem{theo}{Theorem}
\newtheorem{prop}{Proposition}[section]
\newtheorem{conj}[prop]{Conjecture}
\newtheorem{coro}[prop]{Corollary}

\newtheorem{lemma}[prop]{Lemma}

\theoremstyle{definition}
\newtheorem{example}[prop]{Example}
\AtEndEnvironment{example}{\null\hfill$\diamond$}%

\newtheorem{rem}[prop]{Remark}
\AtEndEnvironment{rem}{\null\hfill$\diamond$}%

\newcommand\BL{\operatorname{BL}}
\newcommand\BLC{\operatorname{BLC}}
\newcommand\BLS{\operatorname{BLS}}
\newcommand\old{\operatorname{old}}
\newcommand\Flag{\operatorname{Flag}}
\newcommand\card{\operatorname{card}}
\newcommand\med{\operatorname{med}}
\newcommand\Frenet{\mathfrak{F}}

\newcommand{\Word}{{\mathbf W}}
\newcommand{\Ideal}{{\mathbf I}}
\newcommand{\Filter}{{\mathbf U}}

\newcommand{\SubMatrix}{\operatorname{SubMatrix}}
\newcommand{\SW}{\operatorname{SW}}
\newcommand{\conv}{\operatorname{convex}}
\newcommand{\nconv}{\operatorname{non-convex}}

\newcommand{\tok}{\preceq}

\newcommand{\Sing}{\operatorname{Sing}}
\newcommand{\sing}{\operatorname{sing}}
\newcommand{\PA}{\operatorname{PA}}
\newcommand{\PPre}{\operatorname{PP}}

\newcommand\SO{\operatorname{SO}}

\newcommand\GL{\operatorname{GL}}
\newcommand\Lo{\operatorname{Lo}}
\newcommand\Up{\operatorname{Up}}

\newcommand\so{\operatorname{\mathfrak{so}}}

\newcommand\Spin{\operatorname{Spin}}
\newcommand\spin{\mathfrak{spin}}

\newcommand\inv{\operatorname{inv}}

\newcommand\Diag{\operatorname{Diag}}
\newcommand\diag{\operatorname{diag}}
\newcommand\B{\operatorname{B}}
\newcommand\Quat{\operatorname{Quat}}

\newcommand{\resultant}{\operatorname{resultant}}
\newcommand{\discriminant}{\operatorname{discriminant}}
\newcommand{\longacute}{\operatorname{acute}}
\newcommand{\longgrave}{\operatorname{grave}}

\newcommand{\longhat}{\operatorname{hat}}

\newcommand{\iti}{\operatorname{iti}}

\newcommand{\bM}{{\mathbf{M}}}

\newcommand{\nmesmo}{\llbracket  n \rrbracket}
\newcommand{\nmaisum}{\llbracket n+1 \rrbracket}

\newcommand{\interior}{\operatorname{int}}

\newcommand{\mult}{\operatorname{mult}}
\newcommand{\bump}{\operatorname{bump}}

\newcommand{\NN}{{\mathbb{N}}}
\newcommand{\ZZ}{{\mathbb{Z}}}

\newcommand{\RR}{{\mathbb{R}}}

\newcommand{\Ss}{{\mathbb{S}}}
\newcommand{\BB}{{\mathbb{B}}}
\newcommand{\DD}{{\mathbb{D}}}

\newcommand{\HH}{{\mathbb{H}}}

\newcommand{\cL}{{\cal L}}

\newcommand{\cU}{{\cal U}}

\newcommand{\cA}{{\cal A}}

\newcommand{\cD}{{\cal D}}

\newcommand{\cM}{{\cal M}}
\newcommand{\cY}{{\cal Y}}
\newcommand{\cX}{{\cal X}}

\newcommand{\bQ}{{\mathbf{Q}}}

\newcommand{\fa}{{\mathfrak a}}

\newcommand{\Pos}{\operatorname{Pos}}
\newcommand{\Neg}{\operatorname{Neg}}
\newcommand{\Bru}{\operatorname{Bru}}

\begin{document}

\title{
A CW complex homotopy equivalent to \\ spaces of locally convex curves}
\author{Victor Goulart
\and Nicolau C. Saldanha}
\date{\today}

\maketitle

\begin{abstract}
Locally convex (or nondegenerate) curves in the sphere $\Ss^n$
(or the projective space)
have been studied for several reasons,
including the study of linear ordinary differential equations
of order $n+1$.
Taking Frenet frames allows us to obtain
corresponding curves $\Gamma$ in the group $\Spin_{n+1}$;
recall that $\Pi: \Spin_{n+1} \to \Flag_{n+1}$
is the universal cover of the space of flags.
Determining the homotopy type of
spaces of such curves $\Gamma$ with prescribed
initial and final points appears to be a hard problem.
Due to known results, 
we may focus on $\cL_n$,
the space of (sufficiently smooth) locally convex curves
$\Gamma: [0,1] \to \Spin_{n+1}$
with $\Gamma(0) = 1$ and $\Pi(\Gamma(1)) = \Pi(1)$.
Convex curves form a contractible connected component of $\cL_n$;
there are $2^{n+1}$ other components,
corresponding to non convex curves, one for each endpoint.
The homotopy type of $\cL_n$ has so far been determined only for $n=2$
(the case $n = 1$ is trivial).
This paper is a step towards solving the problem
for larger values of $n$.

The \textit{itinerary} of a locally convex curve
$\Gamma: [0,1] \to \Spin_{n+1}$ belongs to $\Word_n$,
the set of finite words in the alphabet $S_{n+1} \smallsetminus \{e\}$.
The itinerary of a curve lists
the non open Bruhat cells crossed by the curve.
Itineraries yield a stratification of the space $\cL_n$.
We construct a CW complex $\cD_n$ which is a kind of dual
of $\cL_n$ under this stratification:
the construction is similar to Poincaré duality.
The CW complex $\cD_n$ is homotopy equivalent to $\cL_n$.
The cells of $\cD_n$ are naturally labeled by words in $\Word_n$
so that $\cD_n$ is infinite but locally finite.
Explicit glueing instructions are described for lower dimensions.

As an application, we describe an open subset $\cY_n \subset \cL_n$,
a union of strata of $\cL_n$.
In each non convex component of $\cL_n$,
the intersection with $\cY_n$ is connected and dense.
Most connected components of $\cL_n$ are contained in $\cY_n$.
For $n > 3$, in the other components
the complement of $\cY_n$ has codimension at least $2$.
We prove that $\cY_n$ is homotopically equivalent to
the disjoint union of $2^{n+1}$ copies of $\Omega\Spin_{n+1}$.
In particular, for all $n \ge 2$,
all connected components of $\cL_n$ are simply connected.
\end{abstract}

\medskip

\section{Introduction}
\label{sect:intro}

Let $J \subset \RR$ be an interval.
A sufficiently smooth curve $\gamma: J \to \Ss^n \subset \RR^{n+1}$
is \emph{(positive) locally convex} 
\cite{Alves, Alves-Saldanha, Saldanha3, Saldanha-Shapiro}
or \emph{(positive) nondegenerate} 
\cite{Goulart-Saldanha, Khesin-Ovsienko, Khesin-Shapiro2, Little} 
if it satisfies 
\[ \forall t \in J, \;
\det(\gamma(t),\gamma'(t),\ldots,\gamma^{(n)}(t))>0. \]
Such a curve $\gamma$ can be associated with
$\Gamma = \Frenet_{\gamma}: J \to \SO_{n+1}$ 
where the column-vectors of $\Gamma(t)$
are the result of applying the 
Gram-Schmidt algorithm to the ordered basis 
$(\gamma(t),\gamma'(t),\ldots,\gamma^{(n)}(t))$ of $\RR^{n+1}$. 
The curve $\Gamma$ can then be lifted 
to the double cover $\Spin_{n+1}$.
The curve $\Gamma: J \to \Spin_{n+1}$
is also called locally convex; 
a direct description is in order.

For $j\in\nmesmo=\{1,2, \ldots,n\}$, 
consider the skew-symmetric tridiagonal matrices 
$\fa_j=e_{j+1}e_j^\top-e_je_{j+1}^\top\in\so_{n+1}$ 
as elements of the Lie algebra $\spin_{n+1} = \so_{n+1}$.
The identification between the two Lie algebras
is induced by the covering map $\Pi: \Spin_{n+1} \to \SO_{n+1}$.
A sufficiently smooth curve $\Gamma:J\to\Spin_{n+1}$ 
is called \emph{locally convex} if
its logarithmic derivative is of the form 
\begin{equation}
\label{equation:locallyconvex}
(\Gamma(t))^{-1}\Gamma'(t)=
\sum_{j\in\nmesmo}\kappa_j(t)\fa_j, 
\end{equation} 
for positive functions $\kappa_1,\ldots,\kappa_n:J\to(0,+\infty)$. 
Recall that $\Pi_{\Flag}: \Spin_{n+1} \to \Flag_{n+1}$ 
is the universal cover of the flag space.
Let $\Quat_{n+1} = \Pi_{\Flag}^{-1}[\Pi_{\Flag}[\{1\}]] \subset \Spin_{n+1}$:
this is a subgroup of order $2^{n+1}$,
isomorphic to $\pi_1(\Flag_{n+1})$.

Given sufficiently large $r\in\NN^\ast = \{1, 2, 3, \ldots\}$
and $z_0,z_1\in\Spin_{n+1}$, 
let $\cL_n^{[C^r]}(z_0;z_1)$ denote the space of 
locally convex curves $\Gamma:[0,1]\to\Spin_{n+1}$ 
of differentiability class $C^r$ with endpoints 
$\Gamma(0)=z_0$ and $\Gamma(1)=z_1$. 
We endow this space with the usual  
$C^r$ topology and consider the problem 
of describing its homotopy type. 
It is easy to see that $\cL_n^{[C^r]}(z_0;z_1)$ 
is homeomorphic to $\cL_n^{[C^r]}(1;z_0^{-1}z_1)$
so we assume $z_0 = 1$.
As discussed in \cite{Goulart-Saldanha1} (among others),
the value of $r$ does not affect the homotopy type
so we drop it and write $\cL_n(1;z)$.
It is proved in \cite{Goulart-Saldanha0,Saldanha-Shapiro} that
for every $z \in \Spin_{n+1}$ there exists $q \in \Quat_{n+1}$
such that $\cL_n(1;z)$ and $\cL_n(1;q)$
are homotopy equivalent
(the map from $z$ to $q$ is explicitly described).
We aim to study the homotopy type of $\cL_n(1;q)$, 
$q \in \Quat_{n+1}$, or equivalently that of
the disconnected space
\[ \cL_n = \bigsqcup_{q \in \Quat_{n+1}} \cL_n(1;q). \]
In this paper we contruct and study a CW complex $\cD_n$
which is homotopy equivalent to $\cL_n$.
We shall see that $\cD_n$ is infinite but locally finite.
Cells are naturally labeled by finite words
in the alphabet $S_{n+1} \smallsetminus \{e\}$
($e \in S_{n+1}$ being the identity):   
such words correspond to \textit{itineraries}
of curves $\Gamma \in \cL_n$ (itineraries are discussed
at length in \cite{Goulart-Saldanha1}).
We proceed to discuss the symmetric group $S_{n+1}$
and itineraries of curves.
In the process we review the main results from \cite{Goulart-Saldanha1}. 

\bigbreak

We consider the symmetric group $S_{n+1}$ as a Coxeter-Weyl group,
that is, we use as generators the transpositions
$a_j = (j,j+1)$, $1 \le j \le n$.
The number of inversions $\inv(\sigma)$ of a permutation $\sigma$
is its length with these generators.
Notice that for $\sigma \in S_{n+1}$ and
$j \in \nmaisum = \{1, 2, \ldots, n, n+1\}$
we write $j^\sigma$ (instead of $\sigma(j)$);
composition is defined by
$(j^{\sigma_0})^{\sigma_1} = j^{(\sigma_0\sigma_1)}$.
Let $\eta\in S_{n+1}$ be the top permutation,
i.e., the permutation with the maximum number of inversions: 
$\inv(\eta)=n(n+1)/2$, $j^\eta = n+2-j$. 
The element $\eta$ is the longest element of $S_{n+1}$ 
with the above generators
and is often denoted in the literature by $w_0$.
Permutations are often written in the generators $a_j$;
for $n \le 4$ we also write $a = a_1$, $b = a_2$, $c = a_3$ and $d = a_4$.
Thus, for instance, for $n = 2$ we have $\eta = aba = bab$;
for $n = 3$ we have $\eta = abacba$.

Let $\B_{n+1}$ be the 
group of signed permutation matrices;
this is also Coxeter-Weyl group 
but we shall not use such generators.
Let $\B_{n+1}^{+} = \B_{n+1}\cap \SO_{n+1}$.
Let $\widetilde\B_{n+1}^{+} \subset \Spin_{n+1}$
be the preimage of $\B_{n+1}^{+}$ 
under the covering map $\Pi: \Spin_{n+1} \to \SO_{n+1}$.
The subgroup $\Quat_{n+1} \subset \widetilde\B_{n+1}^{+}$
is the preimage (under $\Pi$)
of the subgroup $\Diag_{n+1}^{+} \subset \B_{n+1}^{+}$
of diagonal matrices.
We have a short exact sequence
\begin{equation}
\label{equation:Quat}
1 \to \Quat_{n+1} \to \widetilde\B_{n+1}^{+} \to S_{n+1} \to 1. 
\end{equation}
The spin group is stratified by contractible Bruhat cells
$\Bru_z \subset \Spin_{n+1}$,
$z \in \widetilde\B_{n+1}^{+}$.
The unsigned (and more usual) Bruhat cells are 
\begin{equation}
\label{equation:Bruz}
\Bru_\sigma = \bigsqcup_{z \in \Pi^{-1}[\{\sigma\}]} \Bru_z,
\qquad \sigma \in S_{n+1}, 
\end{equation}
where $\Pi: \widetilde\B_{n+1}^{+} \to S_{n+1}$
is the homomorphism in the exact sequence above
(we shall say more about the subsets
$\Pi^{-1}[\{\sigma\}] \subset \widetilde\B_{n+1}^{+}$ below).
For $z \in \widetilde\B_{n+1}^{+}$ with $\Pi(z) = \sigma$,
$\Bru_z \subset \Bru_\sigma$
is the connected component of $\Bru_\sigma$ containing $z$.

The set $\Bru_\eta$ is a dense open subspace of the spin group
with $2^{n+1}$ contractible connected components. 
Let $\Sing_{n+1}=\Spin_{n+1}\smallsetminus\Bru_\eta$ 
be the union of the non open Bruhat cells.
We define the \emph{singular set} of a 
locally convex curve $\Gamma:[0,1]\to\Spin_{n+1}$ as 
$\sing(\Gamma)= \Gamma^{-1}[\Sing_{n+1}]\smallsetminus\{0,1\}$.
It turns out that $\sing(\Gamma)$ is finite
and a continuous function of $\Gamma \in \cL_n$
(with the Hausdorff metric in compact subsets of $(0,1)$):
this is Theorem 1 in \cite{Goulart-Saldanha1}. 
The elements of $\sing(\Gamma)$ are sometimes called
the \textit{moments of non-transversality} of $\Gamma$
(as in \cite{Saldanha-Shapiro-Shapiro}).
Given $\Gamma \in \cL_n$,
write $\sing(\Gamma) = \{\tau_1 < \cdots < \tau_\ell\}$.
Let $\sigma_1, \ldots, \sigma_\ell \in S_{n+1}$
be such that $\Gamma(\tau_j) \in \Bru_{\eta\sigma_j}$:
the itinerary of $\Gamma$ is the word
$\iti(\Gamma) = (\sigma_1, \ldots, \sigma_{\ell})$.
Figure \ref{fig:aba-cw} shows itineraries of a few
locally convex curves in $\Ss^2$.

\begin{figure}[ht]
\def\svgwidth{75mm}
\centerline{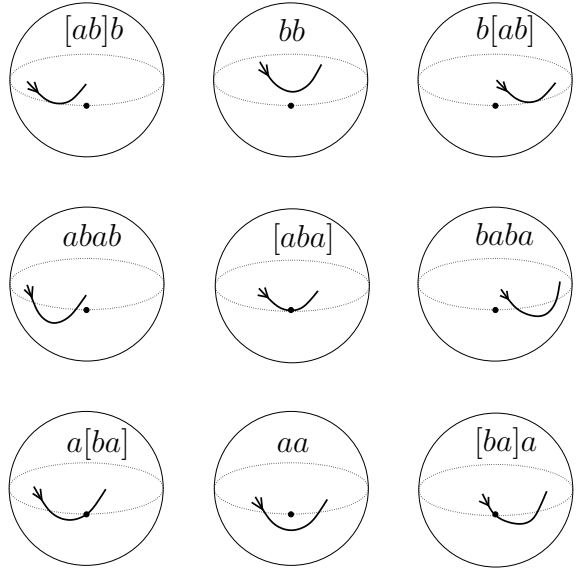}
\caption{A family of curves in $\cL_2$.
The equator is dashed and the fat dot indicates $e_1$.
The vector $e_2$ is at the right.
Here we use the simplified notation for words in $\Word_{2}$:
for instance, $abab$ denotes $(a,b,a,b)$ and $[aba]$ denotes $(aba)$.}
\label{fig:aba-cw}
\end{figure}


\smallskip

Let $\Word_{n}$ be the set of finite words in the 
alphabet $S_{n+1}\smallsetminus\{e\}$.
A simplified notation for
the word $w =(\sigma_1,\ldots,\sigma_\ell)\in\Word_n$
is often used, which we exemplify for a few words in $\Word_3$:
\[ acbac = (a,c,b,a,c), \quad
[ac]b[ac] = (ac,b,ac), \quad [acb] = (acb). \]
\goodbreak
For $w =(\sigma_1,\ldots,\sigma_\ell)\in\Word_n$,
define
\begin{equation}
\label{equation:dim}
\dim(w) = \dim(\sigma_1) + \cdots + \dim(\sigma_\ell), 
\qquad \dim(\sigma) = \inv(\sigma) - 1. 
\end{equation}
Let $\cL_n[w] \subset \cL_n$ be the set of curves $\Gamma$
with $\iti(\Gamma) = w$.
As in \cite{Goulart-Saldanha1},
we prefer to work with Hilbert spaces,
more precisely with the Sobolev space $H^r$.
For large $r \in \NN$,
let $\cL_n^{[H^r]}$ be the Hilbert manifold
of locally convex curves $\Gamma$ of class $H^r$
with the usual boundary conditions.
From now on, assume $\cL_n = \cL_n^{[H^r]}$;
recall that this does not affect the homotopy type of $\cL_n$.
Theorem 2 in  \cite{Goulart-Saldanha1} 
tells us that $\cL_n[w] \subset \cL_n$ 
is a (non empty) contractible embedded submanifold
of codimension $\dim(w)$ and of class $C^{r-1}$.
Notice that  \cite{Goulart-Saldanha1} 
discusses the case $H^1$ in considerable detail:
here we consistently prefer the case $H^r$, $r$ large.

\bigbreak

This stratification by itineraries of $\cL_n$
is not as well behaved as might be desired.
We do not, for instance, have the Whitney property.
Indeed, for $n = 3$ and
$w_0 = [ac]b[ac] = (a_1a_3,a_2,a_1a_3)$ and $w_1 = [acb] = (a_1a_3a_2)$
we have
\[ \cL_n[w_1] \cap \overline{\cL_n[w_0]} \ne \varnothing,
\qquad
\cL_n[w_1] \not\subseteq \overline{\cL_n[w_0]}. \]
Notice that $\cL_3[w_0], \cL_3[w_1] \subset \cL_3$
both have codimension 2:
this example is extensively discussed in Section 9
of \cite{Goulart-Saldanha1}. 
On the other hand,
the stratification of $\cL_n$ is well behaved enough for
the construction of a homotopically equivalent CW complex $\cD_n$.
In a sense, the CW complex is a dual of the stratification.

The maps
$\longacute, \longgrave: S_{n+1} \to \widetilde\B_{n+1}^{+}$
and $\longhat:  S_{n+1} \to \Quat_{n+1}$
are discussed in \cite{Goulart-Saldanha0}
and revised in Section \ref{sect:permutations} below.
For $w \in\Word_n$ define $\hat w \in \Quat_{n+1}$ by
\[ w = \sigma_1 \cdots \sigma_\ell =
(\sigma_1,\ldots,\sigma_\ell)\in\Word_n, \qquad
\hat w=\hat\sigma_1\cdots\hat\sigma_\ell\in\Quat_{n+1};  \]
here $\hat\sigma = \longhat(\sigma) \in \Quat_{n+1}$.
Theorem 2 in  \cite{Goulart-Saldanha1} 
tells us that
$\cL_n[w] \subset \cL_n(1;\acute\eta\hat w\acute\eta) \subset \cL_n$. 
Thus, words with different values of $\longhat$
do not interact directly and can be studied separately.

\bigbreak

There is a form of duality (a variant of Poincar{\'e} duality)
which allows us to relate the stratification of $\cL_n$
with a rather explicit CW complex $\cD_n$.
We now present a short statement;
a more detailed one is given in Theorem \ref{theo:newCW}.
The statement of Theorem \ref{theo:newCW}
assumes familiarity with the partial order $\sqsubseteq$ in $\Word_n$
(discussed in Section \ref{sect:partial}),
a variant of another partial order $\preceq$ in $\Word_n$
(introduced in \cite{Goulart-Saldanha1},
and based on previous work by Shapiro and Shapiro).

\begin{theo}
\label{theo:shortCW}
There exists a CW complex $\cD_n$
with one cell $c_w$ of dimension $\dim(w)$
for each word $w \in \Word_n$.
Furthermore, there exists a continuous map $i: \cD_n \to \cL_n$
which is a homotopy equivalence.
\end{theo}

A similar but simpler example is presented in \cite{Alves-Saldanha2}:
there we have a finite dimensional manifold
$\BL_\sigma \subset \Lo^{1}_{n+1}$ ($\sigma \in S_{n+1}$)
which is stratified by finitely many contractible submanifolds
$\BLS_\varepsilon \subseteq \BL_\sigma$.
A finite homotopically equivalent CW complex $\BLC_\sigma$ is then constructed.
The result in \cite{Alves-Saldanha2} most relevant
for the above description is Theorem 2
(for the homotopy equivalence $i_\sigma: \BLC_\sigma \to \BL_\sigma$),
which in turn relies on Lemma 14.1
(a topological result which validates the inductive step).
The present paper does not presuppose these results.

Glueing maps for cells of dimensions $1$ and $2$ are described explicitly.
This allows us to construct 
the $1$ and $2$-skeletons of $\cD_n$
in Sections \ref{sect:Dn1} and \ref{sect:Dn2}.
The following result is then proved.

\begin{theo}
\label{theo:simplyconnected}
Let $z_0, z_1 \in \Spin_{n+1}$.
Every connected component of $\cL_n(z_0;z_1)$ is simply connected.
\end{theo}

In order to state our third main result we need
a few algebraic definitions.
A permutation $\sigma \in S_{n+1}$ is \textit{parity alternating}
if $k^\sigma \not\equiv (k+1)^\sigma \pmod 2$
for all $k$, $1 \le k \le n$.
We shall see that $\sigma$ is parity alternating
if and only if $\hat\sigma \in Z(\Quat_{n+1})$
(the center of $\Quat_{n+1}$).
For $n > 3$, if $\sigma \ne e$ is parity alternating
then $\inv(\sigma) \ge 3$.
Let $\Ideal_Y \subset \Word_n$ be the set of words
containing at least one letter which is \textit{not} parity alternating.
It is not hard to verify that $\Ideal_Y \subset \Word_n$ is a lower set 
with respect to the partial orders
$\preceq$ and $\sqsubseteq$.
It then follows from \cite{Goulart-Saldanha1}
(or from Theorem \ref{theo:newCW})
that $\cY_n = \cL_n[\Ideal_Y] \subset \cL_n$ is an open subset.
For $z \in \Spin_{n+1}$,
let $\Omega\Spin_{n+1}(1;z)$ be the space of all continuous curves
$\Gamma: [0,1] \to \Spin_{n+1}$, $\Gamma(0) = 1$ and $\Gamma(1) = z$.
The space  $\Omega\Spin_{n+1}(1;z)$ is homeomorphic
to the loop space  $\Omega\Spin_{n+1} = \Omega\Spin_{n+1}(1;1)$ 
and we have a natural inclusion
$\cL_n(1;z) \subset \Omega\Spin_{n+1}(1;z)$.

\begin{theo}
\label{theo:Y}
Let $\cY_n \subset \cL_n$ be the open subset constructed above.
For $q \in \Quat_{n+1}$ let $\cY_n(1;q) = \cY_n \cap \cL_n(1;q)$.
For any $q \in \Quat_{n+1}$,
the open subset $\cY_n(1;q) \subset \cL_n(1;q)$
is dense in the connected component of non convex curves
and the inclusion $\cY_n(1;q) \subset \Omega\Spin_{n+1}$
is a homotopy equivalence.
\end{theo}

The case $n = 2$ of the above theorem
is the hardest result in \cite{Saldanha3}.
The set $\cY_2$ is there called the set of \textit{complicated} curves and 
the complement $\cL_2 \smallsetminus \cY_2$ is the closed subset
of \textit{multiconvex} curves.
We may therefore say that the basic structure
of the arguments in \cite{Saldanha3},
in the case $n=2$, seems to survive in the cases $n>2$,
except that the complement $\cL_n \smallsetminus \cY_n$
does not admit a correspondingly direct description.
We could also say that we have to study 
what is the correct generalization of the concept of a multiconvex curve
to the cases $n > 2$.
We expect to address these issues in future work.

The case $n = 3$ deserves special attention;
in Section \ref{sect:finalremarks} we state and prove
an alternative to Theorem \ref{theo:Y} for this special case.
We also state in Theorem \ref{theo:L3}
a full description of the homotopy type of $\cL_3$:
this is the main result in \cite{Alves-Goulart-Saldanha}.

The following is an easy consequence of Theorem \ref{theo:Y}.
Notice that, for large $n$, it gives us the homotopy type
for most of the connected components of $\cL_n$.


\begin{coro}
\label{coro:notcenter}
If $q \in \Quat_{n+1}$ does not belong to its center
$Z(\Quat_{n+1}) \subset \Quat_{n+1}$ then
the inclusion $\cL_n(1;q) \subset \Omega\Spin_{n+1}(1;q)$
is a homotopy equivalence.
\end{coro}

We expect to discuss higher dimensional skeletons
and glueing maps in future work.
Some conjectures are listed in Section \ref{sect:finalremarks}:
in particular, we believe that if $q \in Z(\Quat_{n+1})$
then $\cL_n(1;q)$ is \textit{not} homotopy equivalent
to $\Omega\Spin_{n+1}$.

\bigbreak

Our problem is related to the study of 
linear ordinary differential operators.
This point of view was the original motivation of V. Arnold, B. Khesin, 
V. Ovsienko, B. Shapiro and M. Shapiro for considering this class of questions 
in the early nineties 
\cite{Khesin-Ovsienko, Khesin-Shapiro1, Khesin-Shapiro2, Shapiro-Shapiro3}.
The second author was first led to consider  
this subject while studying the critical 
sets of nonlinear differential operators with periodic 
coefficients, in a series of works with D. Burghelea 
and C. Tomei 
\cite{Burghelea-Saldanha-Tomei1,Burghelea-Saldanha-Tomei2,Burghelea-Saldanha-Tomei3,Saldanha-Tomei}. 

\smallskip

In Section~\ref{sect:permutations} we see several results
concerning the combinatorics of the symmetric group $S_{n+1}$:
these results will of course be useful later.
In Section~\ref{sect:bruhat} we review classical results
about Bruhat cells, some recent results from our previous papers
(\cite{Goulart-Saldanha0, Goulart-Saldanha1})
and a few other results which will be needed later.
In Section~\ref{sect:partial} we present the two partial orders
$\preceq$ and $\sqsubseteq$ in $\Word_n$:
we see some examples and conjectures.
In Section~\ref{sect:lowersets} we continue to study
$\Word_n$ as a poset:
we are now interested in lower and upper sets.
At this point, in Section~\ref{sect:valid},
we are ready to state and prove Theorem~\ref{theo:newCW},
a stronger and more detailed version of Theorem~\ref{theo:shortCW},
and to construct the CW complex $\cD_n$.
After we define some auxiliary concepts and state Lemma~\ref{lemma:goodstep} 
(a technical result, essentially the inductive step
in the proof of Theorem~\ref{theo:newCW}),
most of the section is dedicated to the proof of Lemma~\ref{lemma:goodstep}.
In Section~\ref{sect:Dn1} we study the $1$-skeleton
of our newly obtained CW complex $\cD_n$:
as an exercise, we give a new proof of the well known
classification of the connected components of $\cL_n$.
In Section~\ref{sect:Dn2} we study the $2$-skeleton of $\cD_n$:
this is the highest dimension to be completely described.
Some partial results concerning higher dimensions
are discussed in Section~\ref{sect:boundary}:
this is where the combinatorial results
from Section~\ref{sect:permutations} are applied.
The concept of \textit{loose} maps has been discussed in previous papers:
it is revised and adapted in Section~\ref{sect:loosetight}.
In Section~\ref{sect:looseideals} this concept is adapted
to ideals (subsets of the poset $\Word_n$):
Proposition~\ref{prop:ideal0isloose} is a pivotal result
with a long and technical proof.
In Section~\ref{sect:simplyconnected}
we finally prove Theorem~\ref{theo:simplyconnected}.
In Section~\ref{sect:Yn2} we prepare for the proof of Theorem~\ref{theo:Y}
by stating and proving Proposition~\ref{prop:Yn2},
which in a sense is an improvement of 
Proposition~\ref{prop:ideal0isloose}.
Theorem~\ref{theo:Y} is proved in Section~\ref{sect:Y}.
Section~\ref{sect:finalremarks}, the final remarks,
includes the statement of Theorem~\ref{theo:L3},
the main result in \cite{Alves-Goulart-Saldanha}:
this gives the homotopy type of $\cL_3$.
Theorem~\ref{theo:tildeY},
a retouched version of Theorem~\ref{theo:Y} for the case $n=3$,
aims specifically at the proof of Theorem~\ref{theo:L3}.
We also state as conjectures other results which
we hope to prove in future work.

\bigbreak

We would like to thank:  
Em\'ilia Alves, 
Boris Khesin, 
Ricardo Leite, 
Carlos Gustavo Moreira, 
Paul Schweitzer, 
Boris Shapiro, 
Michael Shapiro, 
Carlos Tomei, 
David Torres, 
Cong Zhou 
and 
Pedro Z\"{u}lkhe 
for helpful conversations; 
the University of Toronto and the University of Stockholm 
for the hospitality during our visits. 
Both authors thank CAPES, CNPq and FAPERJ (Brazil) for financial support.
More specifically, the first author benefited from
CAPES-PDSE grant 99999.014505/2013-04 
during his Ph. D. and also 
CAPES-PNPD post-doc grant 88882.315311/2019-01.

\section{Permutations}
\label{sect:permutations}

Consider the symmetric group $S_{n+1}$
(acting on $\nmaisum = \{1, 2, \ldots, n+1 \}$)
as the Coxeter-Weyl group of type $A_n$, i.e.,
use the $n$ generators
$a = a_1 = (12)$, $b = a_2 = (23)$, \dots, $a_n = (n,n+1)$.
For $\sigma \in S_{n+1}$ and $k \in \nmaisum$
we use the notation $k^\sigma$ (rather than $\sigma(k)$)
so that $(k^{\sigma_1})^{\sigma_2} = k^{(\sigma_1\sigma_2)}$.
Given $\sigma \in S_{n+1}$, let $P_\sigma$ be the $(n+1) \times (n+1)$
permutation matrix satisfying $e_i^\top P_\sigma = e_{i^\sigma}^\top$;
notice that $P_{\sigma_0} P_{\sigma_1} = P_{\sigma_0\sigma_1}$.
For $\sigma \in S_{n+1}$, let $\inv(\sigma)$ be the length of $\sigma$
with the generators $a_i$, $1 \le i \le n$.

For $\sigma_0, \sigma_1 \in S_{n+1}$, write $\sigma_0 \equiv \sigma_1 \pmod 2$
if $i^{\sigma_0} \equiv i^{\sigma_1} \pmod 2$ for all $i \in \nmaisum$.
A permutation $\sigma \in S_{n+1}$ is \emph{parity preserving}
if $i^\sigma \equiv i \pmod 2$ for all $i \in \nmaisum$, i.e., 
if $\sigma \equiv e \pmod 2$.
A permutation $\sigma \in S_{n+1}$ is \emph{parity alternating}
if $i^\sigma \not\equiv (i+1)^\sigma \pmod 2$ for all $i \in \nmesmo$.
The sets $S_{\PPre}$ and $S_{\PA}$ of parity preserving and parity alternating
permutations, respectively, are both subgroups:
$S_{\PPre} \le S_{\PA} \le S_{n+1}$.
If $n$ is even, $S_{\PPre} = S_{\PA}$;
if $n$ is odd, $S_{\PPre}$ is a subgroup of index $2$ of $S_{\PA}$.
We have $\sigma_0 \equiv \sigma_1 \pmod 2$ if and only if
$\sigma_1^{-1} \sigma_0 \in S_{\PPre}$, or, equivalently,
if $\sigma_0 S_{\PPre} = \sigma_1 S_{\PPre}$.

\begin{example}
\label{example:smallPA}
For $n = 2$ we have
$S_{\PA} = S_{\PPre} = \{e, \eta = aba\}$.
For $n = 3$ we have $S_{\PPre} = \{e, aba, bcb, bacb\}$
and $S_{\PA} = S_{\PPre} \cup \{ac, abc, cba, \eta = abacba\}$.
For $n = 4$ we have $S_{\PA} = S_{\PPre}$ and
\[ S_{\PPre}  = \{ e, aba, bacb, bcb, cbdc, cdc, 
abacbdc, abacdc, abcdcba, cdcaba, cbdcaba, \eta \}, \]
with $\eta = abacbadcba$.
\end{example}

The group $S_{\PPre}$ is isomorphic to
$S_{\lfloor\frac{n+1}{2}\rfloor} \times S_{\lceil\frac{n+1}{2}\rceil}$.
For odd $n$,
the top element $\eta$ and $\sigma = a_1 a_3 \cdots a_{n-2} a_n$
are both elements of $S_{\PA} \smallsetminus S_{\PPre}$:
$\sigma$ is the element in this set with minimal number of inversions,
equal to $(n+1)/2$.

The maps
$\longacute, \longgrave: S_{n+1} \to \widetilde\B_{n+1}^{+}$
and $\longhat:  S_{n+1} \to \Quat_{n+1}$
are discussed in \cite{Goulart-Saldanha0}
(these are not homomorphisms!).
We also write
$\acute\sigma = \longacute(\sigma)$,
$\grave\sigma = \longgrave(\sigma)$,
$\hat\sigma = \longhat(\sigma)$.
If $\Pi: \widetilde\B_{n+1}^{+} \to S_{n+1}$ is as 
in the exact sequence \eqref{equation:Quat} above
we have $\Pi(\acute\sigma) = \Pi(\grave\sigma) = \sigma$.
In order to define $\acute a_i \in \widetilde\B_{n+1}^{+}$,
consider the two preimages under the surjective homomorphism
$\widetilde\B_{n+1}^{+} \to \B_{n+1}$
of the signed permutation matrix $P$ with entries
$P_{i+1,i} = 1$,
$P_{i,i+1} = -1$ and $P_{j,j} = 1$ for $j \notin \{i,i+1\}$:
the element $\acute a_i$ is the preimage closest to $1 \in \Spin_{n+1}$.
We have $\grave a_i = (\acute a_i)^{-1}$.
For a reduced word $\sigma = a_{i_1}\cdots a_{i_\ell}$,
set 
\begin{equation}
\label{equation:acutegrave}
\acute\sigma = \acute a_{i_1}\cdots \acute a_{i_\ell}, \qquad
\grave\sigma = \grave a_{i_1}\cdots \grave a_{i_\ell}, \qquad
\hat\sigma = \acute\sigma(\grave\sigma)^{-1}.  
\end{equation}
For all $\sigma \in S_{n+1}$ we have
$\Pi^{-1}[\{\sigma\}] = \Quat_{n+1} \acute \sigma = \acute\sigma \Quat_{n+1}$.


\begin{lemma}
\label{lemma:PAhat}
A permutation $\sigma \in S_{n+1}$ is parity alternating
if and only if $\hat\sigma \in Z(\Quat_{n+1})$.
Also, $\sigma \in S_{\PPre}$ if and only if $\hat\sigma = \pm 1$.
\end{lemma}

\begin{rem}
\label{rem:ZQuat}
It is easy to verify that $Z(\Quat_{n+1})$ (the center of the group)
equals $\{\pm 1\}$ for $n$ even and
$\{\pm 1, \pm (\hat a_1 \hat a_3 \cdots \hat a_n) \}$ for $n$ odd.
For even $n$ we have $\eta \in S_{\PPre}$;
for odd $n$ we have $\eta \in S_{\PA} \smallsetminus S_{\PPre}$.
Remark 3.7 in \cite{Goulart-Saldanha0} confirms that
$\hat\eta = \pm 1$ for even $n$
and $\hat\eta = \pm (\hat a_1 \hat a_3 \cdots \hat a_n)$ for odd $n$.

A computation verifies that $\acute\eta$ commutes
with every element of $Z(\Quat_{n+1})$.
In particular, $q \in Z(\Quat_{n+1})$ if and only if
$\acute\eta q\acute\eta \in Z(\Quat_{n+1})$.
\end{rem}

The \textit{multiplicity} of $\sigma \in S_{n+1}$
is a vector $\mult(\sigma) \in \NN^n$
(where $\NN = \{0,1,2,\ldots\}$)
with coordinates 
\begin{equation}
\label{equation:multiplicity}
\mult_j(\sigma) = (1^\sigma + \cdots + j^\sigma) - (1 + \cdots + j),
\qquad 1 \le j \le n; 
\end{equation}
see \cite{Goulart-Saldanha0} for the reason of the name.

\begin{proof}[Proof of Lemma \ref{lemma:PAhat}]
Clearly, $\sigma \in S_{\PPre}$ if and only if 
$\mult_k(\sigma)$ is even for all $k$.
Similarly, $\sigma \in S_{\PA} \smallsetminus S_{\PPre}$ if and only if
$\mult_k(\sigma) \equiv k \pmod 2$ for all $k$.
We know from Lemma 3.3 in \cite{Goulart-Saldanha0} that
$\hat\sigma =
\pm \hat a_1^{\mult_1(\sigma)} \cdots \hat a_n^{\mult_n(\sigma)}$.
Thus, $\sigma \in S_{\PPre}$ if and only if $\hat\sigma = \pm 1$
and, for odd $n$, $\sigma \in S_{\PA} \smallsetminus S_{\PPre}$
if and only if
$\hat\sigma = \pm (\hat a_1 \hat a_3 \cdots \hat a_n)$ for odd $n$.
The result now follows from Remark \ref{rem:ZQuat}.
\end{proof}


Write $\sigma_0 \vartriangleleft \sigma_1$
if and only if $\sigma_0 < \sigma_1$ (in the strong Bruhat order)
and $\inv(\sigma_1) = \inv(\sigma_0) + 1$.
If $\sigma_0 \vartriangleleft \sigma_1$
then there exist unique $i_0 < i_1$ and $j_0 < j_1$ such that
$\sigma_1 = (i_0i_1) \sigma_0 = \sigma_0 (j_0j_1)$.
We then have $i_k^{\sigma_0} = j_k$, $i_k^{\sigma_1} = j_{1-k}$
(for $k \in \{0,1\}$);
also, if $i \in (i_0,i_1) \cap \ZZ$ then
$i^{\sigma_k} \notin (j_0,j_1) \cap \ZZ$
(where again $k \in \{0,1\}$).
Also, if $\sigma_0 \vartriangleleft \sigma_1$
then there exists a reduced word for $\sigma_1$
such that a reduced word for $\sigma_0$
is obtained by deleting a single letter:
$\sigma_0 = \sigma_a\sigma_b$, $\sigma_1 = \sigma_a a_i \sigma_b$,
$\inv(\sigma_0) = \inv(\sigma_a)+\inv(\sigma_b)$.
Write $\sigma_0 \blacktriangleleft \sigma_1$
if  $\sigma_0 \vartriangleleft \sigma_1$
and $\sigma_0 \equiv \sigma_1 \pmod 2$.
If  $\sigma_0 \vartriangleleft \sigma_1$
and $\sigma_1 = (i_0i_1) \sigma_0 = \sigma_0 (j_0j_1)$
then $\sigma_0 \blacktriangleleft \sigma_1$
if and only if $j_0 \equiv j_1 \pmod 2$.

\begin{rem}
\label{rem:blacktriangle}
It follows directly from Lemma 3.5 in \cite{Goulart-Saldanha0} that
$\sigma_0 \blacktriangleleft \sigma_1$
implies $\hat\sigma_0 = \hat\sigma_1$;
also, 
if $\sigma_0 \vartriangleleft \sigma_1$ and $\hat\sigma_0 = \hat\sigma_1$
then $\sigma_0 \blacktriangleleft \sigma_1$.
\end{rem}

\begin{lemma}
\label{lemma:2mult}
If $\sigma_0, \sigma_1 \in S_{n+1}$ satisfy
\[ \sigma_0 \vartriangleleft \sigma_1, \qquad
2\mult(\sigma_0) \le \mult(\sigma_1) \]
then we have one of the cases:
\[ \sigma_0 = e \vartriangleleft \sigma_1 = a_i, \quad
\sigma_0 = a_i \vartriangleleft \sigma_1 = \sigma_0 a_{i \pm 1},
\quad
\sigma_0 = a_i a_{i+2} \vartriangleleft \sigma_1 = \sigma_0 a_{i+1}. \]
In particular, if $\sigma_0 \blacktriangleleft \sigma_1$ then
$2 \mult(\sigma_0) \not\le \mult(\sigma_1)$.
\end{lemma}

\begin{proof}
If $\sigma_0 \vartriangleleft \sigma_1 =
(i_0 i_1) \sigma_0 = \sigma_0 (j_0 j_1)$ then
\[ \mult_k(\sigma_1) = \mult_k(\sigma_0) + (j_1 - j_0) [i_0 \le k < i_1], \]
as is shown in \cite{Goulart-Saldanha0}, Lemma 2.4.
Thus, $2\mult(\sigma_0) \le \mult(\sigma_1)$ implies
$\mult_k(\sigma_0) = 0$ for $k < i_0$ or $k > i_1$.
Thus $i_0 < i < i_1$ implies $i_0 \le i^{\sigma_0} \le i_1$.
We must then have $j_1 - j_0 = 1$.
Indeed, if $j_1 - j_0 > 1$ take $j$ with $j_0 < j < j_1$
and $i$ with $i^{\sigma_0} = j$:
we have $i_0 < i < i_1$, contradicting $\sigma_0 \vartriangleleft \sigma_1$.
We thus have $\mult_k(\sigma_0) \le 1$ for $i_0 \le k < i_1$.
This implies that $\sigma_0$ is a product of commuting generators
$a_i$, $i_0 \le i < i_1$;
a case by case analysis completes the proof.
\end{proof}

Define the subset $Y \subset S_{n+1}$ by
\begin{equation}
\label{equation:defY}
Y = S_{n+1} \smallsetminus S_{\PA} = 
\{ \sigma \in S_{n+1} \;|\; \hat\sigma \notin Z(\Quat_{n+1}) \}
\end{equation}
(the equivalence follows from Lemma \ref{lemma:PAhat}).
For $\sigma \in Y$,
let $k$ be the smallest integer with
$k^\sigma \equiv (k+1)^\sigma \pmod 2$.
Let $\sigma' = a_k \sigma$.
Notice that $\sigma' \equiv \sigma \pmod 2$
and that $\sigma'' = \sigma$ so that we have
involutions in $Y$
and in each lateral class $\sigma S_{\PPre} \subset S_{n+1}$,
$\sigma \in Y$.
We call $\sigma \in Y$
\textit{low} if $\sigma \blacktriangleleft \sigma'$ 
and \textit{high} otherwise.
Thus, $\sigma \in Y$ is high if and only if
$\sigma' \blacktriangleleft \sigma$.

\begin{lemma}
\label{lemma:vader}
Let $\sigma_0, \sigma_1 \in Y$
with $\sigma_0$ high, $\sigma_1 \blacktriangleleft \sigma_0$
and $\sigma_1 \ne \sigma'_0$:
then $\sigma_1$ is high.
\end{lemma}

\begin{proof}
We have $\sigma_0 \equiv \sigma_1 \pmod 2$.
Let $k$ be the smallest integer with
$k^{\sigma_0} \equiv (k+1)^{\sigma_0} \pmod 2$:
we have $\sigma'_0 = a_k \sigma_0 \blacktriangleleft \sigma_0$
and $\sigma'_1 = a_k \sigma_1$.
We need to prove that $k^{\sigma_1} > (k+1)^{\sigma_1}$.
Write $\sigma_1 = (i_0i_1) \sigma_0$, $i_0 < i_1$.
If $i_0 = k$ and $i_1 = k+1$ then $\sigma_1 = \sigma'_0$,
contradicting the statement.
If $\{i_0,i_1\}$ is disjoint from $\{k,k+1\}$ then
$k^{\sigma_1} = k^{\sigma_0} > (k+1)^{\sigma_0} = (k+1)^{\sigma_1}$
and therefore $\sigma_1$ is high.
If $i_0 = k$ and $i_1 > k+1$ we must have
$(k+1)^{\sigma_0} < i_1^{\sigma_0} < k^{\sigma_0}$
and therefore again $\sigma_1$ is high.
If $i_0 = k+1$ we must have
$i_1^{\sigma_0} < (k+1)^{\sigma_0} < k^{\sigma_0}$
and therefore also in this case $\sigma_1$ is high.
The two remaining cases are: $i_1 = k$ and $i_1 = k+1$ with $i_0 < k$;
both are similar to the cases above and imply $\sigma_1$ high.
\end{proof}

\begin{rem}
\label{rem:blackpartition}
A \textit{$\blacktriangleleft$-matching} of a set $X \subset S_{n+1}$
is a partition of $X$ into subsets of cardinality $2$ of the form
$\{\sigma_{-}, \sigma_{+}\}$ with
$\sigma_{-} \blacktriangleleft \sigma_{+}$.
For instance, for $n = 2$ and $X = S_3 \smallsetminus \{e,\eta = aba\}$
the following is a $\blacktriangleleft$-matching:
\[ \{ \{ a, ba\}, \{b, ab\} \}. \]
Indeed, we have $a \blacktriangleleft ba$ and 
$b \blacktriangleleft ab$.

Partition $S_{n+1}$ into lateral classes $\sigma S_{\PPre} \subset S_{n+1}$.
The involution $\sigma \leftrightarrow \sigma'$ above defines
a {$\blacktriangleleft$-matching} in each class
\emph{except} $S_{\PPre}$ and,
for $n$ odd, $S_{\PA} \smallsetminus S_{\PPre} = \eta S_{\PPre}$.
See the case $n = 2$ above.

For $n = 3$, $S_4$ is partitioned into $6$ classes of cardinality $4$ each.
For instance, $b S_{\PPre} = \{b,ab,cb,acb\}$ is partitioned as
$\{\{b,ab\},\{cb,acb\}\}$.
The classes $S_{\PPre} = \{e, aba, bcb, bacb\}$ and
$\eta S_{\PPre} = \{ ac, abc, cba, \eta \}$
admit no $\blacktriangleleft$-matching.

For $n = 4$, we have $S_{\PA} = S_{\PPre}$ and $|S_{\PPre}| = 12$;
the list of elements is given both in Example \ref{example:smallPA}
and in Figure \ref{fig:blackaba}.
The group $S_5$ is then partitioned into $10$ classes of cardinality $12$.
The class $S_{\PPre}$
admits no $\blacktriangleleft$-matching;
see Figure \ref{fig:blackaba}.

\begin{figure}[ht]
\def\svgwidth{14cm}
\centerline{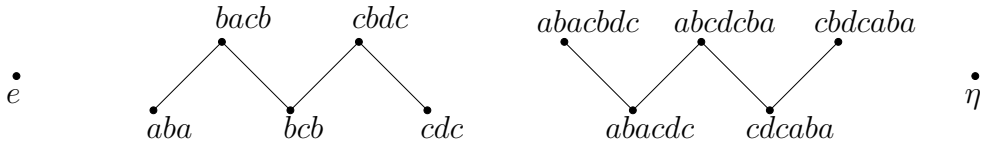}
\caption{The $\blacktriangleleft$-pairs
in $S_{\PPre} \subset S_5$.}
\label{fig:blackaba}
\end{figure}

For any value of $n$,
there is no permutation $\sigma \in S_{n+1}$ for which
$e \blacktriangleleft \sigma$ or
$\sigma \blacktriangleleft e$.
Similarly, there is no permutation $\sigma \in S_{n+1}$ for which
$\eta \blacktriangleleft \sigma$ or
$\sigma \blacktriangleleft \eta$.
Thus, there is no {$\blacktriangleleft$-matching}
of either $S_{\PPre}$ or of $\eta S_{\PPre}$.

Removing $e$ and $\eta$ does not help much.
Figure \ref{fig:blackaba} shows the 
$\blacktriangleleft$-connected components of $aba$ and 
of $abcdcba \in S_{\PPre}$ for $n = 4$.
The figure for the connected component of $aba$
is similar for other values of $n$.
The number of elements is odd and
trying to construct a {$\blacktriangleleft$-matching}
always leaves out an element.
\end{rem}

\begin{rem}
\label{rem:Yk}
Let $Y \subset S_{n+1}$ be as in Equation \eqref{equation:defY}.
For $k \ge 2$,
let $Y_k \subseteq Y$ be the set of permutations $\sigma \in Y$
which are either
low with $\inv(\sigma) < k$ or
high with $\inv(\sigma) \le k$.
For instance, for $n \ge 2$: 
\[ Y_2 = \{ a_1, a'_1 = a_2a_1, a_2, a'_2 = a_1a_2, \ldots,
a_{n-1}, a'_{n-1} = a_{n-2}a_{n-1}, a_n, a'_n = a_{n-1}a_n \}. \]
The correspondence $\sigma \leftrightarrow \sigma'$ 
defines an involution in $Y_k$ and
a $\blacktriangleleft$-matching of $Y_k$.
It follows from Lemma \ref{lemma:vader} that
if $\sigma_0 \in Y_k$ and $\sigma_1 \blacktriangleleft \sigma_0$
then $\sigma_1 \in Y_k$;
if we also have $\sigma_1 \ne \sigma'_0$ then $\sigma_1 \in Y_{k-1}$.
\end{rem}

\section{Bruhat cells}
\label{sect:bruhat}

In this section we revise some notation and results
and prove some results concerning Bruhat cells
and related constructions.

Two invertible matrices $A_0, A_1$ are (unsigned) \textit{Bruhat equivalent}
if there exist $U_0, U_1 \in \Up_{n+1}$ with $A_0 = U_0 A_1 U_1$
(recall that $\Up_{n+1}$ is the group of
invertible real upper triangular matrices).
Any invertible matrix $A$ is Bruhat equivalent
to a unique permutation matrix $P_{\rho}$, $\rho \in S_{n+1}$.
We then say that $A$ is Bruhat equivalent to $\rho$.
Given $A \in \GL_{n+1}$ and a pair $(i,j)$, let
\begin{equation}
\label{equation:SW}
\SW(A,i,j) = \SubMatrix(A,i \ldots n+1, 1 \ldots j),
\end{equation}
a $(n+2-i) \times j$ submatrix of $A$.
As is well known,
two matrices $A_0$ and $A_1$ are Bruhat equivalent if
for every pair $(i,j)$ 
the ranks of $\SW(A_0,i,j)$ and of $\SW(A_1,i,j)$ are equal.
Let $(A_k)_{k \in \NN}$ be a sequence of matrices in $\GL_{n+1}$
converging to $A_{\infty} \in \GL_{n+1}$;
if, for all $k \in \NN$, $A_k$ is Bruhat equivalent to $\rho_0$
and $A_\infty$ is Bruhat equivalent to $\rho_1$
then $\rho_1 \le \rho_0$ in the strong Bruhat order.
Indeed, this condition is often used to define the Bruhat order.

Given a signed permutation matrix $Q_0 \in \B_{n+1}^{+}$
corresponding to the permutation $\rho \in S_{n+1}$,
let $\sigma = \eta\rho$ and
consider the affine space of matrices
\[ \bM_{Q_0} = Q_0 \Lo_{n+1}^1 \cap \Lo_{n+1}^1 Q_0 
= \Up_{\sigma} Q_0
= Q_0 \Lo_{\sigma^{-1}} \subset \GL_{n+1}. \]
The subgroups $\Up_{\sigma} \subseteq \Up_{n+1}^1$
and $\Lo_{\sigma} \subseteq \Lo_{n+1}^1$ (for $\sigma \in S_{n+1}$)
are implicitly defined by this equation
and are discussed in Section 2 of \cite{Goulart-Saldanha0}.
We have $\dim(\bM_{Q_0}) = \dim(\Lo_{\sigma^{-1}}) = \inv(\sigma)
= \inv(\eta) - \inv(\rho)$.
In order to construct a parametrization of $\bM_{Q_0}$,
start with $Q_0$ and run through the zero entries.
An entry $(i,j)$ satisfies both $j < i^\rho$ and $j^{\rho^{-1}} < i$
if and only if it is both below a non zero entry of $Q_0$
and to the left of another non zero entry of $Q_0$.
In this case, replace it by a free real variable $x_k$.
This space of matrices is constructed
in the proof of Theorem 2 in \cite{Goulart-Saldanha0}
(where its generic element is called $M$)
and discussed in Section 7 in \cite{Goulart-Saldanha1}
(where its generic element is called $\tilde M$).
We have $Q_0 \Lo_{n+1}^1 = \Up_{\rho} \bM_{Q_0}$:
every element $A \in Q_0 \Lo_{n+1}^1$ can be uniquely written
as a product $A = UM$ with $U \in \Up_{\rho}$ and $M \in \bM_{Q_0}$,
and $U \in \Up_{\rho}$ and $M \in \bM_{Q_0}$ imply $UM \in Q_0 \Lo_{n+1}^1$.
Clearly,  $M$ is Bruhat equivalent to $A = UM$.

\begin{example}
\label{example:Q1M1}
Consider $n = 5$ and the matrix
\[
Q_1 = \begin{pmatrix}
 \cdot & \cdot & -1 & \cdot & \cdot & \cdot \\
 \cdot & 1 & \cdot & \cdot & \cdot & \cdot \\
 \cdot & \cdot & \cdot & \cdot & \cdot & 1 \\
 1 & \cdot & \cdot & \cdot & \cdot & \cdot \\
 \cdot & \cdot & \cdot & \cdot & 1 & \cdot \\
 \cdot & \cdot & \cdot & 1 & \cdot & \cdot 
\end{pmatrix} \in \B_{n+1}^{+} \]
so that the corresponding permutation is $\rho_1 = [326154]$
and $\sigma_1 = \eta\rho_1 = [451623]$.
The affine space $\bM_{Q_1}$ is the set of matrices
$\tilde M_1$ of the form
\[
\tilde M_1 = \begin{pmatrix}
 \cdot & \cdot & -1 & \cdot & \cdot & \cdot \\
 \cdot & 1 & \cdot & \cdot & \cdot & \cdot \\
 \cdot & x_1 & x_2 & \cdot & \cdot & 1 \\
1 & \cdot & \cdot & \cdot & \cdot & \cdot \\
x_3 & x_4 & x_5 & \cdot & 1 & \cdot \\
x_6 & x_7 & x_8 & 1 & \cdot & \cdot 
\end{pmatrix}, \qquad
x_1, \ldots, x_8 \in \RR. 
\]
Notice that $\inv(\rho_1) = 7$
and $\dim(\bM_{Q_1}) = \inv(\sigma_1) = 8$.
\end{example}

\begin{lemma}
\label{lemma:Mrho}
Consider $Q_1 \in \B_{n+1}^{+}$ corresponding to $\rho_1 \in S_{n+1}$
and the affine space $\bM_{Q_1}$.
\begin{enumerate}
\item{If $M \in \bM_{Q_1}$ is equivalent to $\rho \in S_{n+1}$
then $\rho_1 \le \rho$.
The only such matrix $M$ equivalent to $\rho_1$ is $M = Q_1$.}
\item{Consider $\rho_0 \in S_{n+1}$, $\rho_1 \vartriangleleft \rho_0$,
$\rho_0 = (i_0 i_1) \rho_1 = \rho_1 (j_0 j_1)$, $i_0 < i_1$, $j_0 < j_1$.
Take $X = e_{i_1} e_{j_0}^\top$ be the matrix with a single non zero
entry equal to $1$ in position $(i_1,j_0)$.
Then $Q_1 + tX \in \bM_{Q_1}$ for all $t \in \RR$.
Moreover, $M \in \bM_{Q_1}$ is equivalent to $\rho_0$
if and only if $M$ is of the form $M = Q_1 + tX$
for $t \in \RR \smallsetminus \{0\}$.}
\end{enumerate}
\end{lemma}

\begin{example}
\label{example:Q1rho0}
Consider $Q_1$ and $\rho_1$ and in Example \ref{example:Q1M1}.
Take $\rho_0 = [356124]$
so that $\rho_1 \vartriangleleft \rho_0$,
$i_0 = 2$, $i_1 = 5$, $j_0 = 2$ and $j_1 = 5$.
The matrices in $\bM_{Q_1}$ which are equivalent to $\rho_0$
are precisely
\[
\tilde M_1 = Q_1 + x_4 X = \begin{pmatrix}
 \cdot & \cdot & -1 & \cdot & \cdot & \cdot \\
 \cdot & 1 & \cdot & \cdot & \cdot & \cdot \\
 \cdot & \cdot & \cdot & \cdot & \cdot & 1 \\
1 & \cdot & \cdot & \cdot & \cdot & \cdot \\
 \cdot & x_4 & \cdot & \cdot & 1 & \cdot \\
 \cdot & \cdot & \cdot & 1 & \cdot & \cdot 
\end{pmatrix}, \qquad
x_4 \in \RR \smallsetminus \{0\}.  \]
Notice that for $x_4 = 0$ we have $\tilde M_1 = Q_1$,
which is equivalent to $\rho_1$.
\end{example}

\begin{coro}
\label{coro:Mrho}
Let $\rho_0$, $\rho_1$ and $Q_1$ be as in Lemma \ref{lemma:Mrho}.
The matrices $A \in Q_1\Lo_{n+1}^{1}$ which are Bruhat equivalent
to $\rho_1$ are precisely $A \in \Up_{\rho_1} Q_1 = Q_1 \Lo_{\rho_1^{-1}}$.
The matrices $A \in Q_1\Lo_{n+1}^{1}$ which are Bruhat equivalent
to $\rho_0$ are precisely those of the form
$U (Q_1 + tX)$ for $U \in \Up_{\rho_1}$
and $t \in \RR \smallsetminus \{0\}$
(and $X$ as in Lemma \ref{lemma:Mrho}).
\end{coro}

\begin{proof}
This follows from Lemma \ref{lemma:Mrho}
and $Q_1 \Lo_{n+1}^{1} = \Up_{\rho_1} \bM_{Q_1}$.
\end{proof}

\begin{proof}[Proof of Lemma \ref{lemma:Mrho}]
Clearly $Q_1$ is equivalent to $\rho_1$ and
$Q_1 + tX$ is equivalent to $\rho_0$ if $t \ne 0$.
Set 
\[ D_0(\lambda) = \diag(
\lambda^{1^{\rho_1}},\ldots,\lambda^{i^{\rho_1}},\ldots), \qquad
D_1(\lambda) = \diag(\lambda^{-1}, \ldots, \lambda^{-i},\ldots). \]
Given $M_1 \in \bM_{Q_1}$, set $M_\lambda = D_0(\lambda) M_1 D_1(\lambda)$.
For $\lambda \in (0,1)$ we have that $M_\lambda \in \bM_{Q_1}$
is Bruhat equivalent to $M_1$.
Also, 
\[ \lim_{\lambda \searrow 0} M_\lambda = Q_1; \]
from the characterization of the strong Bruhat order
above we have  $\rho_1 \le \rho$.


Assume that $M \in \bM_{Q_1}$ is Bruhat equivalent to $Q_1$.
If $M \ne Q_1$, take $(i,j)$ such that $M_{i,j} \ne (Q_1)_{i,j}$
with minimal $j-i$, that is,
the position furthest to the southwest direction.
The submatrices
$S_Q = \SW(Q_1,i,j)$ and $S_M = \SW(M,i,j)$
(with the notation of Equation \eqref{equation:SW})
have the same rank
and are equal except for the northeast entry.
We must have $j < i^{\rho_1}$ and therefore
the first row of $S_Q$ is the zero vector.
The first row of $S_M$ 
is therefore of the form $(0, \ldots, 0, c)$, $c \ne 0$.
By equality of rank,
it has to be a linear combination of other rows of $S_Q$.
But we must have $j^{\rho_1^{-1}} < i$,
so that the last column of $S_Q$ is the zero vector,
a contradiction.

Assume that $M \in \bM_{Q_1}$ is Bruhat equivalent to $\rho_0$.
If $M$ is not of the form $Q_1 + tX$,
take $(i,j)$ to be a position which violates that form.
We cover the entries in a convenient order:
assume $(i,j)$ to be the first violation in this order.
The cases $i > i_1$ and $j < j_0$ are just as in the
$\rho_1$ case, discussed in the previous paragraph.
The case $(i,j) = (i_1,j_0)$ is easily ruled out.
The case $i = i_1$, $j_0 < j < j_1$ is divided in two subcases:
if $j^{\rho^{-1}} > i_1$ then the position is zero
and therefore $M_{i,j} = (Q_1)_{i,j} = 0$, so there is no violation.
Otherwise, notice that the ranks of
$\SW(Q_1,i_0,j-1)$ and of $\SW(Q_1,i_0,j)$ are equal
and that therefore the $j$-th column of $\SW(M,i_0,j)$
is a linear combination of previous columns.
Notice also that for $j' \le j$, $j' \ne j_0$,
we have $M_{i_0,j'} = 0$:
for $j' < j_0$ this follows from the previous case;
for $j' > j_0$ this follows from $M \in \bM_{Q_1}$.
In particular, $M_{i_0,j} = 0$.
Furthermore, for $j' < j$, $j' \ne j_0$,
we have $M_{i_1,j'} = 0$, by the minimality hypothesis.
Thus, the linear combination of columns $1$ to $j-1$
of $\SW(M,i_0,j)$ which produces column $j$
has zero coefficient for $j_0$,
implying $M_{i_1,j} = 0$, as desired.
The case $j = j_0$, $i_0 < i < i_1$ is similar.
The cases $i = i_1$, $j \ge j_1$ and $j = j_0$, $i \le i_0$ are trivial.
The cases $i_0 < i < i_1$, $j_0 < j < j_1$
are handled similarly to the $\rho_1$ case, but with $\rho_0$ instead.
The cases $i = i_0$, $j_0 < j \le j_1$ and
$j = j_1$, $i_0 \le i < i_1$ are trivial.
The case $j > j_1$, $i_0 < i < i_1$ is also similar,
but for clarity we describe it.
The rank of $\SW(M,i,j)$ equals that of $\SW(M,i,j-1)$
and therefore the $j$-th column of $\SW(M,i,j)$
is a linear combination of the first $j-1$ columns.
But we have $M_{i,j'} = 0$ for all $j' < j$,
implying $M_{i,j} = 0$.
The case $i < i_0$, $j_0 < j < j_1$ is similar.
The final case is $i < i_0$, $j > j_1$:
this is again similar to the $\rho_1$ case,
completing the proof.
\end{proof}

Below we will make heavy use of notation
and results from \cite{Goulart-Saldanha0};
we briefly recall some basic facts.
Given $\sigma \in S_{n+1}$,
let $\Bru_\sigma \subset \SO_{n+1}$
be the set of matrices $Q \in \SO_{n+1}$
which are Bruhat equivalent to $\sigma$.
For $\Pi: \Spin_{n+1} \to \SO_{n+1}$,
the inverse image $\Pi^{-1}[\Bru_\sigma]$
is also called $\Bru_\sigma \subset \Spin_{n+1}$.
We have $\Bru_\sigma \subseteq \overline{\Bru_\rho}$
if and only if $\sigma \le \rho$ (in the strong Bruhat order).
The connected components of $\Bru_\sigma \subset \Spin_{n+1}$
are all contractible submanifolds of dimension $\inv(\sigma)$.

Recall from Equation \eqref{equation:acutegrave} that
if $\sigma \in S_{n+1}$ then $\acute\sigma \in \tilde \B_{n+1}^{+}$;
also, $\Quat_{n+1} \subset  \widetilde\B_{n+1}^{+}$
is a normal subgroup of order $2^{n+1}$,
defined in Equation \eqref{equation:Quat}.
Each element $q \acute\sigma$, $q \in \Quat_{n+1}$,
belongs to a different connected
component of $\Bru_\sigma$, called $\Bru_{q \acute\sigma}$ so that
\[ \Bru_\sigma = \bigsqcup_{q \in \Quat_{n+1}} \Bru_{q \acute\sigma},
\qquad q\acute\sigma \in \Bru_{q \acute\sigma},
\]
as in Equation \eqref{equation:Bruz}.
The following lemma discusses some special cases of
the relation $\Bru_{z_1} \subseteq \overline{\Bru_{z_0}}$,
$z_0, z_1 \in \widetilde\B_{n+1}^{+}$.

\begin{lemma}
\label{lemma:trianglemanifold}
Consider $\rho_1 \vartriangleleft \rho_0$, 
$\rho_1 = \rho_a\rho_b$, $\rho_0 = \rho_a a_i \rho_b$,
$\inv(\rho_1) = \inv(\rho_a)+\inv(\rho_b)$.
Consider $z_1 = q \acute\rho_1 = q \acute\rho_a \acute\rho_b \in 
\widetilde\B_{n+1}^{+}$, $q \in \Quat_{n+1}$. 
Set $z_0^{+} = q \acute\rho_a \acute a_i \acute\rho_b$,
$z_0^{-} = q \acute\rho_a \grave a_i \acute\rho_b$.
Then
there exists a tubular neighborhood $\;\cU_{z_1}$ of $\Bru_{z_1}$
such that 
\[ \cU_{z_1} \cap \Bru_{\rho_0} =
\cU_{z_1} \cap (\Bru_{z_0^{+}} \cup \Bru_{z_0^{-}}). \]
Moreover,
$N = \Bru_{z_0^{+}} \cup \Bru_{z_1} \cup \Bru_{z_0^{-}}
\subset \Spin_{n+1}$ is a smooth submanifold of dimension
$\inv(\rho_0)$.
\end{lemma}

\begin{rem}
\label{remark:Uz1}
The open subsets $\;\cU_z \subset \Spin_{n+1}$,
$z \in \widetilde\B_{n+1}^{+}$, 
have been discussed in Section 4 of \cite{Goulart-Saldanha0};
Theorem 2 of \cite{Goulart-Saldanha0}
gives a system of coordinates for $\;\cU_z$. 
The set $\Bru_{\acute\eta} \subset \Spin_{n+1}$ is an open neighborhood
of $\acute\eta \in \Spin_{n+1}$.
For $z \in \widetilde\B_{n+1}^{+}$,
the set $\cU_{z} = z\grave\eta\Bru_{\acute\eta} \subset \Spin_{n+1}$
is an open neighborhood of $z$ and an open tubular neighborhood
of $\Bru_z \subset \Spin_{n+1}$.
\end{rem}

\begin{proof}[Proof of Lemma \ref{lemma:trianglemanifold}]
Set $Q_1 = \Pi(z_1)$ (where $\Pi: \Spin_{n+1} \to \SO_{n+1}$)
and consider the affine space $Q_1\Lo_{n+1}^1 \subset \GL_{n+1}$.
Define $\bQ: Q_1\Lo_{n+1}^1 \to \cU_{z_1}$ by
$\bQ(A) = z$ if and only if there exists $R \in \Up_{n+1}^{+}$
with $A = \Pi(z) R$.
The map $\bQ$ is a diffeomorphism.
Furthermore, if $A$ is Bruhat equivalent to $\rho$
then $\bQ(A) \in \Bru_{\rho}$.
It therefore follows from Corollary \ref{coro:Mrho} that 
\[ \cU_{z_1} \cap \Bru_{\rho_1} = \bQ[\Up_{\rho_1}Q_1], \qquad
\cU_{z_1} \cap \Bru_{\rho_0} =
\bQ[\Up_{\rho_1}(Q_1 + (\RR \smallsetminus \{0\})X)]. \]
In particular,
$\cU_{z_1} \cap \Bru_{\rho_1}$ is connected and
$\cU_{z_1} \cap \Bru_{\rho_0}$ has at most two connected components.
We present another system of coordinates which identifies
the two connected components of 
$\,\cU_{z_1} \cap \Bru_{\rho_0}$ as being contained in $\Bru_{z_0^\pm}$.

Let $D =  \Bru_{q\acute\rho_a} \times (-\pi,\pi) \times \Bru_{\acute\rho_b}$.
Consider the smooth map $\Phi: D \to \Spin_{n+1}$,
$\Phi(z_a,\theta,z_b) = z_a \alpha_i(\theta) z_b$.
Consider the following subsets of the domain $D$:
$D_{-} = {\Bru_{q\acute\rho_a} \times (-\pi,0)
\times \Bru_{\acute\rho_b}}$,
$D_{0} = {\Bru_{q\acute\rho_a} \times \{0\}   
\times \Bru_{\acute\rho_b}}$,
$D_{+} = {\Bru_{q\acute\rho_a} \times (0,\pi)
\times \Bru_{\acute\rho_b}}$.
It follows from Theorem 1 in \cite{Goulart-Saldanha0}
that the following restrictions are diffeomorphisms:
$\Phi|_{D_{-}}: D_{-} \to \Bru_{z_0^{-}}$,
$\Phi|_{D_{0}}: D_{0} \to \Bru_{z_1}$,
$\Phi|_{D_{+}}: D_{+} \to \Bru_{z_0^{+}}$,
completing the proof.
\end{proof}

\section{Partial orders}
\label{sect:partial}

Recall that $\Word_n$ is the set of finite words
in the alphabet $S_{n+1} \smallsetminus \{e\}$.
The set $\Word_n$ plays a very important part in our discussion.
In this section we discuss a few partial orders in $\Word_n$.

A partial order $\preceq$ on $\Word_n$
is defined and discussed in \cite{Goulart-Saldanha1};
we recall the definition.
For $\sigma \in S_{n+1} \smallsetminus \{e\}$,
take $\sigma = \eta\rho$ and consider
$z_0 = q_a \acute\rho = q_c \grave\rho
\in \acute\rho\Quat_{n+1} \subset \widetilde\B_{n+1}^{+}$,
so that $q_a, q_c \in \Quat_{n+1}$.
Define the open neighborhood $\cU_{z_0} \subset \Spin_{n+1}$
as in Remark \ref{remark:Uz1}.
Define $z_{\pm 1} \in \acute\eta \Quat_{n+1}$
by $z_{+1} = q_a \acute\eta$, $z_{-1} = q_c \grave\eta$.
Consider a short convex curve
$\Gamma_1: [-1,+1] \to \cU_{z_0} \subset \Spin_{n+1}$ with
\[ \Gamma_1(0) = z_0, \quad
t \in [-1,0) \to \Gamma_1(t) \in \Bru_{z_{-1}}, \quad
t \in (0,+1] \to \Gamma_1(t) \in \Bru_{z_{+1}}. \]
Notice that the itinerary of $\Gamma_1$ is $\sigma$.
We have $w \preceq \sigma$ if and only if there exists
a convex curve $\Gamma_0: [-1,+1] \to \cU_{z_0} \subset \Spin_{n+1}$
with $\Gamma_0(-1) = \Gamma_1(-1)$, $\Gamma_0(+1) = \Gamma_1(+1)$
and itinerary $w$.
It is not hard to verify that the above definition
does not depend on the choice of $\Gamma_1$.
For $w_1 = \sigma_1\cdots\sigma_\ell$,
write $w_0 \preceq w_1$ if there exist words
$w_{0,1}, \ldots, w_{0,\ell}$ with
$w_0 = w_{0,1}\cdots w_{0,\ell}$ and
$w_{0,i} \preceq \sigma_i$ (for all $i$).
We prove in \cite{Goulart-Saldanha1} that
\begin{equation}
\label{equation:preceq}
\cL_n[w_1] \cap \overline{\cL_n[w_0]} \ne \varnothing 
\quad\implies\quad
w_0 \preceq w_1. 
\end{equation}
This implication also holds for topologies
which we are not considering in the present paper, such as $H^1$.
Some basic properties of the partial order $\preceq$
are proved in Theorem 3 of \cite{Goulart-Saldanha1}. 

Unfortunately,
the partial order $\preceq$ is complicated:
there are basic open questions concerning it,
including a conjecture by Shapiro and Shapiro
\cite{Shapiro-Shapiro, Shapiro-Shapiro3}.
This conjecture is essentially equivalent
to Conjecture 2.2 of \cite{Goulart-Saldanha1}.
In order to state it,
we need to introduce the concept of multiplicity of a word.
Recall that the \textit{multiplicity} of $\sigma \in S_{n+1}$
is a vector $\mult(\sigma) \in \NN^n$,
defined in Equation \eqref{equation:multiplicity};
for $w = \sigma_1\cdots\sigma_\ell = (\sigma_1, \ldots, \sigma_\ell)$,
define
$\mult(w) = \mult(\sigma_1) + \cdots + \mult(\sigma_\ell)$.
Conjecture 2.2 of \cite{Goulart-Saldanha1} asks:
does $w \preceq \sigma$ imply $\mult(w) \le \mult(\sigma)$?
Theorem 4 in \cite{Goulart-Saldanha1} proves a related result.
See also \cite{Saldanha-Shapiro-Shapiro}
and \cite{Saldanha-Shapiro-Shapiro1} for other new related results.
The original conjecture, however, remains open as of this writing.

In the present paper, we prefer therefore to work
with a more manageable partial order $\sqsubseteq$ (also on $\Word_n$).
For $w \in \Word_n$ and $\sigma \in S_{n+1}$,
we define
\[ (w \sqsubseteq \sigma) \quad\iff\quad
((w \ne ()) \land (\mult(w) \le \mult(\sigma)) \land (\hat w = \hat\sigma)); \]
here $()$ denotes the empty word.
For words $w_0$ and $w_1 = (\sigma_1, \ldots, \sigma_\ell) \in \Word_n$,
we have $w_0 \sqsubseteq w_1$ if and only if there exists non empty words
$w_{01}, \ldots, w_{0\ell}$ such that
$w_0$ equals the concatenation $w_{01}\cdots w_{0\ell}$
and $w_{0j} \sqsubseteq \sigma_j$ for every $j$, $1 \le j \le \ell$.
It follows from Theorems 3 and 4 from \cite{Goulart-Saldanha1}
that, if we work in $H^r$ and $r > ((n+1)/2)^2$ 
we have an implication similar to the one in 
Equation \eqref{equation:preceq}:
\begin{equation}
\label{equation:sqsubseteq}
\cL_n[w_1] \cap \overline{\cL_n[w_0]} \ne \varnothing 
\quad\implies\quad
w_0 \sqsubseteq w_1. 
\end{equation}
Notice that given $w_0 \in \Word_n$
both sets 
\[ \{ w \in \Word_n, w \sqsubseteq w_0\}, \qquad
\{ w \in \Word_n, w_0 \sqsubseteq w\} \]
are finite.
Indeed, $w_0 \sqsubseteq w_1$ implies $\mult(w_0) \le \mult(w_1)$
and $\ell(w_0) \ge \ell(w_1)$
(where $\ell(w)$ is the length of the word $w$).
A similar result for $\preceq$ follows from
\cite{Saldanha-Shapiro-Shapiro1}.

It is natural to ask at this point how the partial orders
$\preceq$ and $\sqsubseteq$ are related.
We first consider the implication
$(w_0 \preceq w_1) \to (w_0 \sqsubseteq w_1)$.
We know that $w \preceq \sigma$ implies
$w \ne ()$ and $\hat w = \hat\sigma$.
We do not know, however, whether
$w \preceq \sigma$ implies $\mult(w) \le \mult(\sigma)$:
this is essentially Conjecture 1.2 in \cite{Goulart-Saldanha1}.
It is relatedly also an open problem whether the implication
$(w_0 \preceq w_1) \to (w_0 \sqsubseteq w_1)$ holds.
The next example shows that the converse implication is false.

\begin{example}
\label{example:partialorders}
Let $n = 4$, $\sigma_0 = aba$ and $\sigma_1 = \eta = abacbadcba$.
It follows from Example 3.7 in \cite{Goulart-Saldanha0}
that $\hat\sigma_0 = \hat\sigma_1 = -1$
and a simple computation then verifies that
$(\sigma_0) \sqsubseteq (\sigma_1)$.
We now prove, however, that $\sigma_0 \not\preceq \sigma_1$
and therefore, from Equation \eqref{equation:preceq},
\[ \cL_4[(\sigma_1)] \cap \overline{\cL_4[(\sigma_0)]} = \varnothing. \]
As in \cite{Goulart-Saldanha1}, we work in the nilpotent group $\Lo_5^1$
and its subsets $\Pos_\eta, \Neg_\eta \subset \Lo_5^1$.
The only point in $\overline{\Pos_\eta} \cap \overline{\Neg_\eta}$
is the identity.
Consider a locally convex curve $\Gamma_L: [-1,1] \to \Lo_{5}^{1}$
with $\Gamma_L(-1) \in \Neg_\eta$ and $\Gamma_L(1) \in \Pos_\eta$.
Set $\Gamma(t) = q \bQ(\Gamma_L(t))$, $q \in \Quat_{n+1}$.
If $\Gamma_L(t) = I$ for some $t$ then $\iti(\Gamma) = \eta$.
Otherwise, $\Gamma_L$ must cross $\partial\Neg_\eta$ and $\partial\Pos_\eta$
at two distinct times $t_{-} < t_{+}$, $t_{-}, t_{+} \in \sing(\Gamma)$:
thus $\iti(\Gamma) \ne \sigma_0$.

In the following examples we also have
$(\sigma_0) \sqsubseteq (\sigma_1)$ and 
$(\sigma_0) \not\preceq (\sigma_1)$:
\begin{enumerate}
\item{$n = 6$, $\sigma_0 = a_1a_2a_1a_4a_5a_4$, $\sigma_1 = \eta$;}
\item{$n = 7$, $\sigma_0 = a_1a_3a_5a_7$, $\sigma_1 = \eta$;}
\item{$n = 8$, $\sigma_0 = a_1a_2a_1$,
$\sigma_1 = a_1a_2a_1a_4a_5a_4a_7a_8a_7$.}
\end{enumerate}
In all cases the verification of the $\sqsubseteq$ condition
is an easy computation.
For the first two items, the proof of the $\preceq$ claim
is similar to the one given above;
the third item requires a modification of the argument
(which we neither present nor need in this paper).
\end{example}

\begin{rem}
\label{remark:preceqdim}
Neither partial order $\preceq$ nor $\sqsubseteq$ respects dimension of words
(defined in Equation~\eqref{equation:dim}).
For instance, we have
$[ac]b[ac] \preceq [acb]$,
$[ac]b[ac] \sqsubseteq [acb]$
and $\dim([ac]b[ac]) = \dim([acb]) = 2$:
this is the example discussed in detail in Section~9 
in \cite{Goulart-Saldanha1}.
\end{rem}

\section{Lower and upper sets}
\label{sect:lowersets}

A subset $\Ideal \subseteq \Word_n$ is a \emph{lower set}
(for $\sqsubseteq$) if,
for all $w_1 \in \Ideal$ and $w_0 \in \Word_n$,
$w_0 \sqsubseteq w_1$ implies $w_0 \in \Ideal$.
In particular, $\varnothing$ and $\Word_n$ are lower sets.
If $\Ideal$ is a lower set then
it follows from Equation \eqref{equation:sqsubseteq} that
\[ \cL_n[\Ideal] = \bigsqcup_{w \in \Ideal} \cL_n[w] \subseteq \cL_n \]
is an open subset.
Given $w_0 \in \Word_n$,
let $\Ideal(w_0) = \{w \in \Word_n, w \sqsubseteq w_0\}$ and
$\Ideal^\ast(w_0) = \Ideal(w_0) \smallsetminus \{w_0\}$:
both are lower sets.
The map $\phi: \DD^k \to \cL_n$ transversal section to $\cL_n[\sigma]$
constructed in Lemma 7.1 of \cite{Goulart-Saldanha1} 
is of the form  $\phi: \DD^k \to \cL_n[\Ideal(\sigma)] \subset \cL_n$.
Figure \ref{fig:subaba-cw} is a diagram of 
\[ \Ideal([aba]) =
\{ aa, abab, bb, baba, [ba]a, a[ba], [ab]b, b[ab], [aba] \}; \]
compare with Figure \ref{fig:aba-cw}.

\begin{figure}[ht]
\def\svgwidth{10cm}
\centerline{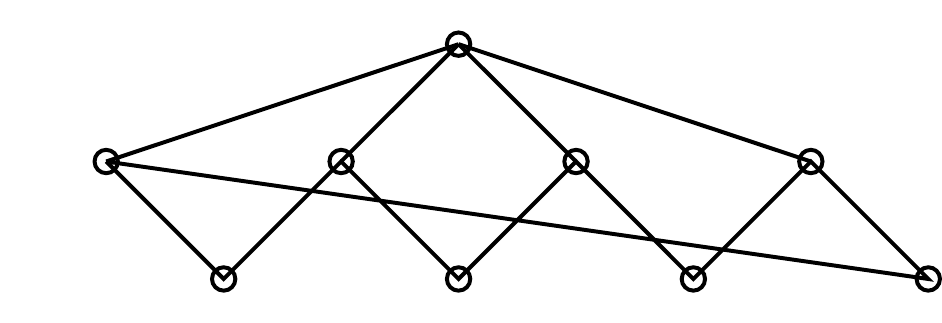}
\caption{The lower set $\Ideal([aba])$.}
\label{fig:subaba-cw}
\end{figure}

As another example of a lower set,
let $\Ideal_{(\omega)}$ be the set of words of dimension $0$,
i.e., words whose letters are generators $a_k$.
(The subscript is an ordinal,
a relic of notation used in \cite{Goulart-Saldanha}, starting in Section 14.
The notation shall not be defined or needed in its general form.)
The set $\cL_n[\Ideal_{(\omega)}]$
is a disjoint union of contractible open sets $\cL_n[w]$,
$w \in \Ideal_{(\omega)}$.

Similarly, $\Filter \subseteq \Word_n$ is an \emph{upper set}
(for $\sqsubseteq$) if,
for all $w_0 \in \Filter$ and $w_1 \in \Word_n$,
$w_0 \sqsubseteq w_1$ implies $w_1 \in \Filter$.
If $\Filter$ is an upper set then
$$ \cL_n[\Filter] = \bigsqcup_{w \in \Filter} \cL_n[w] \subseteq \cL_n $$
is a closed subset
(again from Equation \eqref{equation:sqsubseteq}).

A few examples are in order.

\begin{example}
\label{example:IY}
The set $\Ideal_{Y} \subset \Word_n$ of words containing
\textit{at least} one letter belonging to
$Y = S_{n+1} \smallsetminus S_{\PA}$ is also a lower set.
Indeed, consider $\sigma_0 \in Y$ and a word $w_1 \sqsubseteq \sigma_0$,
$w_1 = \sigma_{1,1}\cdots\sigma_{1,\ell}$.
From Lemma \ref{lemma:PAhat}, $\hat\sigma_0 \notin Z(\Quat_{n+1})$.
We thus have $\hat\sigma_{1,1}\cdots\hat\sigma_{1,\ell} =
\hat\sigma_0 \notin Z(\Quat_{n+1})$
and therefore $\hat\sigma_{1,k}  \notin Z(\Quat_{n+1})$ for some $k$:
again from Lemma \ref{lemma:PAhat},
we have $\sigma_{1,k} \in Y$, as desired.
The open subset $\cY_n = \cL_n[\Ideal_Y] \subset \cL_n$
is studied in Theorem \ref{theo:Y}.
\end{example}

\begin{example}
\label{example:IY2}
Recall from Remark \ref{rem:Yk} that for $n \ge 2$
we have 
\[ Y_2 = \{ a_1, a'_1 = a_2a_1, a_2, a'_2 = a_1a_2, \ldots,
a_{n-1}, a'_{n-1} = a_{n-2}a_{n-1}, a_n, a'_n = a_{n-1}a_n \}. \]
The set $\Ideal_{Y_2} \subset \Word_n$ of words containing
\textit{at least} one letter belonging to $Y_2$ is also a lower set.
Indeed, if $\sigma \in Y_2$, $\inv(\sigma) > 1$ and $w \sqsubseteq \sigma$
then either $w = \sigma$ or $w$ is a word of dimension $0$
and length $1$ or $3$.
The same conclusion holds assuming instead $w \preceq \sigma$,
implying the $\Ideal_{Y_2} \subset \Word_n$
is also a lower subset for $\preceq$.
It follows that $\cL_n[\Ideal_{Y_2}] \subseteq \cL_n$ is an open subset.
The lower set $\Ideal_{Y_2} \subset \Ideal_Y$ is considered again
in Proposition~\ref{prop:Yn2}.
\end{example}

\begin{example}
\label{example:I0}
Let $\Ideal_{[0]} \subset \Word_n$
be the set of words containing
\textit{at least} one letter of dimension $0$.
As in the previous examples, $\Ideal_{[0]}$ is a lower set
for both $\sqsubseteq$ and $\preceq$.
The lower set $\Ideal_{[0]} \subset \Ideal_{Y_2}$
is considered again
in Proposition~\ref{prop:ideal0isloose}.
\end{example}

\begin{example}
\label{example:uppersets}
The following are examples of upper sets:
\begin{gather*}
\Filter_{2,2} = \{[aba]\} \subset \Word_2, \qquad
\Filter_{2,3} = \{[aba], [bacb], [bcb] \} \subset \Word_3, \\
\{ [abc], [ac], [cba] \} \subset \Word_3, \qquad
\Filter_{\star,3} = \{[cba]\} \subset \Word_3, \\
\Filter_{4,4} = \{[34521], [32541], [52341], [52143], [54123] \}
\subset \Word_4.
\end{gather*}
The example $\Filter_{2,2}$ is easy:
there is no letter in $S_3$
of greater dimension than $[aba]$.
For $\Filter_{2,3}$,
verify that there are no other letters
$\sigma \in S_4$ with $\hat\sigma = \longhat([aba]) = -1$
(notice also that $[aba] \sqsubseteq [bacb]$, $[bcb] \sqsubseteq [bacb]$).
Also, the only letters $\sigma \in S_4$ with
$\hat\sigma = \hat a\hat c$ are $[abc]$, $[ac]$ and $[cba]$
(notice that $[ac] \sqsubseteq [abc]$, $[ac] \sqsubseteq [cba]$).
The only letters $\sigma \in S_5 \smallsetminus \{e\}$ with $\hat\sigma = 1$
are the elements of $\Filter_{4,4}$.
\end{example}

\begin{example}
\label{example:tokuppersets}
Consider $n = 4$ and the set
\[ \Filter_{2,4} = \{[aba], [bacb], [bcb], [cbdc], [cdc] \} \subset \Word_4. \]
Notice that these are the letters in Figure \ref{fig:blackaba}.
The set $\Filter_{2,4}$ is not an upper set for $\sqsubseteq$:
we have $\sigma \sqsubseteq \eta$ for all $\sigma \in \Filter_{2,4}$.
It is, however, a double $(\sqsubseteq,\preceq)$ upper set,
meaning that
\[ \forall w_0 \in \Filter_{2,4}, \forall w_1 \in \Word_4,
((w_0 \sqsubseteq w_1) \land (w_0 \preceq w_1))
\implies w_1 \in \Filter_{2,4}: \]
this follows from Example \ref{example:partialorders}.
This is sufficient to prove 
(with Equations \eqref{equation:preceq} and \eqref{equation:sqsubseteq})
that the subset
$\cL_4[\Filter_{2,4}] \subset \cL_4$ is closed.
\end{example}

\section{Construction of $\cD_n$ and
Theorems \ref{theo:shortCW} and \ref{theo:newCW}}
\label{sect:valid}

We now describe the construction of the CW complex $\cD_n$.
In a nutshell, we proceed as follows.
For each word $w$, there is a cell $c_w$ of dimension $\dim(w)$.
The CW complex $\cD_n$ can be constructed
in increasing order (with respect to $\sqsubseteq$ in $\Word_n$).
For each new word $w$, the glueing map $g_w$ for the new cell $c_w$ 
takes $\Ss^{d-1}$ (for $d = \dim(w)$)
to previously constructed cells $c_{\tilde w}$, $\tilde w \sqsubseteq w$.
The glueing map is determined
by studying a transversal section to $\cL_n[w] \subset \cL_n$,
as constructed in Section 7 of \cite{Goulart-Saldanha1}.
Notice that the image of such a transversal section
is contained in the union of the strata
$\cL_n[\tilde w]$, $\tilde w \sqsubseteq w$.

More precisely, 
consider a lower set $\Ideal \subseteq \Word_n$
(for the partial order $\sqsubseteq$).  Let 
\[ \cL_n[\Ideal] = \bigcup_{w \in \Ideal} \cL_n[w] \subseteq \cL_n. \]
The set $\cL_n[\Ideal] \subset \cL_n$ is an open subset
and therefore a Hilbert manifold
(recall that we work in the Hilbert space
$H^r$ of curves).
For each lower set we have a CW complex $\cD_n[\Ideal]$
and a homotopy equivalence $\cD_n[\Ideal] \to \cL_n[\Ideal]$.
Furthermore, we may assume the homotopy equivalence 
to be an inclusion $\cD_n[\Ideal] \subset \cL_n[\Ideal]$.
Indeed, the induction step is as follows.
Let $\tilde\Ideal \subset \Ideal \subseteq \Word_n$
be lower sets with $\Ideal = \tilde\Ideal \cup \{w_0\}$.
Assume that a CW complex $\cD_n[\tilde\Ideal] \subset \cL_n[\tilde\Ideal]$
has been previously constructed and that the inclusion
is a homotopy equivalence.
We show in Lemma \ref{lemma:goodstep}
how to add a cell $c_{w_0}$ to $\cD_n[\tilde\Ideal]$
to obtain the desired CW complex $\cD_n[\Ideal] \subset \cL_n[\Ideal]$,
the inclusion also being a homotopy equivalence.
It then suffices to take the union of a chain of CW complexes
to obtain $\cD_n$ (see Lemma \ref{lemma:goodchain}).

We sum up our construction in a theorem.

\begin{theo}
\label{theo:newCW}
There exists a CW complex $\cD_n$
and a continuous map $i: \cD_n \to \cL_n$
with the following properties:
\begin{enumerate}
\item{For each word $w \in \Word_n$ there exists a cell $c_w$
of $\cD_n$ of dimension $\dim(w)$.}
\item{For every lower set $\Ideal \subseteq \Word_n$
(with respect to the partial order $\sqsubseteq$)
the union of the cells $c_w$, $w \in \Ideal$,
is a subcomplex $\cD_n[\Ideal] \subseteq \cD_n$.}
\item{For every lower set $\Ideal \subseteq \Word_n$
the union of the strata $\cL_n[w]$, $w \in \Ideal$,
is an open subset $\cL_n[\Ideal] \subseteq \cL_n$.}
\item{For every lower set $\Ideal \subseteq \Word_n$
the restriction $i|_{\cD_n[\Ideal]}: \cD_n[\Ideal] \to \cL_n[\Ideal]$
is a homotopy equivalence.}
\end{enumerate}
\end{theo}

Clearly, Theorem \ref{theo:shortCW} follows from Theorem \ref{theo:newCW}.
We remind the reader that a similar but simpler construction
(for a finite stratification of another space)
is presented in \cite{Alves-Saldanha2}
(the proof of that Theorem 2 is particularly relevant).

We state a few definitions aimed at the proof of Theorem \ref{theo:newCW}.
Here $\DD^k$ is the closed disk of dimension $k$.
Let $\Ideal \subseteq \Word_n$ be a lower set (for $\sqsubseteq$).
A \emph{valid (CW) complex} $\cD_n[\Ideal]$ (for $\Ideal$)
has the following ingredients and properties:
\begin{enumerate}[label=(\roman*)]
\item{The CW complex $\cD_n[\Ideal]$ has
one cell of dimension $\dim(w)$ for each $w \in \Ideal$.
For every $w \in \Ideal$, the continuous glueing map
$g_w: \partial\DD^{\dim(w)} \to \cD_n[\Ideal]$
has image contained in the restriction
$\cD_n[\Ideal_w^\ast] \subset \cD_n[\Ideal]$
(which is also a valid CW complex).
Let $i_w: \DD^{\dim(w)} \to \cD_n[\Ideal]$ be the inclusion
(with quotient at the boundary).}
\item{For each $w \in \Ideal$,
we have a continuous map $c_w: \DD^{\dim(w)} \to \cL_n[\Ideal_w]$
(which will also be called the $w$ cell).}
\item{The maps $c_w$ are compatible:
if $i_{w_0}(p_0) = i_{w_1}(p_1)$ then
$c_{w_0}(p_0) = c_{w_1}(p_1)$.
In particular, if $w_1 \sqsubseteq w_0$,
$p_0 \in \partial\DD^{\dim(w_0)}$,  $p_1 \in \DD^{\dim(w_1)}$ and
$g_{w_0}(p_0) = i_{w_1}(p_1)$
then $c_{w_0}(p_0) = c_{w_1}(p_1)$.
The union of the cells $c_w$ naturally
defines a continuous map $c_\Ideal: \cD_n[\Ideal] \to \cL_n[\Ideal]$.}
\item{For any $w \in \Ideal$,
the cell $c_w$ is a topological embedding in the interior of $\DD^{\dim(w)}$,
smooth in the ball of radius $\frac12$,
and intersects $\cL_n[w]$ transversally at $c_w(0) \in \cL_n[w]$.}
\item{If $s \in \DD^{\dim(w)}$, $s \ne 0$,
then $c_w(s) \in \cL_n[\Ideal^\ast(w)]$
(and therefore $c_w(s) \notin \cL_n[w]$).}
\end{enumerate}


We abuse notation and identify a valid complex
with the corresponding family of cells and write
$\cD_n[\Ideal] = (c_w)_{w \in \Ideal}$
(notice that this family yields all the necessary information).
Given a valid complex $\cD_n[\Ideal] = (c_w)_{w \in \Ideal}$
and a lower set $\tilde\Ideal \subseteq \Ideal$,
the subcomplex $(c_w)_{w \in \tilde\Ideal}$
is also valid (for $\tilde\Ideal$).
We call this subcomplex $\cD_n[\tilde\Ideal]$
the \emph{restriction} of $\cD_n[\Ideal]$ to $\tilde\Ideal$.

\begin{rem}
\label{rem:orientation}
It follows from (iv) and (v) that the map $c_w$ intersects
a tubular neighborhood $\hat\cA_w \supset \cL_n[w]$
transversally around $c_w(0) \in \cL_n[w]$.
Such a tubular neighborhood $\hat\cA_w$ is 
constructed in the proof of Theorem 2
in \cite{Goulart-Saldanha1}, Section 6.
A map $\hat F_w: \hat\cA_w \to \DD^{\dim(w)}$
with $\hat F_w^{-1}[\{0\}] = \cL_n[w]$ is also constructed.
The map $\hat F_w \circ c_w$ is then a homeomorphism from
an open neighborhood of $0 \in \DD^{\dim(w)}$
to an open neighborhood of $0 \in \DD^{\dim(w)}$.
The map $\hat F_w$ also yields a transversal orientation to
$\cL_n[w] \subset \cL_n$
(the preimage orientation for $\hat F_w$;
see \cite{Goulart-Saldanha1} after Lemma 7.1).
We assume that the map $c_w$ respects this orientation,
that is, that the composition $\hat F_w \circ c_w$
respects orientation.
Lower dimensional cases shall be explicitly presented in examples.
\end{rem}




A valid complex $\cD_n[\Ideal]$ is \emph{good} if for every lower set
$\tilde\Ideal \subseteq \Ideal$ the map
$c_{\tilde\Ideal}: \cD_n[\tilde\Ideal] \to \cL_n[\tilde\Ideal]$
is a weak homotopy equivalence.
Clearly, a restriction of a good complex is also a good complex.

Let $J$ be a totally ordered set.
Let $(\Ideal_j)_{j \in J}$ be a family of lower sets such that
$j_0 < j_1$ implies $\Ideal_{j_0} \subseteq \Ideal_{j_1}$.
Let $\Ideal = \bigcup_{j \in J} \Ideal_j$,
which is therefore also a lower set.
For each $j \in J$, let $\cD_n[\Ideal_j]$ be a valid complex.
Assume that if $j_0 < j_1$ then
$\cD_n[\Ideal_{j_0}]$ is the restriction
of $\cD_n[\Ideal_{j_1}]$ to $\Ideal_{j_0}$.
Then the union $\cD_n[\Ideal]$ of all the valid complexes 
$\cD_n[\Ideal_j]$ (for all $j \in J$) is a valid complex for $\Ideal$.
A family $(\cD_n[\Ideal_j])_{j \in J}$ of valid complexes is a \emph{chain}
if it satisfies the conditions above.


\begin{lemma}
\label{lemma:goodchain}
Let $(\cD[\Ideal_j])_{j \in J}$ be a chain of good complexes.
Let $\Ideal = \bigcup_{j \in J} \Ideal_j$ and
$\cD_n[\Ideal]$ be the union of the complexes $\cD_n[\Ideal_j]$.
Then $\cD_n[\Ideal]$ is a good complex.
\end{lemma}

\begin{proof}
Let $\alpha: \Ss^k \to \cD_n[\Ideal]$ be a continuous map
which is homotopically trivial in $\cL_n[\Ideal]$,
i.e., there exists a map
$A_0: \DD^{k+1} \to \cL_n[\Ideal]$,
$A_0|_{\Ss^k} = c_{\Ideal} \circ \alpha$.
By compactness, there exists $j \in J$ such that the image of $A_0$
is contained in $\cL_n[\Ideal_j]$ so that we may write
$A_0: \DD^{k+1} \to \cL_n[\Ideal_j]$.
Since $\cD_n[\Ideal_j]$ is good,
there exists 
$A_1: \DD^{k+1} \to \cD_n[\Ideal_j]$
such that $A_1|_{\Ss^k} = \alpha$.

Let $\alpha_0: \Ss^k \to \cL_n[\Ideal]$ be a continuous map.
By compactness, there exists $j \in J$ such that the image
of $\alpha_0$ is contained in $\cL_n[\Ideal_j]$.
Since $\cD_n[\Ideal_j]$ is good,
there exists $\alpha_1: \Ss^k \to \cD_n[\Ideal_j] \subset \cD_n[\Ideal]$
such that $c_{\Ideal_j} \circ \alpha_1$
is homotopic to $\alpha_0$ in $\cL_n[\Ideal_j]$
(and therefore also in $\cL_n[\Ideal]$).

This completes the proof that $c_{\Ideal}: \cD_n[\Ideal] \to \cL_n[\Ideal]$
is a weak homotopy equivalence.
The proof for $\tilde\Ideal \subseteq \Ideal$ is similar.
\end{proof}

\begin{lemma}
\label{lemma:goodstep}
Let $\tilde\Ideal \subset \Ideal \subseteq \Word_n$ be lower sets
with $\Ideal \smallsetminus \tilde\Ideal = \{w_0\}$.
Let $\cD_n[\tilde\Ideal]$ be a good complex (for $\tilde\Ideal$).
Then this complex can be extended to a valid complex
$\cD_n[\Ideal]$ (for $\Ideal$).
Moreover, all such valid complexes are good.
\end{lemma}

The proof of this lemma is the longest and most technical part
of the present section: we postpone it for a while.
Before proving Lemma \ref{lemma:goodstep} we present 
the main conclusion for this section.

\begin{lemma}
\label{lemma:goodcomplex}
For any lower set $\Ideal \subseteq \Word_n$
there exists a valid complex $\cD_n[\Ideal] \subset \cL_n[\Ideal]$.
All such valid complexes are good.
\end{lemma}

\begin{proof}
The partial order $\sqsubseteq$ is well founded and
can therefore be extended to a (total) well-ordering.
Let $\gamma^{+}$ be the set of all initial segments of $\Ideal$
under this well-ordering.
This defines a chain of lower sets
$(\Ideal_{(\beta)})_{\beta \in \gamma^{+}}$
with a top element $\Ideal_{(\gamma)} = \Ideal$.
From now on in this proof we identify 
such $\beta$, $\gamma$ and $\gamma^{+}$ with ordinal numbers;
in particular, $\gamma^{+} = \gamma + 1$.

We prove by transfinite induction on $\alpha \le \gamma$
that there exists a chain of good complexes
$(\cD_n[\Ideal_{(\beta)}])_{\beta < \alpha}$.
For $\alpha = 0$ there is nothing to do.
For $\alpha = \tilde\alpha + 1$ (a successor ordinal),
apply Lemma \ref{lemma:goodstep}.
For $\alpha$ a limit ordinal apply Lemma \ref{lemma:goodchain}.

The fact that all valid complexes are good is likewise
proved by transfinite induction.
We again use Lemma  \ref{lemma:goodstep} for successor ordinals
and Lemma \ref{lemma:goodchain} for limit ordinals.
\end{proof}

Theorems~\ref{theo:shortCW} and~\ref{theo:newCW} are now easy.

\begin{proof}[Proof of Theorems \ref{theo:shortCW} and \ref{theo:newCW}]
All claims in the statement of Theorem \ref{theo:newCW}
follow from Lemma \ref{lemma:goodcomplex},
together with the definitions of valid and good complexes.
Theorem~\ref{theo:shortCW} is a direct consequence of Theorem~\ref{theo:newCW}.
\end{proof}




\begin{example}
\label{example:goodcomplex}
Following the above construction for $\Ideal_{[aba]}$
we have the cell shown in Figure \ref{fig:ababcb};
the transversal map $\tilde c$ in the proof above
is shown in Figure \ref{fig:aba-cw}.
Notice that $c_{[bcb]}$ is very similar.
The immediate predecessors of $[aba]$ are
$[ba]a$, $b[ab]$, $[ab]b$ and $a[ba]$ (see Figure \ref{fig:subaba-cw})
and we have
\[ \partial[aba] = [ba]a + b[ab] - [ab]b - a[ba], \quad
\partial[bcb] = [cb]b + c[bc] - [bc]c - b[cb]. \]
Here we use homological notation.
The cells are oriented as in Remark \ref{rem:orientation}.
\end{example}

\begin{figure}[ht]
\def\svgwidth{12cm}
\centerline{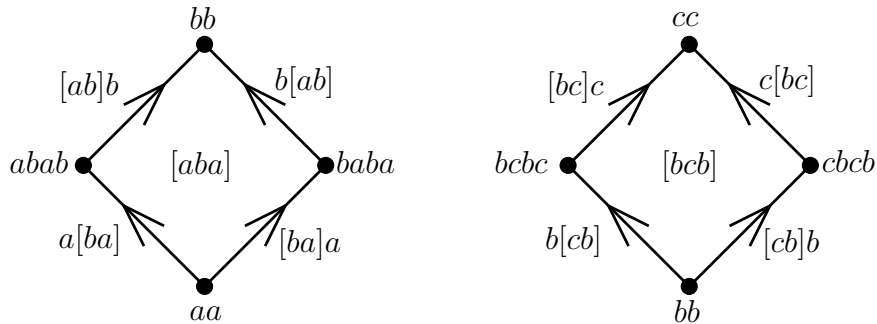}
\caption{The cells $c_{[aba]}$ and $c_{[bcb]}$.}
\label{fig:ababcb}
\end{figure}

\begin{rem}
\label{rem:goodcomplex}
As we shall see,
for cells of very low dimension the complex
can be taken to be polyhedric;
it is not clear whether this is true in general.
We shall not attept to clarify this and related issues here.
\end{rem}

\begin{proof}[Proof of Lemma \ref{lemma:goodstep}]
Assume $\cD_n[\tilde\Ideal]$ given (and good):
we construct the cell $c_{w_0}$ so that $\cD_n[\Ideal]$ is also good.
Let $k = \dim(w_0)$.
For a small ball $B \subset \RR^{k}$ around the origin,
construct as in Section 7 
of \cite{Goulart-Saldanha1} (after Lemma 7.1) a smooth map
$\phi: B \to \cL_n[\Ideal] \subseteq \cL_n$
with $\phi(0) = \Gamma_0 \in \cL_n[w_0]$
and
transversal to $\cL_n[w_0] \subset \cL_n$ at this point.
By taking $r > 0$ sufficiently small,
the image under $\phi$ of a ball $B(2r) \subset B$
of radius $2r$ around the origin satisfies the following condition:
$s \in B(2r)$ and $s \ne 0$ imply
$\phi(s) \in \cL_n[\Ideal^\ast_{w_0}] \subseteq \cL_n[\tilde\Ideal]$.

Let $S(r) \subset B(2r)$ be the sphere of radius $r$ around the origin
so that $\phi|_{S(r)}: S(r) \to \cL_n[\Ideal^\ast_{w_0}]$.
Since $\cD_n[\tilde\Ideal]$ is good,
the map $c_{\Ideal^\ast_{w_0}}:
\cD_n[\Ideal^\ast_{w_0}] \to \cL_n[\Ideal^\ast_{w_0}]$
is a weak homotopy equivalence.
There exists therefore a map
$g_{w_0}: \Ss^{k-1} \to \cD_n[\Ideal^\ast_{w_0}]$
such that $c_{\Ideal^\ast_{w_0}} \circ g_{w_0}$
is homotopic (in $\cL_n[\Ideal^\ast_{w_0}]$) to $\phi|_{S(r)}$.
In other words, there exists $c_{w_0}: \DD^k \to \cL_n[\Ideal_{w_0}]$
coinciding with $\phi$ in $B(r)$,
assuming values in $\cL_n[\Ideal^\ast_{w_0}]$ outside $0$
and with boundary $c_{\Ideal^\ast_{w_0}} \circ g_{w_0}$.
This completes the construction of a valid complex $\cD_n[\Ideal]$
and the proof of the first claim.


We now prove the second claim.
More precisely, let $c_{w_0}$ be such that extending
the good complex $\cD_n[\tilde\Ideal]$ by $c_{w_0}$
defines a valid complex $\cD_n[\Ideal]$ (for $\Ideal$):
we prove that this is also a good complex.
In other words, we prove 
that the map $c_{\hat\Ideal}: \cD_n[\hat\Ideal] \to \cL_n[\hat\Ideal]$
is a weak homotopy equivalence for any lower set
$\hat\Ideal \subseteq \Ideal$.
Again, we present the proof for $\hat\Ideal = \Ideal$;
the general case is similar.
The construction is similar to that of certain classical results
involving CW complexes;
here $\cL_n[w_0] \subset \cL_n[\Ideal]$
is a submanifold of class $C^{r-1}$.
Recall that $\cL_n[w_0] \subseteq \cL_n[\Ideal]$
is a contractible closed subset (of $\cL_n[\Ideal]$) and
a (globally) collared
submanifold of codimension $d = \dim(w_0)$.
Indeed, in the proof of Theorem 2 
in \cite{Goulart-Saldanha1}
(Section~6, near Remark~6.8)
we constructed a tubular neighborhood of $\hat\cA_{w_0} \supset \cL_n[w_0]$,
a projection $\Pi: \hat\cA_{w_0} \to \cL_n[w_0]$
and a map $\hat F = \hat F_{w_0}:  \hat\cA_{w_0} \to \BB^d$
(where $\BB^d$ is the unit open ball)
such that $(\Pi,\hat F): \hat\cA_{w_0} \to (\cL_n[w_0], \BB^d)$
is a homeomorphism.
By construction,
the map $c_{w_0}$ intersects $\hat\cA_{w_0} \supset \cL_n[w_0]$
transversally so that
$c_{w_0}^{-1}[\hat\cA_{w_0}]$ is an open neighborhood of $0 \in \BB^d$.
The map $\hat F \circ c_{w_0}$ (where defined) is a homeomorphism
from a neighborhood of $0$ to a neighborhood of $0$;
indeed, in the example constructed in the first part of the proof,
it is the multiplication by a positive constant.
We may assume that $c_{w_0}$ is such that
$\Pi(c_{w_0}(p)) = \Gamma_0$ for $p$ in an open neighborhood
$B_0 \subset \BB^d$, $0 \in B_0$
(this is actually true for the maps constructed in \cite{Goulart-Saldanha1}).
Let $\bump: \BB^d \to [0,1]$ be a bump function
with support contained in the neighborhood $B_0$ above
and constant equal to $1$ in a smaller open ball
$B_1 \subset B_0$, $0 \in B_1$.

Consider a compact manifold $M_0$ of dimension $k$
and continuous map $\alpha_0: M_0 \to \cL_n[\Ideal]$.
We prove that there exists
$\alpha_1: M_0 \to \cD_n[\Ideal]$
such that
$\alpha_0$ is homotopic to $c_{\Ideal} \circ \alpha_1$. 
First deform $\alpha_0$ to obtain a map
$\alpha_{\frac{1}{3}}$ which intersects
$\cL_n[w_0]$ transversally.
By using the contractibility of $\cL_n[w_0]$,
we may deform $\alpha_{\frac{1}{3}}$ to obtain a map 
$\alpha_{\frac{2}{3}}$ which intersects
a thinner tubular neighborhood contained in $\hat\cA_{w_0}$
only along the image of $c_{w_0}$.
More precisely, take $\Gamma_0 = c_{w_0}(0) \in \cL_n[w_0]$ as a base point;
let $H: [0,1] \times \cL_n[w_0] \to \cL_n[w_0]$
be a homotopy such that,
for all $\Gamma \in \cL_n[w_0]$,
$H(0,\Gamma) = \Gamma$ and
$H(1,\Gamma) = \Gamma_0$.
For $s \in [\frac13,\frac23]$, we define $\alpha_s$
to coincide with $\alpha_{\frac13}$ outside $\hat\cA_{w_0}$;
if $\alpha_{\frac13}(p) \in \hat\cA_{w_0}$ define
$\alpha_s(p) \in \hat\cA_{w_0}$ to satisfy
$ \hat F_{w_0}(\alpha_s(p)) = \hat F_{w_0}(\alpha_{\frac13}(p)) $
and 
\[ \Pi(\alpha_s(p)) = 
H((3s-1)\bump(\hat F_{w_0}(\alpha_{\frac13}(p))),\Pi(\alpha_{\frac13}(p))). \]
Thus, the set of points $p \in M_0$ such that
$\alpha_{\frac23}(p) \in \hat\cA_{w_0}$ and
$\hat F_{w_0}(\alpha_{\frac23}(p)) \in B_1$
consists of open tubes taken by $\alpha_{\frac23}$
to the image of the cell $c_{w_0}$.
By removing these open tubes we obtain $M_1 \subseteq M_0$,
a manifold with boundary, also of dimension $k$.
Let $\cD_n^\ast[\Ideal] =
\cD_n[\Ideal] \smallsetminus i_{w_0}[B_1]$ so that
both the map $c_{\Ideal}|_{\cD_n^\ast[\Ideal]}:
\cD_n^\ast[\Ideal] \to \cL_n[\tilde\Ideal]$
and the inclusion $\cD_n[\tilde\Ideal] \subset \cD_n^\ast[\Ideal]$
are weak homotopy equivalences
(since $\cD_n[\tilde\Ideal]$ is assumed to be good).
The restriction
$\alpha_{\frac{2}{3}}|_{M_1}: M_1 \to \cL_n[\tilde\Ideal]$
has boundary
$\alpha_{\frac23}|_{\partial M_1} =
c_{\Ideal} \circ \beta_1$
for $\beta_1: \partial M_1 \to \cD_n^\ast[\Ideal]$.
There exists
$\beta: M_1 \to \cD_n^\ast[\Ideal]$
with $\beta_1 = \beta|_{\partial M_1}$
and such that
$c_{\Ideal} \circ \beta: M_1 \to \cL_n[\tilde\Ideal]$
is homotopic to 
$\alpha_{\frac{2}{3}}|_{M_1}: M_1 \to \cL_n[\tilde\Ideal]$
(in $\cL_n[\tilde\Ideal]$).
Define $\alpha_1$ to coincide with $\beta$ in $M_1$
and such that $c_{w_0} \circ \alpha_1$
coincides with $\alpha_{\frac{2}{3}}$
in the tubes $M_0 \smallsetminus M_1$.
This is our desired map.

Conversely, consider a compact manifold with boundary $M$
and its boundary $\partial M = N$.
Consider a maps $\alpha_0: M \to \cL_n[\Ideal]$
and $\beta_0: N \to \cD_n[\Ideal]$
such that $c_{\Ideal} \circ \beta_0 = \alpha_0|_{N}$.
We prove that there exists a map $\alpha_1: M \to \cD_n[\Ideal]$
with $\alpha_1|_{N} = \beta_0$.
Furthermore, $\alpha_0$ and
$c_{\Ideal} \circ \alpha_1$ are homotopic in $\cL_n[\Ideal]$,
with the homotopy constant in the boundary.
Again, we may assume without loss of generality that
$\alpha_0$ is transversal to $\hat\cA_{w_0}$.
As in the previous paragraph,
use the contractibility of $\cL_n[w_0]$ to deform
$\alpha_0$ in $\hat\cA_{w_0}$ towards $c_{w_0}$,
thus defining $\alpha_{\frac12}$;
notice that this keeps the boundary fixed, as required.
Again remove tubes around points taken to $c_{w_0}$,
thus defining $M_1 \subset M$, also a manifold with boundary.
By hypothesys, $\alpha_{\frac12}$ can be deformed in $M_1$
to obtain $\alpha_1: M_1 \to \cD_n[\tilde\Ideal]$;
more precisely, $\alpha_{\frac12}: M_1 \to \cL_n[\tilde\Ideal]$
is homotopic to $c_{\tilde\Ideal}\circ\alpha_1$.
As in the previous paragraph we fill in the tubes to define
$\alpha_1: M \to \cD_n[\Ideal]$, the desired map.
\end{proof}

\section{Examples of lower sets}
\label{sect:Dn1}


In this section we consider a few simple examples
of lower sets and valid complexes.
Recall that $\Ideal_{(\omega)} \subset \Word_n$ is
the lower set of words of dimension $0$.
Two other examples will be $\Ideal_{(\omega 2)}$,
the lower set of words of dimension at most $1$
and $\Ideal_{(\omega^2)}$,
the lower set of words with all letters of dimension at most $1$.
We clearly have
$\Ideal_{(\omega)} \subset \Ideal_{(\omega 2)} \subset \Ideal_{(\omega^2)}$:
we shall construct the corresponding valid complexes
$\cD_n[\Ideal_{(\omega)}] \subset
\cD_n[\Ideal_{(\omega 2)}] \subset \cD_n[\Ideal_{(\omega^2)}]$.

For each $w \in \Ideal_{(\omega)}$,
let $c_w \in \cL_n[w]$ be an arbitrary curve:
we think of $c_w$ as a vertex of the complex $\cD_n$.
In other words, the $0$-skeleton of $\cD_n$ is the infinite countable
set of vertices $\cD_n[\Ideal_{(\omega)}] = \{c_w, w \in \Ideal_{(\omega)} \}$;
the inclusion $\cD_n[\Ideal_{(\omega)}] \subset \cL_n[\Ideal_{(\omega)}]$
is a homotopy equivalence.

A word $w \in \Ideal_{(\omega 2)}$
of dimension $1$ is of the form $w = w_0 [a_k a_l] w_1$
where $w_0, w_1 \in \Ideal_{(\omega)} \subset \Word_n$
are (possibly empty) words of dimension $0$
and $k \ne l$ so that $a_ka_l \in S_{n+1}$ is an element of dimension $1$
(i.e., $\inv(a_ka_l) = 2$).
There are three cases:
case (i) is $l = k-1$,
case (ii) is $l = k+1$ and
case (iii) is $l > k+1$
(if $l < k-1$ we write $a_la_k$ instead;
notice that in this case the permutations $a_k$ and $a_l$ commute).
Recall from Theorem~2
in \cite{Goulart-Saldanha1} that
in each case the stratum $\cL_n[w]$ is a hypersurface
with an open stratum on either side.
From Remark \ref{rem:orientation},
there is a transversal orientation to $\cL_n[w]$.
Call the two open strata $\cL_n[w^{+}]$ and $\cL_n[w^{-}]$
in such a way that the transversal orientation is
from $\cL_n[w^{-}]$ to $\cL_n[w^{+}]$.
The words $w^{+}, w^{-} \in \Word_n$ have dimension $0$;
their values according to case are:
\begin{enumerate}[label=(\roman*)]
\item{$w^{-} = w_0 a_l w_1$, $w^{+} = w_0 a_k a_l a_k w_1$;}
\item{$w^{-} = w_0 a_k a_l a_k w_1$, $w^{+} = w_0 a_l w_1$;}
\item{$w^{-} = w_0 a_k a_l w_1$, $w^{+} = w_0 a_l a_k w_1$.}
\end{enumerate}
From computing multiplicities in each case,
it follows easily that
the only two words $\tilde w \in \Word_n$ such that $\tilde w \sqsubseteq w$
are $\tilde w = w^{\pm}$.
These words also satisfy $w^{\pm} \preceq w$,
and are the only two words apart from $w$ to do so.
In order to construct the $1$-skeleton $\cD_n[\Ideal_{(\omega 2)}]$ of $\cD_n$,
we add an oriented edge $c_w$ from $c_{w^{-}}$ to $c_{w^{+}}$.
The edge may be assumed to be contained in
$\cL_n[w^{-}] \cup \cL_n[w] \cup \cL_n[w^{+}]$
and to cross $\cL_n[w]$
transversally and precisely once;
edges are also assumed to be simple, and disjoint except at endpoints.
Again, the inclusion
$\cD_n[\Ideal_{(\omega 2)}] \subset \cL_n[\Ideal_{(\omega 2)}]$
is a homotopy equivalence.

The four sides of Figure \ref{fig:aba-cw} above provide examples
of edges of cases (i) and (ii).
Figure \ref{fig:abbaac} shows the edges of $\cD_n$.
Here we omit the initial and final words $w_0$ and $w_1$
(since they are not involved anyway).
We also prefer to present examples
(instead of spelling out conditions as above;
a more formal discussion is given
in Section~7
from \cite{Goulart-Saldanha1} 
and in Section \ref{sect:valid} above).
Thus, cases (i), (ii) and (iii) are represented by
$[ba]$, $[ab]$ and $[ac]$, respectively.

\begin{figure}[ht]
\def\svgwidth{10cm}
\centerline{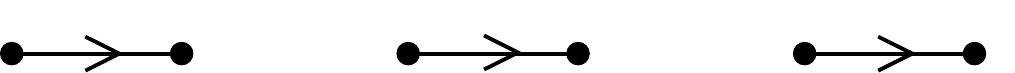}
\caption{Examples of edges of $\cD_n$.}
\label{fig:abbaac}
\end{figure}

In a homological language, we write
\[ \partial[ba] = bab - a, \quad \partial[ab] = b - aba, \quad
\partial[ac] = ca - ac. \]
Notice that in all cases $\hat w = \hat w^{+} = \hat w^{-} \in \Quat_{n+1}$,
consistently with Equation \ref{equation:sqsubseteq}.

As a simple application of our methods,
we compute the connected components of $\cD_n$
(i.e., we are reproving in our notation
the main result of \cite{Shapiro-Shapiro}).
We need only consider words of dimension $0$ (the vertices of $\cD_n$),
i.e., words in the generators $a_k$, $1 \le k \le n$.
Write $w_0 \sim w_1$ if $c_{w_0}$ and $c_{w_1}$
are in the same connected component;
in particular, 
if $w_0 \sim w_1$ then $\hat w_0 = \hat w_1$.
Figure \ref{fig:abbaac} illustrates that
\[ a \sim bab, \qquad aba \sim b, \qquad ac \sim ca. \]

For $()$ the empty word, the vertex $c_{()}$ is not attached
to any edge and thus forms a connected component in the complex $\cD_n$
(see also Remark \ref{rem:connected} below).
Notice that $\hat a \hat a \hat a \hat a = \longhat( {()} ) = 1$
but $aaaa \not\sim ()$.

\begin{prop}
\label{prop:connected}
Consider two non-empty words $w_0, w_1 \in \cD_n$ of dimension $0$.
Then $w_0 \sim w_1$ if and only if $\hat w_0 = \hat w_1 \in \Quat_{n+1}$.
\end{prop}

\goodbreak

\begin{rem}
\label{rem:connected}
For $q \in \Quat_{n+1}$,
the set $\Ideal_q = \{ w \in \Word_n \;|\; \hat w = q \}$
is both a lower set and an upper set.
Theorem 2 in  \cite{Goulart-Saldanha1} 
tells us that $\cD_n[\Ideal_q] \subset \cL_n(1;\acute\eta q\acute\eta)$.
For $q = 1$, we can further split $\Ideal_q$ into
$\Ideal_{q,\conv} = \{()\}$ and
$\Ideal_{q,\nconv} = \Ideal_q \smallsetminus \Ideal_{q,\conv}$.
It follows from Proposition~\ref{prop:connected}
that the connected components of $\cD_n$ are 
$\cD_n[\Ideal_q]$ ($q \in \Quat_{n+1}$, $q \ne 1$),
$\cD_n[\Ideal_{1,\conv}]$ and $\cD_n[\Ideal_{1,\nconv}]$.
Thus, from Theorem \ref{theo:newCW},
$\cL_n(1;q)$ is connected for $q \ne \hat\eta$.
The set $\cL_n(1;\hat\eta)$ has two connected components:
one of them is $\cL_{n,\conv}(1;\hat\eta) = \cL_n[()]$,
which is contractible;
its complement $\cL_n(1;\hat\eta) \smallsetminus \cL_n[()] =
\cL_n[\Ideal_{1,\nconv}]$ is also connected.

These results are of course well known 
but also follow from our methods.
\end{rem}


\begin{proof}[Proof of Proposition \ref{prop:connected}]
We have already seen that $w_0 \sim w_1$ implies $\hat w_0 = \hat w_1$.
We prove the other implication for non-empty words. 
A \emph{basic word} is either:
\begin{enumerate}[label=(\roman*)]
\item{ a non-empty word $a_{k_1}a_{k_2}\cdots a_{k_l}$
with $k_1 < k_2 < \cdots < k_l$; }
\item{ of the form $aaa_{k_1}a_{k_2}\cdots a_{k_l}$
with $k_1 < k_2 < \cdots < k_l$ (here the words $aa$ and $aaa$ are allowed:
they correspond to $l = 0$ and $l = 1$, $k_1 = 1$); }
\item{ $aaaa$.}
\end{enumerate}
Clearly, for each $z \in \Quat_{n+1}$ there exists a unique basic word $w$
with $z = \hat w$.
Using the edges in Figure \ref{fig:abbaac}, first notice that
\[ aa \sim abab \sim bb \sim bcbc \sim cc \sim \cdots \sim a_ka_k; \]
also 
\[ a_{k+1}a_k \sim a_{k+1}a_{k+1}a_ka_{k+1} \sim aaa_ka_{k+1}; \]
furthermore, $baa \sim babab \sim aab$ and, for $k > 2$,
$a_k aa \sim aa a_k$.
Also,
\[ a \sim bab \sim ababababa \sim aaaaa. \]
Thus $aa$ commutes with all generators $a_k$
and can be brought to the beginning of the word;
other generators either commute ($a_ka_l \sim a_la_k$ if $|k-l| \ne 1$)
or satisfy $a_{k+1}a_k \sim aaa_ka_{k+1}$
(corresponding to the fact that $\hat a_k$ and $\hat a_{k+1}$
anticommute in $\Quat_{n+1}$, and that $\hat a\hat a = -1$).
Thus, for an arbitrary non-empty word $w$,
generators can be arranged in increasing order of index,
at the price of creating copies of $aa$
which are taken to the beginning of the word.
Duplicate generators can also be transformed into further copies of $aa$.
Finally, if there are more than $4$ copies of $a$,
they can be removed $4$ by $4$ thus arriving at a basic word.
\end{proof}

\begin{example}
\label{example:abacadaba}
Take $n = 4$ and $w = abacadaba \in \Word_4$.
Applying the procedure in the proof of Proposition \ref{prop:connected}
we have
\begin{gather*}
abacadaba \sim abacadb \sim abacabd \sim abaacbd
\sim aaabcbd \sim aaacd,
\end{gather*}
and $w_1 = aaacd$ is a basic word,
with $\hat w = \hat w_1 = -\hat a\hat c\hat d$.
\end{example}

\begin{example}
\label{example:earlyaiH}
It follows from Proposition \ref{prop:connected}
that for every non-empty word $w \in \Word_n$
of dimension $0$ there exists a path $w^H$
in the $1$-skeleton of $\cD_n$ joining $w$ and $aaaaw$.
(Or more precisely, the vertices $c_w$ and $c_{aaaaw}$;
we omit the $c$'s for conciseness.)
We proceed to construct explicit paths
which shall be extensively used later.



The simplest example is:
\[ a^{H} = (a \xleftarrow{[ba]} bab
\xleftarrow{[ab]ab} \cdot \xrightarrow{aba[ba]b} \cdot \xrightarrow{ababab[ab]}
ababababa
\xleftarrow{a[ba]ababa} \cdot \xleftarrow{aaa[ba]a} aaaaa), \]
which already appeared implicitly
in the proof of Proposition \ref{prop:connected}.
If the word starts with $a$, we fall back on this example:
\[ (aw)^{H} = a^{H}w = (aw \rightarrow babw
\leftarrow\cdot\rightarrow\cdot\leftarrow ababababaw
\leftarrow\cdot\leftarrow aaaaaw). \]
Next, define
\begin{align*}
b^{H} &= (b\xleftarrow{[ab]} aba \xrightarrow{a^{H}ba} aaaaaba
\xrightarrow{aaaa[ab]} aaaab) \\
&= (b\leftarrow aba \rightarrow babba
\leftarrow\cdot\rightarrow\cdot\leftarrow abababababa
\leftarrow\cdot\leftarrow aaaaaba \rightarrow aaaab)
\end{align*}
and $(bw)^{H} = b^{H}w$.
Define recursively
\[ a_{k+1}^{H} = (a_{k+1}\xleftarrow{[a_ka_{k+1}]} a_ka_{k+1}a_k
\xrightarrow{a_k^{H}a_{k+1}a_k} aaaaa_ka_{k+1}a_k
\xrightarrow{aaaa[a_ka_{k+1}]} aaaaa_{k+1}). \]
Notice that for every $k$, the path $a_k^H$ begins as
\[ a_k^H = a_k \xleftarrow{a'_k} a^\star_k \cdots . \]
The permutation $a'_k$ is defined in
Section \ref{sect:permutations} (above Lemma \ref{lemma:vader});
the words $a_k^\star$ are given by
\[ a' = [ba], \quad a^\star = bab, \quad
a'_{k+1} = [a_ka_{k+1}], \quad a^\star_{k+1} = a_ka_{k+1}a_k. \] 
For longer words, set $(a_kw)^{H} = a_k^{H}w$.
\end{example}



We now construct the complex $\cD_n[\Ideal_{(\omega^2)}]$.
Notice that $w \in \Ideal_{(\omega^2)}$ if and only if
$w$ is of the form $w = w_0 \sigma_1 w_1 \cdots \sigma_l w_l$
where $\dim(w_j) = 0$ and $\dim(\sigma_j) = 1$
(some of the $w_j$ may equal the empty word).
Set $c_w$ to be a \emph{product cell} of dimension $l$,
i.e., the $l$-th dimensional cube
\[ c_w = c_{w_0} \times c_{\sigma_1} \times c_{w_1} \times \cdots
\times c_{\sigma_l} \times c_{w_l}. \]
In homological notation we have
\begin{align*}
\partial(w_0 \sigma_1 w_1 \cdots \sigma_l w_l) &= 
w_0(\partial \sigma_1)w_1 \cdots \sigma_l w_l 
- w_0 \sigma_1 w_1(\partial \sigma_2)\cdots \sigma_l w_l
+ \cdots \\
&\phantom{=} \cdots
+ (-1)^{l+1} w_0 \sigma_1 w_1 \cdots \partial(\sigma_l) w_l. 
\end{align*}
See Figure \ref{fig:baabacbac} for the following examples
of the previous formula:
\begin{gather*}
\partial([ba][ab]) = bab[ab] - [ba]b - a[ab] + [ba]aba, \\
\partial([ac]b[ac]) = cab[ac] - [ac]bca - acb[ac] + [ac]bac.
\end{gather*}

\begin{figure}[ht]
\def\svgwidth{10cm}
\centerline{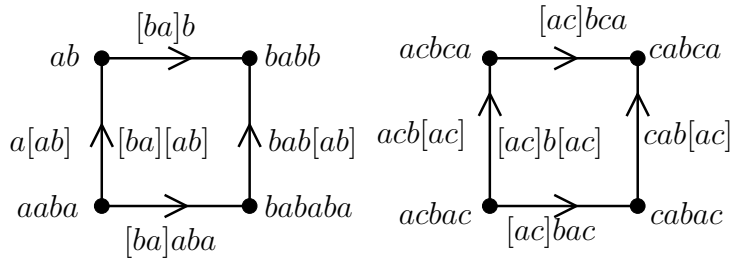}
\caption{The $2$-cells $[ba][ab]$ and $[ac]b[ac]$. }
\label{fig:baabacbac}
\end{figure}

\begin{rem}
\label{rem:productcells}
The construction of product cells works in greater generality.
If $w \in \Word_n$ is a word
containing more than one letter of positive dimension,
define $c_w$ as a \emph{product cell}.
As before, write $w = w_0 \sigma_1 w_1 \cdots \sigma_l w_l$
where $\dim(w_j) = 0$ and $\dim(\sigma_j) > 0$
(some of the $w_j$ may equal the empty word).
Set 
\[ c_w = c_{w_0} \times c_{\sigma_1} \times c_{w_1} \times \cdots
\times c_{\sigma_l} \times c_{w_l}. \]
Here we assume the cells $c_{\sigma_j}$ to have been previously constructed.
\end{rem}

\begin{rem}
\label{rem:ababcb}
The reader will agree that an important part
in the study of the CW complex $\cD_n$
is the construction of the boundary maps
for permutations $\sigma \in S_{n+1}$,
or, equivalently, of the cells $c_\sigma$.
The basic blocks in this construction are permutations $\sigma$
for which $1^\sigma > 1$.
For instance, the construction of the boundary map for $\sigma = bcb$
is easy 
if we already discussed $\tilde\sigma = aba$.

Indeed, let $\sigma \in S_{n+1}$ be a letter of dimension $k$
such that $1^\sigma  = 1$.
Equivalently, a reduced word for $\sigma$ does not use the generator $a_1$.
Write a reduced word
$\sigma = a_{n_1}\cdots a_{n_{k+1}}$ and $s = -1 +\min n_j > 0$.
Set $\tilde\sigma = a_{n_1 - s}\cdots a_{n_{k+1}- s} \in S_{n-s+1}$:
the cell $c_{\tilde\sigma}$ is assumed to be already constructed.
Define $c_{\sigma}$ from $c_{\tilde\sigma}$
by adding $s$ to the index of every generator of every letter.
Notice that this fits with our construction of the $1$-skeleton;
see also Figure  \ref{fig:ababcb} for $c_{[bcb]}$.
\end{rem}

\section{Cells of dimension $2$}
\label{sect:Dn2}

Given $n$ and Remarks \ref{rem:productcells} and \ref{rem:ababcb},
there is a finite (and short) list of possibilities
of letters of dimension $2$.
In $S_3$, the only letter of dimension $2$ is $[aba]$:
Figure \ref{fig:subaba-cw} shows the elements of $\Word_2$ (or $\Word_n$)
below $[aba]$, i.e., the smallest lower set containing $[aba]$
(in this case the partial orders $\preceq$ and $\sqsubseteq$ agree).
See also Figures \ref{fig:aba-cw} and \ref{fig:ababcb} for 
a transversal surface to the submanifold $\cL_2[[aba]]$
and for the cell $c_{[aba]}$.

In $S_4$ we also have $[bcb]$
(which, from Remark \ref{rem:ababcb},
is similar to $[aba]$, as in Figure \ref{fig:ababcb}).
The example $[acb]$ deserves special attention:
it is discussed in detail in Section 9 of \cite{Goulart-Saldanha1},
valid cells are given in Example \ref{example:acb}
and in Figure \ref{fig:newacb} below.
The three remaining cells are
and $[abc]$, $[bac]$ and $[cba]$,
for which transversal sections and valid cells are shown
in Figure \ref{fig:cD3};
in a homological notation:
\begin{align*}
\partial[abc] &= abc[ab]+a[bc]+[ac]-[bc]a-[ab]cba+ab[ac]ba, \\
\partial[bac] &= [ac]b-[bc]ab-bc[ba]-b[ac]-ba[bc]-[ba]cb, \\
\partial[cba] &= [ac] + c[ba]+cba[cb]+cb[ac]bc-[cb]abc-[ba]c.
\end{align*}

\begin{figure}[ht]
\def\svgwidth{\linewidth}
\centerline{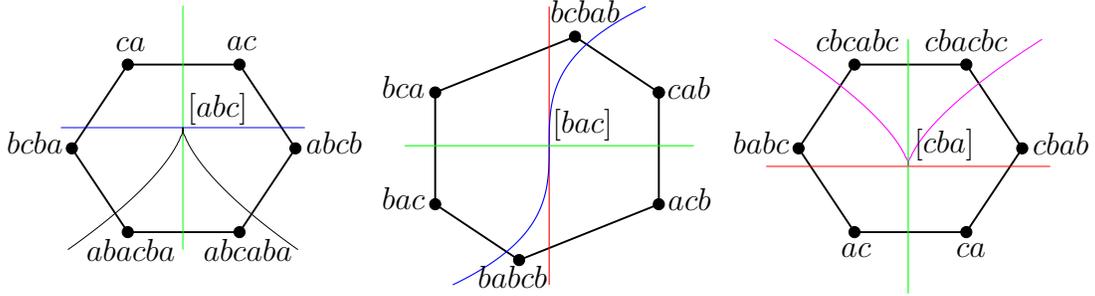}
\caption{The cells $c_{[abc]}$, $c_{[bac]}$ and $c_{[cba]}$.}
\label{fig:cD3}
\end{figure}
\bigskip

In each figure, we consider the family of paths constructed
in Section~7
of \cite{Goulart-Saldanha1}
so that the functions $m_j$ are polynomials
in the real parameters $x_1$ and $x_2$ and in the variable $t$.
The semialgebraic curves in the figure
indicate regions for which the itinerary has positive dimension
and therefore separate open regions
for which the itinerary has dimension $0$.

Consider $\sigma \in S_{n+1} \smallsetminus \{e\}$, $\dim(\sigma) = k$,
and the map $\phi: \DD^k \to \cL_n$ transversal to $\cL_n[\sigma]$
constructed in Lemma~7.1 
of \cite{Goulart-Saldanha1} 
(and discussed above in Section \ref{sect:valid}).
By construction of $\phi$,
if $x \in \DD^k \smallsetminus \{0\}$ and $\iti(\phi(x)) = w$
there exists a continuous map $h: [0,1] \to \DD^k$ such that
$h(0) = 0$, $h(1) = x$ and $\iti(\phi(h(s))) = w$ for all \( s \in (0,1] \).
In particular, we have both $w \sqsubseteq \sigma$ and $w \tok \sigma$.
As we shall see in Example \ref{example:acb}, the reciprocal does not hold.

\begin{example}
\label{example:abc}
Take $n = 3$ and $\sigma = abc$:
we discuss the first diagram of Figure \ref{fig:cD3}.
Following the construction in \cite{Goulart-Saldanha1},
we have
\[ M = \begin{pmatrix}
-\frac{t^2}{2} & -t & -1 & 0 \\
t & 1 & 0 & 0 \\
1 & 0 & 0 & 0 \\
\frac{t^3}{6} + x_2 t + x_1 & \frac{t^2}{2} + x_2 & t & 1
\end{pmatrix}. \]
and therefore
\[ m_1(t) = \frac{t^3}{6} + x_2 t + x_1, \qquad
m_2(t) = \frac{t^2}{2}+ x_2, \qquad
m_3(t) = -t. \]
The vertical line, corresponding to $[ac]$, has equation
\[ \resultant(m_1,m_3;t) = x_1 = 0. \]
More precisely, the positive semiaxis corresponds to itinerary $[ac]$
and the negative semiaxis corresponds to itinerary $ab[ac]ba$.
The horizontal line, corresponding to $[bc]$, has equation
\[ \discriminant(m_2;t) = -2 x_2 = 0: \]
the positive semiaxis corresponds to itinerary $a[bc]$
and the negative semiaxis to $[bc]a$.
Finally, the cusp-like curve in the figure, corresponding to $[ab]$,
has equation 
\[ \resultant(m_1,m_2;t) = \frac{x_2^3}{9} + \frac{x_1^2}{8} = 0; \]
in the third quadrant the itinerary is $[ab]cba$;
in the fourth quadrant it is $abc[ab]$.
\end{example}

\begin{example}
\label{example:acb}
Take $n = 3$ and $\sigma = acb$.
The transversal section constructed
in Example~7.3 
of \cite{Goulart-Saldanha1} implies that
\[ [ac]b[ac] \sqsubseteq [acb], \qquad [ac]b[ac] \tok [acb]. \]
Notice that $\dim([ac]b[ac]) = \dim([acb]) = 2$.
By transitivity we have
\[ acbac, cabca, [ac]bac, [ac]bca, acb[ac], cab[ac] \sqsubseteq [acb] \]
(and similarly for $\tok$)
but none of the itineraries on the left hand side appear
in the transversal section constructed 
in Example~7.3 
of \cite{Goulart-Saldanha1}.

It is easy, on the other hand, to perturb the map $\phi$
to obtain other maps transversal to $\cL_n[[ac]b[ac]]$.
Take
\[ M = \begin{pmatrix}
-t & -1 & 0 & 0 \\
\frac{t^3}{6} + x t & \frac{t^2}{2} + x & t & 0 \\
ut + 1 & u & 0 & 0 \\
\frac{t^2}{2} + y & t & 1 & 0
\end{pmatrix} \]
where $u$ is to be thought of as a fixed real number
of small absolute value: say, $|u| < 1/4$.
Figure \ref{fig:newacb} shows the resulting sections.
The construction in Example~7.3 
of \cite{Goulart-Saldanha1} 
corresponds to $u = 0$.
See Example~7.4 
of \cite{Goulart-Saldanha1}
for a slightly different construction;
see also Section~9 
of \cite{Goulart-Saldanha1}.

\begin{figure}[ht]
\def\svgwidth{125mm}
\centerline{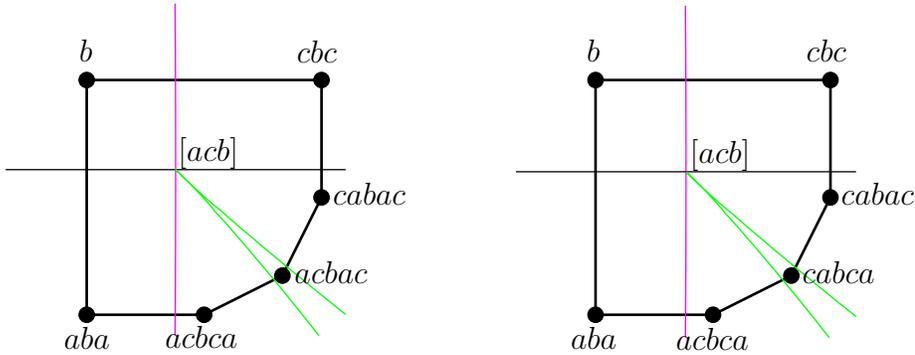}
\caption{Two other transversal sections to $\cL_3[[acb]]$
($u < 0$ and $u > 0$).}
\label{fig:newacb}
\end{figure}
\bigskip

Notice that the two diagrams differ combinatorially.
For $u < 0$, the itinerary $acbac$ appears and $cabca$ does not;
for $u > 0$ it is the other way round.
Provided $\Ideal \subset \Word_3$ contains
the $2$-dimensional word $[ac]b[ac]$,
the two boundary maps (from $\Ss^1$ to the $1$-skeleton)
are homotopic in $\cD_3[\Ideal]$,
thus guaranteeing that $\cD_3[\Ideal \cup \{[acb]\}]$ is well defined,
consistently with the results from Section \ref{sect:valid}.
\end{example}

\begin{figure}[ht]
\def\svgwidth{10cm}
\centerline{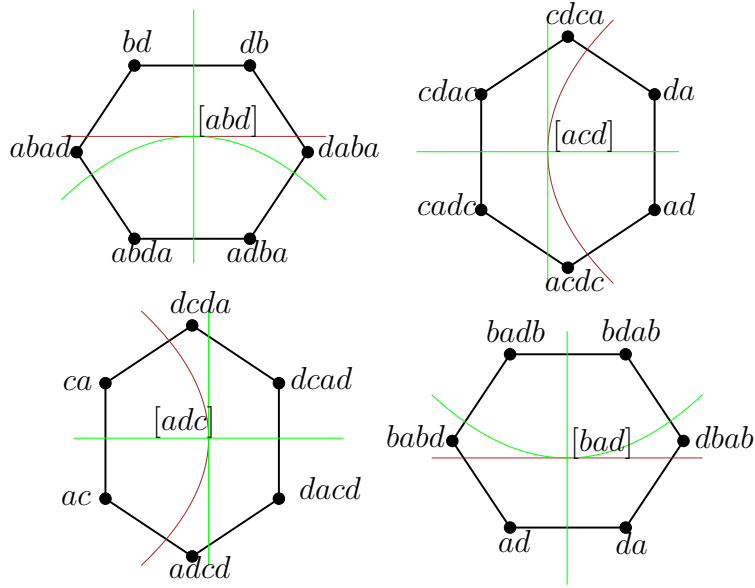}
\caption{The cells $c_{[abd]}$, $c_{[acd]}$, $c_{[adc]}$ and $c_{[bad]}$}
\label{fig:cD4}
\end{figure}

In $\cD_4$ there exist faces (i.e., cells of dimension $2$)
equivalent to those above in the sense of Remark \ref{rem:ababcb}
(such as $[bcd]$, which is equivalent to $[abc]$)
but also a few genuinely new ones:
$[abd]$, $[acd]$, $[adc]$ and $[bad]$,
all shown in Figure \ref{fig:cD4}; more generally, we have
\begin{align*}
\partial[aba_k] &=
ab[aa_k] + a[ba_k]a + [aa_k]ba + a_k[ab] - [ba_k] - [ab]a_k, \quad
k \ge 4; \\
\partial[aa_ka_{k+1}] &=
a[a_ka_{k+1}] + [aa_{k+1}] - [a_ka_{k+1}]a \\
& \qquad - a_ka_{k+1}[aa_k] - a_k[aa_{k+1}]a_k - [aa_k]a_{k+1}a_k, \quad
k \ge 3; \\
\partial[aa_{k+1}a_k] &=
a[a_{k+1}a_k] + [aa_{k+1}]a_ka_{k+1} + a_{k+1}[aa_k]a_{k+1} \\
& \qquad + a_{k+1}a_k[aa_{k+1}] - [a_{k+1}a_k]a - [aa_k], \quad
k \ge 3; \\
\partial[baa_k] &=
[aa_k] + a_k[ba] - [ba_k]ab - b[aa_k]b - ba[ba_k] - [ba]a_k, \quad
k \ge 4.
\end{align*}
The only genuinely new face in $\cD_5$ is $[a_1a_3a_5]$,
shown in Figure \ref{fig:cD5}; we have
\[ \partial[aa_ka_l] =
a[a_ka_l] + [aa_l]a_k + a_l[aa_k] - [a_ka_l]a - a_k[aa_l] - [aa_k]a_l, \quad
3 \le k < l-1. \]

\begin{figure}[ht]
\def\svgwidth{6cm}
\centerline{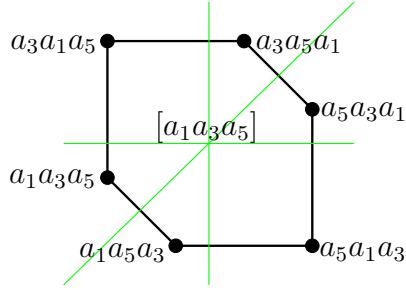}
\caption{The cell $c_{[a_1a_3a_5]}$.}
\label{fig:cD5}
\end{figure}

\section{Boundaries}
\label{sect:boundary}

The previous sections contain several examples
of boundary maps of cells $c_\sigma$, $\sigma \in S_{n+1}$.
We do not know a simple and general
description of such boundary maps in higher dimension.
In this section we prove some special cases.

Recall from Section \ref{sect:permutations}
that $\sigma_0 \blacktriangleleft \sigma_1$
if  $\sigma_0 \vartriangleleft \sigma_1$
and $\sigma_0 \equiv \sigma_1 \pmod 2$,
i.e., if $j_0 \equiv j_1 \pmod 2$.
The following proposition is the main result of this section.

\begin{prop}
\label{prop:blacktriangle}
Let $\sigma_0, \sigma_1 \in S_{n+1} \smallsetminus \{e\}$,
$w_0, w_1 \in \Word_n$.
If $\sigma_0 \blacktriangleleft \sigma_1$ then
\[ \cL_n[w_0\sigma_1 w_1] \subset \overline{\cL_n[w_0\sigma_0 w_1]}. \]
A valid boundary map for the cell $c_{w_0\sigma_1 w_1}$
can be taken so as to
include exactly one copy of $c_{w_0\sigma_0 w_1}$.
In particular,
$w_0\sigma_0 w_1 \preceq w_0\sigma_1 w_1$
and 
$w_0\sigma_0 w_1 \sqsubseteq w_0\sigma_1 w_1$.
\end{prop}

The proof of this proposition is postponed to the end of this section
so that we can present examples and a preliminary,
more technical lemma.
It follows easily from Remark \ref{rem:blacktriangle}
that if $\sigma_0 \vartriangleleft \sigma_1$
and $\sigma_0 \not\blacktriangleleft \sigma_1$ then
$\sigma_0 \not\sqsubseteq \sigma_1$.

\begin{example}
\label{example:blacktriangle}
In these examples we take $w_0$ and $w_1$ both empty.
The case $\sigma_0 = a \blacktriangleleft \sigma_1 = ba$ is
illustrated in Figure \ref{fig:abbaac},
as is the case $\sigma_0 = b \blacktriangleleft \sigma_1 = ab$.
The case $\sigma_0 = ac \blacktriangleleft \sigma_1 = abc$
is illustrated in Figure \ref{fig:cD3}
(the top side is $[ac]$).
The case $\sigma_0 = ac \blacktriangleleft \sigma_1 = cba$
is also illustrated in Figure \ref{fig:cD3}
(the bottom side is $[ac]$).
The cases $ab \blacktriangleleft acb$ and $cb \blacktriangleleft acb$
are both illustrated in Figure \ref{fig:newacb}.
Figure \ref{fig:cD4} illustrates the cases
$bd \blacktriangleleft abd$,
$ad \blacktriangleleft acd$,
$ac \blacktriangleleft adc$ and
$ad \blacktriangleleft bad$.

Notice that there is no $\sigma_0 \in S_4$
with $\sigma_0 \blacktriangleleft \sigma_1 = bac$
and consistently no side of $\partial\sigma_1$
is a word of length one;
see Figure \ref{fig:cD3}.
Also, $\sigma_1 \in Y = S_{n+1} \smallsetminus S_{\PA}$
is low, with $\sigma_1 \blacktriangleleft \sigma_1' = abac$.

Consider also $\sigma_1 = a_1a_3a_5 \in S_6$;
no side of the boundary of $c_{\sigma_1}$
is a word of length $1$, see Figure \ref{fig:cD5}.
In this case we have $\sigma_1 \in S_{\PA}$
and $\sigma_1'$ is undefined.
There is no $\sigma_0 \in S_6$ with $\sigma_0 \blacktriangleleft \sigma_1$;
we have, however, $\sigma_1   \blacktriangleleft a_1a_2a_3a_5$.
\end{example}

\begin{lemma}
\label{lemma:triangleleft}
Let $\sigma_0 \vartriangleleft \sigma_1 \in S_{n+1}$
with $\sigma_1 = (i_0i_1) \sigma_0 = \sigma_0 (j_0j_1)$.
For $\Gamma \in \cL_n[\sigma_1]$ 
there exist $0$-dimensional words
$w_0^{+}, w_0^{-}, w_1^{+}, w_1^{-} \in \Word_n$
such that
\[ \Gamma \in
\overline{\cL_n[w_0^{+} \sigma_0 w_1^{+}]} \cap
\overline{\cL_n[w_0^{-} \sigma_0 w_1^{-}]}. \]
In particular, we have
\[
w_0^{+} \sigma_0 w_1^{+} \sqsubseteq \sigma_1,
\qquad
w_0^{+} \sigma_0 w_1^{+} \tok \sigma_1,
\qquad
w_0^{-} \sigma_0 w_1^{-} \sqsubseteq \sigma_1.
\qquad
w_0^{-} \sigma_0 w_1^{-} \tok \sigma_1.
\]
If $\sigma_0 \blacktriangleleft \sigma_1$ then
$w_0^{+}$ and $w_1^{+}$ are both empty and
\[ \mult_k(w_0^{-}) = \mult_k(w_1^{-}) = [i_0 \le k < i_1]. \]
If $\sigma_0 \not\blacktriangleleft \sigma_1$ then
\begin{gather*}
\mult_k(w_0^{-}) = \mult_k(w_1^{+}), \qquad
\mult_k(w_0^{+}) = \mult_k(w_1^{-}), \\
\mult_k(w_0^{-}) + \mult_k(w_0^{+}) = [i_0 \le k < i_1].
\end{gather*}
\end{lemma}

\begin{rem}
\label{rem:triangleleft}
Notice that we are not claiming that the two words
$w_0^{\pm}\sigma_0 w_1^{\pm}$
are the \textit{only} words of the form
$\tilde w_0\sigma_0 \tilde w_1$
for which we have
$\Gamma \in \overline{\cL_n[\tilde w_0\sigma_0 \tilde w_1]}$,
$\tilde w_0\sigma_0 \tilde w_1 \sqsubseteq \sigma_1$ or
$\tilde w_0\sigma_0 \tilde w_1 \preceq \sigma_1$.
In full generality such a claim is not correct.
Indeed, for $\sigma_1 = [acb]$ and $\sigma_0 = [ac]$
we have
\[ acb[ac], cab[ac], [ac]bac, [ac]bca, [ac]b[ac] \sqsubseteq [acb] \]
and similarly for $\preceq$.
\end{rem}




The following proof makes systematic use of both
Theorem 4 and Lemma 7.2 in \cite{Goulart-Saldanha0}.
We recall Theorem 4: consider a smooth locally convex curve 
$\Gamma: (-\epsilon,\epsilon) \to \SO_{n+1}$
with $\Gamma(0) \in \Bru_{\eta\sigma}$.
Let $m_j(t)$ be the determinant of the southwest
$j\times j$ minor of $\Gamma(t)$.
Then $t = 0$ is a zero of multiplicity $\mult_j(\sigma)$ of $m_j$.




\begin{proof}[Proof of Lemma \ref{lemma:triangleleft}:]
Consider $\sigma_0 \vartriangleleft \sigma_1$ and
let $i_\ast$, $j_\ast$ be as in the statement.
For $s \in \RR$ (small), take $M_s: \RR \to \GL_{n+1}$ 
with nonzero entries
\[ (M_s)_{n+2-i,1} = \begin{cases}
t^{(j_1-1)} + s t^{(j_0-1)}, & i = i_0; \\
t^{(i^{\sigma_1}-1)}, & i \ne i_0;
\end{cases} \qquad
(M_s)_{i,j+1} = \frac{d}{dt} (M_s)_{i,j}. \]
This matrix $M_0$ has the form considered in
Lemma 7.2 in \cite{Goulart-Saldanha0}
so its determinant is constant (in $t$), with the sign given by the parity
of $\eta\sigma_1$ ($\eta$ is the top permutation,
the longest element of $S_{n+1}$ in the generators $a_j$).
The matrix $M_s$ is obtained from $M_0$ by row operations so it has
the same determinant.
If necessary, change the sign of the first row of $M_s$
so that $M_s: \RR \to \GL^{+}_{n+1}$.
Perform a QR factorization $M_s(t) = \Gamma_s(t) R_s(t)$
to obtain $\Gamma_s: \RR \to \SO_{n+1}$:
the curve $\Gamma_s$ is locally convex.

Let $(M_s(t))_k$ and $(\Gamma_s(t))_k$ be the $k\times k$
southwest minors of $M_s(t)$ and $\Gamma_s(t)$, respectively.
Consider the smooth functions
\[ \tilde m_{s;k}(t) = \det((M_s(t))_k), \quad
m_{\gamma_s;k}(t) = \det((\Gamma_s(t))_k); \]
the QR factorization shows that one is a positive multiple of the other.
In order to study the multiplicities of zeroes we may as well
use $\tilde m_{s;k}(t)$.

Let
\[ C_{j,k} = \prod_{i_a < i_b \le k} (i_a^{\sigma_j} - i_b^{\sigma_j}), \qquad
\tilde C_k = \frac{C_{1,k}}{C_{0,k}}, \qquad
d_{j,k} = - \frac{k(k-1)}{2} + \sum_{i \le k} i^{\sigma_j}. \]
Notice that for $k < i_0$ and $k \ge i_1$ we have $d_{1,k} = d_{0,k}$
but for $i_0 \le k < i_1$ we have $d_{1,k} = d_{0,k} + (j_1 - j_0)$.
Also, for $k < i_0$ we have $\tilde C_k = 1$;
for $i_0 \le k < i_1$, $\tilde C_k$ has the same sign as $\tilde C_{i_0}$
(this is where we use the condition $\sigma_0 \vartriangleleft \sigma_1$). 

For $k < i_0$
Lemma 7.2 in \cite{Goulart-Saldanha0}
gives
$\tilde m_{s;k}(t) = C_{1,k} t^{d_{1,k}}$.
For $k \ge i_1$, $(M_s(t))_k$ is obtained from $(M_0(t))_k$
by row operations and
Lemma 7.2 in \cite{Goulart-Saldanha0}
again implies
$\tilde m_{s;k}(t) = C_{1,k} t^{d_{1,k}}$.
Finally, for $i_0 \le k < i_1$, we use linearity of the determinant
on row $n+2-i_0$;
more precisely, let $\hat M$ be defined by
\[ (\hat M)_{n+2-i,1} = t^{(i^{\sigma_0})},
\qquad
(\hat M)_{i,j+1} = \frac{d}{dt} (\hat M)_{i,j}, \]
with southwest $k\times k$ minor $(\hat M)_k$
and let $\hat m_{k}(t) = \det((\hat M(t))_k)$.
By linearity and
Lemma 7.2 in \cite{Goulart-Saldanha0}
we have
\begin{align*}
\tilde m_{s;k}(t) &= \det((M_s(t))_k) =
\det((M_0(t))_k) + s \det((\hat M(t))_k) \\
&= \tilde m_{0;k}(t) + s \hat m_{k}(t) =
C_{1,k} t^{d_{1,k}} + C_{0,k} s t^{d_{0,k}} =
C_{0,k} t^{d_{0,k}} ( \tilde C_k t^{(j_1 - j_0)} + s ).
\end{align*}
The multiplicities imply that $\sigma(\gamma_0,0) = \sigma_1$
and $\sigma(\gamma_s,0) = \sigma_0$ for $s \ne 0$
(here $\sigma(\gamma,t) \in S_{n+1}$ denotes the singularity type
of $\gamma$ at time $t$, 
as in the definition of itinerary).
If $j_1-j_0$ is odd there are extra roots with $t \ne 0$;
if $j_1-j_0$ is even and $s$ has the same sign as $\tilde C_{i_0}$
then there are no other real roots.
This completes the proof.
\end{proof}

\begin{proof}[Proof of Proposition \ref{prop:blacktriangle}]
We first consider the case when both words $w_0$ and $w_1$ are empty.
Set $\rho_i = \eta\sigma_i$, $z_1 = \acute\rho_1$
and $d_1 = \inv(\sigma_1) - 1$.
In order to construct a section,
we start with a locally convex curve
$\Gamma: (-\epsilon,\epsilon) \to \Spin_{n+1}$
with $\Gamma(0) = z_1$.
We then translate the curve $\Gamma$ to define a family $\Gamma_s$ 
of curves, $s$ in a small ball around $0 \in \RR^{d_1}$.
From Lemma \ref{lemma:2mult},
if $\sigma_0 \blacktriangleleft \sigma_1$ then
$2 \mult(\sigma_0) \not\le \mult(\sigma_1)$.
From Theorem 4 in \cite{Goulart-Saldanha1},
each curve $\Gamma_s$ intersects the manifold $N$
(as in Lemma \ref{lemma:trianglemanifold}) at most once.
This implies that there exists a curve of values of $s$
through $0$ of curves passing through a (single) point of $N$.
Thus, for a nice ball,
the boundary map contains two points with itinerary including $\sigma_0$:
these are the two points identified by Lemma \ref{lemma:triangleleft}.

The general case (arbitrary $w_0$ and $w_1$)
is handled as in Remark \ref{rem:productcells}.
\end{proof}

\section{Loose and tight maps}
\label{sect:loosetight}

In this section $\ast$ denotes the concatenation of curves:
if $\gamma_0 \in \cL_n(1)$ and $\gamma_1 \in \cL_n(z)$ then
$\gamma_0 \ast \gamma_1 \in \cL_n(z)$ is given by
\[ (\gamma_0 \ast \gamma_1)(t) = 
\begin{cases} \gamma_0(2t), & t \le 1/2; \\
\gamma_1(2t-1), & t \ge 1/2. \end{cases} \]
Clearly, the curve $\gamma_0 \ast \gamma_1$ 
can fail to be smooth at $t = 1/2$:
as discussed in previous occasions (in several papers),
apply a smoothening procedure.
For $\gamma_0 \in \cL_n(1)$ and  $\gamma_1 \in \cL_n$,
the itinerary of $\gamma_0 \ast \gamma_1$ is
$w_{\gamma_0 \ast \gamma_1} = w_{\gamma_0} \eta w_{\gamma_1}$
(where $\eta \in S_{n+1}$ is the top permutation).

There exists a related construction in the CW complex $\cD_n$,
a continuous product $m: \cD_n \times \cD_n \to \cD_n$.
Given two words $w_0, w_1 \in \Word_n$,
we saw in Remark \ref{rem:productcells}
that there exists a bijection from
$c_{w_0} \times c_{w_1}$ to $c_{w_0w_1}$.
The product is defined by taking this bijection
for every pair of cells, so that $m[c_{w_0} \times c_{w_1}] = c_{w_0w_1}$.
Thus, for instance, if is a word $w_0$ of dimension $0$,
the map $m_{w_0}: \cD_n \to \cD_n$
defined by $m_{w_0}(w_1) = m(w_0,w_1)$
is a homeomorphism taking cells to cells
from $\cD_n$ to the open subset $\cD_n[\Ideal_{w_0}]$,
where $\Ideal_{w_0} \subseteq \Word_n$ is the lower set of words $w$
of the form $w = w_0 \tilde w$, $\tilde w \in \Word_n$.
For brevity, for $w_0 \in \Word_n$, $\dim(w_0) = 0$, and $f: K \to \cD_n$
we write $w_0 f: K \to \cD_n$, $(w_0 f)(p) = m(w_0,f(p))$.
Similarly, for $p_0 \in \cD_n$ we write $p_0 f: K \to \cD_n$
for $(p_0 f)(p) = m(p_0,f(p))$.

\smallskip

Let $\gamma_0 \in \cL_n(1)$ be a fixed non-convex closed curve.
Let $M$ be a compact manifold and consider a map $\phi: M \to \cL_n(z)$:
$\phi$ is \emph{loose} if $\phi$ is homotopic to 
\begin{equation}
\label{equation:defloose}
\begin{aligned}
\gamma_0 \ast \phi: M &\to \cL_n(z) \\
s &\mapsto \gamma_0 \ast \phi(s);
\end{aligned}
\end{equation}
otherwise, $\phi$ is \emph{tight}.
These definitions generalize those in \cite{Saldanha3};
there, $n = 2$ and $\gamma_0$ is a fixed curve.
It is easy to see, however, that the definition
does not depend on the choice of $\gamma_0$.
Indeed, if $\gamma_1 \in \cL_n(1)$ is another non-convex closed curve
then $\gamma_0$ and $\gamma_1$ are homotopic
(from Proposition \ref{prop:connected}).
Also, $\gamma_0$ is homotopic to $\gamma_0^2 = \gamma_0 \ast \gamma_0$
and to $\gamma_0^N = \gamma_0 \ast \gamma_0^{N-1} =
\gamma_0 \ast \cdots \ast \gamma_0$.

The following lemma gives a characterization of loose maps $f: K \to \cD_n$.

\begin{lemma}
\label{lemma:aaaaloose}
A map $f_0: K \to \cD_n$ is loose if and only if
$f_0$ is homotopic to the map $f_1 = aaaaf_0$.
\end{lemma}

\begin{proof}
Let $w_0 = \iti(\gamma_0)$.
Take a short convex curve $\gamma_1: [-1,1] \to \Spin_{n+1}$
with $\gamma_1(0) = 1$ so that $\iti(\gamma_1) = \eta$.
Slightly perturb $\gamma_1$ with fixed endpoints
so that its itinerary $w_1$ has dimension $0$.
Thus, for $\gamma \in \cL_n$ with $\iti(\gamma) = w$ we have
$\iti(\gamma_0 \ast \gamma) = w_0 \eta w$
but a small perturbation of $\gamma_0 \ast \gamma$
has itinerary $w_0 w_1 w$.

Let $f_2 = \gamma_0 \ast f_0$.
Perturb each curve $f_2$ near the glueing point
to define a homotopic function $f_3: K \to \cL_n$
with $\iti(f_3(p)) = w_0 w_1 \iti(f_0(p))$.
Take $f_4 = w_0 w_1 f_0: K \to \cD_n \subset \cL_n$.
The functions $f_3$ and $f_4$ are homotopic.
Indeed, up to details in the construction of $\cD_n$
we might have $f_3 = f_4$.
Thus, $f_0$ is loose if and only if $f_0$ and $f_4$ are homotopic.

We claim that $f_1$ and $f_4$ are homotopic,
completing the proof.
Indeed, apply Proposition \ref{prop:connected}
to obtain a path $\delta: [0,1] \to \cD_n$
assuming values in the $1$-skeleton of $\cD_n$
with $\delta(0) = aaaa$, $\delta(1) = w_0w_1$.
Define $H: [0,1] \times K \to \cD_n$,
$H(s,p) = \delta(s) f_0(p) = m(\delta(s),f_0(p))$:
$H$ is the desired homotopy between $f_1$ and $f_4$.
\end{proof}

The following result partially justifies the interest
in considering loose maps.
Recall that $\Omega\Spin_{n+1}(1;z)$
is the space of continuous maps $\Gamma: [0,1] \to \Spin_{n+1}$
with $\Gamma(0) = 1$ and $\Gamma(1) = z$,
so that we have a continuous inclusion map
$i: \cL_n(z) \to \Omega\Spin_{n+1}(1;z)$.

\begin{lemma}
\label{lemma:loose}
Let $M$ be a compact manifold and consider maps
$\phi_0, \phi_1: M \to \cL_n(z)$.
If $\phi_0$ and $\phi_1$ are both loose
and $i \circ \phi_0, i \circ \phi_1$ are homotopic
(in the space $\Omega\Spin_{n+1}(1;z)$)
then $\phi_0$ and $\phi_1$ are homotopic
(in $\cL_n(z)$).
\end{lemma}

\begin{proof}
This result is proved in \cite{Saldanha-Shapiro} and,
for $n = 2$, in \cite{Saldanha3}.
Since this is such a crucial result we present here a brief sketch of proof.

\begin{figure}[ht]
\centerline{\includegraphics[width =8cm]{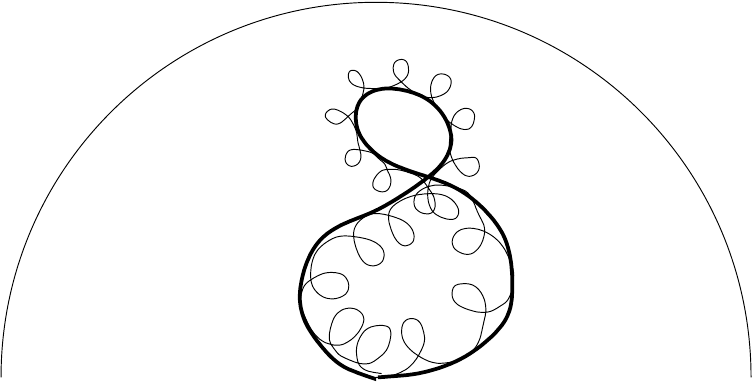}}
\caption{Any base curve becomes locally convex if we add loops.}
\label{fig:addloop}
\end{figure}

If $\phi_0$ is loose then it is homotopic to 
$\gamma_0 \ast \phi_0$ and therefore to
$\gamma_0^N \ast \phi_0$
(for a very large natural number $N$).
The copies of $\gamma_0$ are then spread along $\phi_0$.
At this stage our curve looks like a phone wire
(see Figure \ref{fig:addloop});
the many loops allow us to follow the homotopy in $\Omega\Spin_{n+1}$
and keep our curves locally convex.
\end{proof}

One important construction in \cite{Shapiro-Shapiro} 
and \cite{Saldanha3} is the add-loop procedure, 
which, in certain cases, is used to loosen up 
compact families of nondegenerate curves
through a homotopy in $\cL_n$. 
The resulting families of curly curves are then maleable:
if a homotopy exists in the space of immersions, another
homotopy exists in the space of locally convex curves. 
In \cite{Saldanha3}, for instance, open dense subsets 
$\mathcal{Y}_{\pm}\subset\cL_2(\pm 1)$ are shown to be 
homotopy equivalent to the space of loops $\Omega\Ss^3$.
Thus, certain questions regarding homotopies can be 
transferred to the space of continuous paths in the 
group $\Spin_{n+1}$.
This approach is reminiscent of classical constructions such as 
the proof of Hirsch-Smale Theorem \cite{Hirsch, Smale} and 
Thurston's eversion of the sphere by corrugations 
\cite{Levy-Maxwell-Munzner}.
It can be considered as an elementary instance of the h-principle 
of Eliashberg and Gromov \cite{Eliashberg-Mishachev, Gromov}.

We now restate Lemma 6.6 in \cite{Saldanha3}, with minor changes.
This is in a sense a more general version of Lemma \ref{lemma:loose} above
(with a more complicated statement).
Compared to the situation in \cite{Saldanha3}, 
we are now in arbitrary dimension.
Also, the concept of ``adding two loops'' has to be adapted.
The proof, however, is very much the same;
we give a more succint version here,
the proof given in \cite{Saldanha3} is more careful.

If $\gamma \in \cL_n(1)$ and $z \in \Spin_{n+1}$,
then $z\gamma$ denotes the closed curve
$z\gamma: [0,1] \to \Spin_{n+1}$,
$(z\gamma)(t) = z(\gamma(t))$.

\goodbreak

\begin{lemma}
\label{lemma:66bis}
Let $K$ be a compact manifold
and $f: K \to \cL_n$ a continuous map.
Let $J$ be a positive integer and let $(\gamma_j)_{1 \le j \le J}$
be a family of non-convex closed curves $\gamma_j \in \cL_n(1)$.
Assume that:
\begin{enumerate}
\item{
$t_{1,-}, t_{1,+}, t_{2,-}, \ldots, t_{J,+}: K \to (0,1)$
are continuous functions with
$0 < t_{1,-} < t_{1,+} < t_{2,-} < \cdots < t_{J,+} < 1$; }
\item{ the family of open sets $(U_j)_{j \le J}$ covers $K$,
that is, $K = \bigcup_{1 \le j \le J} U_j$;}
\item{ for $p \in U_j$, ${f(p)}(t_{j,-}(p)) = {f(p)}(t_{j,+}(p))$
and the restriction $f(p)|_{[t_{j,-}(p),t_{j,+}(p)]}$ is a reparametrization
in time of the closed curve $({f(p)}(t_{j,-}(p))) \gamma_j$.}
\end{enumerate}
Then the map $f$ is loose.
\end{lemma}

From Proposition \ref{prop:connected},
each curve $\gamma_j$ is homotopic to $\gamma_0$.
Notice that in Lemma~6.6 in \cite{Saldanha3}
the curves $\gamma_j$, $1 \le j \le J$, and $\gamma_0$
are all equal to $\nu_2$, a circle traversed twice, and
the closed curve $f(p)|_{[t_{j,-}(p),t_{j,+}(p)]}$ 
is a pair of loops.

\begin{proof}[Proof of Lemma \ref{lemma:66bis}]
Recall that $\gamma_j \sim \gamma_0^N \ast \gamma_j$, 
that is, there exists a path in $\cL_n(1)$ from
$\gamma_j$ to $\gamma_0^N \ast \gamma_j$
(here $N$ is a large natural number, to be chosen later).
Apply this path near the boundary of each 
open set $U_j$ so that in a slightly smaller open set
$\tilde U_j$ the closed curve
$f(p)|_{[t_{j,-}(p),t_{j,+}(p)]}$
is now (a reparametrization of)
$({f(p)}(t_{j,-}(p))) (\gamma_0^N \ast \gamma_j)$
We may assume that the open sets $\tilde U_j$
still cover $K$.
For each $p \in K$,
spread the copies of $\gamma_0$
which appear along $f(p)$ in time,
so that now we have ``phone wire'' curves as in 
Figure \ref{fig:addloop}.
As in the proof of Lemma \ref{lemma:loose},
apply the homotopy which is known to exist
in the space $\Omega\Spin_{n+1}$ to the base curves.
For sufficiently large $N$, the ``phone wire''
structure is sufficient to give us local convexity,
thus obtaining the desired homotopy in $\cL_n$.
\end{proof}

\section{Loose ideals}
\label{sect:looseideals}

A non empty lower set $\Ideal \subset \Word_n$
for both $\sqsubseteq$ and $\preceq$ is an \textit{ideal}
if $w_1 \in \Ideal$ and $w_2 \in \Word_n$ imply
$w_1w_2, w_2w_1 \in \Ideal$.
The lower sets $\Ideal_{[0]}$, $\Ideal_{Y_2}$ and $\Ideal_{Y}$
(discussed in
Examples \ref{example:I0}, \ref{example:IY2} and \ref{example:IY})
are ideals.
A lower set $\Ideal$ is a \textit{weakly loose ideal} if
it is an ideal and
every map $f: K \to \cL_n[\Ideal] \subset \cL_n$ is loose.
Notice that we do not claim that the homotopy
between $f$ and $\gamma_0 \ast f$
in Equation \eqref{equation:defloose} has image contained in $\cL_n[\Ideal]$.
A weakly loose ideal $\Ideal$ is \textit{(strongly) loose} if
for every map $f: K \to \cL_n[\Ideal] \subset \cL_n$
there exists a homotopy $H: [0,1] \times K \to \cL_n[\Ideal]$
between $f$ and $\gamma_0 \ast f$.


\smallskip

A \textit{nice loop} is a periodic locally convex but non convex curve 
$\gamma: [0,1] \to \Spin_{n+1}$ with $\gamma(0) = \gamma(1) \in \Bru_{\eta}$
and image contained in $\Bru_{\eta} \cup \bigcup_i \Bru_{\eta a_i}$.
For arbitrary $n$ there exist nice loops:
take an arbitrary periodic locally convex curve $\tilde\gamma \in \cL_n(1)$
and take $\gamma = z\tilde\gamma$ for generic $z \in \Spin_{n+1}$.
The itinerary $\iti(\gamma) \in \Word_n$ of a nice loop $\gamma$
is defined similarly to itineraries of curves in $\cL_n$.
More precisely, $\gamma^{-1}[\bigcup_i \Bru_{\eta a_i}] \subset (0,1)$
is a finite set $\{t_1 < \cdots < t_\ell\}$.
For $j \le \ell$, take $i_{j}$ such that $\gamma(t_j) \in \Bru_{\eta a_{i_j}}$
and set $\iti(\gamma) = a_{i_1}\cdots a_{i_\ell} \in \Word_n$.

A word $w \in \Word_n$ is a \textit{loop word} if there exists
a nice loop $\gamma$ with $\iti(\gamma) = w$.
Notice that if $w$ is a loop word then $\dim(w) = 0$ and $\hat w = 1$.
For $n = 2$, $aaaa$ is a loop word.
Investigating exactly which words are loop words
is an interesting question which we shall not pursue.

\begin{lemma}
\label{lemma:loopword}
Let $w \in \Word_n$ be a loop word.
There exist $\gamma_\star \in \cL_n(\hat\eta)$
with the following properties:
\begin{enumerate}
\item{The itinerary of $\gamma_\star$ is $w = \iti(\gamma_\star)$.}
\item{The arcs $\gamma_\star|_{[0,\frac13]}$
and $\gamma_\star|_{[\frac23,1]}$ are convex.}
\item{We have $\gamma_\star(\frac13) = \gamma_\star(\frac23) = \acute\eta$.}
\end{enumerate}
\end{lemma}

\begin{proof}
Take a nice loop $\tilde\gamma: [\frac13, \frac23] \to \Spin_{n+1}$.
Multiplication by an element of $\Quat_{n+1}$
allows us to assume that $\tilde\gamma(\frac13) \in \Bru_{\acute\eta}$.
By applying a projective transformation (see Remark \ref{rem:projective}),
we may also assume that $\tilde\gamma(\frac13) = \acute\eta$.
Extend the curve $\tilde\gamma$ by two convex arcs
(from $1$ to $\acute\eta$ and from $\acute\eta$ to $\hat\eta$)
to define the desired curve $\gamma_\star$.
\end{proof}

\begin{rem}
\label{rem:projective}
For $\gamma_\star$ as in Lemma \ref{lemma:loopword},
it is the restriction $\tilde\gamma = \gamma_\star|_{[\frac13,\frac23]}$
which is a closed curve, not $\gamma_\star$
(unless $n$ is such that $\hat\eta = 1$).

In the above proof we used projective transformations,
which are discussed in other papers,
particularly in Section 2.1 of \cite{Goulart-Saldanha1}.
We try to avoid introducing extra notation for such transformations
and will use them in other proofs without further comments.
\end{rem}


\begin{prop}
\label{prop:ideal0isloose}
The set $\Ideal_{[0]} \subset \Word_n$ of words 
containing at least one letter of dimension $0$
is a weakly loose ideal.
In particular, if $K$ is a compact manifold
and $f: K \to \cD_n[\Ideal_{[0]}] \subset \cL_n$
is a continuous map then $f$ is loose.
\end{prop}

\begin{rem}
\label{rem:ideal0}
Take $K = \Ss^0 = \{+1, -1\}$ and
$f_0, f_1: K \to \cD_n[\Ideal_{[0]}]$,
$f_1 = aaaa f_0$ so that
\[ f_0(+1) = {a}, \quad f_0(-1) = f_1(+1) = {aaaaa}, \quad
f_1(-1) = {aaaaaaaaa}. \]
Proposition \ref{prop:connected} implies that
$f_0$ and $f_1$ are homotopic in $\cD_n$,
consistently with Proposition \ref{prop:ideal0isloose}
which says that $f_0$ is loose and Lemma \ref{lemma:aaaaloose}
which says that $f_0$ and $f_1$ are therefore homotopic in $\cD_n$.
On the other hand, $f_0$ and $f_1$ are not homotopic in $\cD_n[\Ideal_{[0]}]$
(or in $\cL_n[\Ideal_{[0]}]$):
the cell $c_{a}$ is an isolated vertex in $\cD_n[\Ideal_{[0]}]$.
This proves that $\Ideal_{[0]}$ is not (strongly) loose.
\end{rem}

\begin{proof}[Proof of Proposition \ref{prop:ideal0isloose}]
The set $\Ideal_{[0]}$ is clearly a lower set
for both $\preceq$ and $\sqsubseteq$; it is also an ideal.
All we have to prove is therefore the last claim in the statement:
that the map $f$ is loose.
The idea of the proof is to start with $f_0 = f$ as in the statement
and deform it (i.e., apply a homotopy) 
first in $\cD_n$ and then in $\cL_n$
to obtain a finite sequence of functions $f_i: K \to \cL_n$.
The function $f_3$ satisfies the hypothesis
of Lemma \ref{lemma:66bis} and is therefore loose:
thus, so is the original $f$.
Let $\gamma_\bullet \in \cL_n(1)$ be a periodic non convex curve
(as $\gamma_0$ in Equation \eqref{equation:defloose}).
Let $w_0$ be a loop word,
the itinerary of the nice loop $\gamma_0 = z \gamma_\bullet$,
where $z \in \Spin_{n+1}$ is a generic element.

We first describe the strategy of the proof.
The function $f_1: K \to \cD_n$ has the property
that for every $p \in K$ if $f_1(p) \in c_w$
then $w$ contains a subword $w_0^N$
(where $N$ is a large positive integer).
We then prove that the function $f_1$ satisfies conditions
similar (but not identical) to those in Lemma \ref{lemma:66bis}.
More precisely, there exists a finite open cover 
$(U_j)_{1 \le j \le J}$ of $K$ and continuous functions
$t_{1,-2} < t_{1,+2} < \cdots < t_{J,-2} < t_{J,+2}: K \to (0,1)$
such that if $p \in U_j$ then the arc
$f_1(p)|_{[t_{j,-2}(p), t_{j,+2}(p)]}$ has itinerary $w_0$.
From $f_1$ to $f_2$ the homotopy takes place in $\cL_n$.
The function $f_2: K \to \cL_n$ satisfies yet another variation
of the conditions in Lemma \ref{lemma:66bis}.
More precisely, there exist continuous functions
$t_{j,-}, t_{j,+}: U_j \to (0,1)$ with
$t_{j,-2} < t_{j,-} < t_{j,+} < t_{j,+2}$
such that $p \in U_j$ implies that
$f(p)(t_{j,-}(p)) = f(p)(t_{j,+}(p))$ and the arc
$f(p)|_{[t_{j,-}(p), t_{j,+}(p)]}$
is obtained from $\gamma_0$ by a projective transformation.
Finally, a homotopy takes $f_2$ to $f_3: K \to \cL_n$
satisfying the original conditions:
the arcs $f(p)|_{[t_{j,-}(p), t_{j,+}(p)]}$
are obtained from $\gamma_0$ by multiplication.

\smallskip

For a function $f_0: K \to \cL_n$
we have a finite open cover $(U_j)_{1 \le j \le J}$
of $K$ such that each $U_j$ is connected and
for $p \in U_j$,
the itinerary of $f_0(p)$
is of the form $w_{-}(p) a_{i_j} w_{+}(p)$.
Similarly, a function $f_0: K \to \cD_n$
can be slightly modified by a homotopy so that
we have a finite open cover $(U_j)_{1 \le j \le J}$
of $K$ such that each $U_j$ is connected and
$p \in U_j$ implies that $f(p) \in c_{w(p)}$
where the word $w(p)$
is of the form $w_{-}(p) a_{i_j} w_{+}(p)$.
Moreover, we can take the cover so that there are continuous
functions $f_{\pm,j}: U_j \to \cD_n$ such that, for all $p \in U_j$,
$f(p) = f_{-,j}(p) a_{i_j} f_{+,j}(p)$.
Take a family of compact manifolds with boundary $W_j \subset U_j$
such that the sets $U_j^{(0)} = \interior(W_j)$ cover $K$.
Take tubular neighborhoods of $\partial W_j$, i.e.,
local homeomorphisms $\Phi_j: \partial W_j \times [-1,1] \to U_j$
such that $\Phi_j(p,0) = p$ and $\Phi_j(p,t) \in \interior(W_j)$
for all $p \in \partial W_j$ and $t > 0$.
For $s \in [0,1]$, set
$U_j^{(s)} = \interior(W_j) \smallsetminus \Phi_j[\partial W_j \times [-1,s]]$
so that $s_0 < s_1$ implies $U_j \supset U_j^{(s_0)} \supset U_j^{(s_1)}$.
We also assume that the open sets $(U_j^{(1)})_{1 \le j \le J}$
cover $K$.

From Proposition \ref{prop:connected}, $a_{i} \sim a_{i}w_0^N$,
where $N$ is a large positive integer to be specified later.
In other words, there exists a path $\delta_{i}$
from $[0,1]$ to the $1$-skeleton of $\cD_n$ satisfying:
\begin{equation}
\label{equation:delta}
\delta_{i}: [0,1] \to \cD_n, \qquad
\delta_{i}(0) \in c_{a_{i}} \subset \cD_n, \quad
\delta_{i}(1) \in c_{a_{i}w_0^N}.
\end{equation}
The paths $a_i^H$ constructed in Example \ref{example:earlyaiH}
can be used here.
Let $a'_i$ be, as in Section \ref{sect:permutations},
a letter of dimension $1$ (with abuse, a segment) joining $a_i$
to a word $a_i^\star$ of dimension $0$ and length $3$:
\[
a' = [ba], a^\star = bab, b' = [ab], b^\star = aba, \ldots, 
a'_{i+1} = [a_ia_{i+1}], a_{i+1}^\star = a_ia_{i+1}a_i, \ldots \]
We may assume that 
$\delta_i(t) \in c_{a'_i}$ for $t \in (0,\frac12)$,
$\delta_i(\frac12) \in c_{a^\star_i}$,
$\delta_i(\frac34) \in c_{aaaaa_i}$ and
$\delta_i(1) \in c_{a_{i}w_0^N}$.
Notice that for $t \ge \frac12$ if $\delta_i(t) \in c_w$
then $w$ has a letter of dimension~$0$.

Define a map $f_1: K \to \cD_n$ as follows.
For $p \in U_j^{(1/2)}$,
$f_1(p)$ is obtained from $f_0(p)$ by replacing
$a_{i_j}$ by $a_{i_j}w_0^N$.
More precisely: take $p$ in $K$ and
$w \in \Ideal_{[0]}$ such that $f(p) \in c_w$.
Let $J_p$ be the set of $j$ for which $p \in U_j^{(0)}$.
To each $j$ in $J_p$ corresponds a copy of the letter $a_{i_j}$ in $w$.
Notice that such copies need not be distinct in the word $w$:
for $j_0 \ne j_1 \in J_p$, we may have $i_{j_0} = i_{j_1}$ and
the open sets $U_{j_0}$ and $U_{j_1}$ may both point
to the exact same copy of $a_{i_{j_0}} = a_{i_{j_1}}$ in $w$.
If this happens, we say that $j_0$ and $j_1$ are \textit{equivalent}.
Write 
$w = w_1a_{i_{j_1}}w_2\ldots w_ka_{i_{j_k}}w_{k+1}$
where $j_1, \ldots, j_k$ are representatives
of the equivalent classes in $J_p$.

For each $j \in J_p$, define
$\tilde s_j = 2 \sup \{ s \in [0,1/2] \;|\; p \in U_j^{(s)} \} \in [0,1]$;
notice that for any $p \in K$ there exists $j \in J_p$
such that $\tilde s_j = 1$.
For a representative $j_{\kappa}$ ($1 \le \kappa \le k$),
define $s_{j_{\kappa}}$ to be the maximum of all values
of $\tilde s_j$, $j$ equivalent to $j_{\kappa}$.
Write
$c_w = c_{w_1} \times {a_{i_{j_1}}} \times \cdots
\times {a_{i_{j_k}}} \times c_{w_{k+1}}$
and
\begin{equation}
\label{equation:f0}
f_0(p) = (p_1,a_{i_{j_1}}, \ldots, a_{i_{j_k}}, p_{k+1})
\end{equation}
(so that, for instance, $p_1 \in c_{w_1}$).
(Notice that, as we often do, we use a word of dimension $0$
to denote both a point in $\cD_n$ and the cell whose only
element is said point.)
Define
\begin{equation}
\label{equation:f1}
f_1(p) = 
(p_1,\delta_{i_{j_1}}(s_{j_1}), \ldots,
\delta_{i_{j_k}}(s_{j_k}), p_{k+1}). 
\end{equation}
Similarly, define a homotopy $H_1: K \times [0,1] \to \cD_n$
from $f_0$ to $f_1$ by
\begin{equation}
\label{equation:H1}
H_1(p,s) = 
(p_1,\delta_{i_{j_1}}(\min\{s,s_{j_1}\}), \ldots,
\delta_{i_{j_k}}(\min\{s,s_{j_k}\}), p_{k+1}). 
\end{equation}
This completes the construction of $f_1$.
Notice that $f_1$ assumes values in $\cD_n[\Ideal_{[0]}] \subset \cD_n$
and the homotopy $H_1$ assumes values in $\cD_n$:
nothing guarantees that the image of $H_1$
is contained in $\cD_n[\Ideal_{[0]}]$
(see Remark \ref{rem:ideal0}).
It is not clear at this point that $f_1$ together with the cover $(U_j)$
satisfies the property stated above, 
that is, there are at this point no functions $t_{j,\pm 2}$.
We adress this issue by modifying the cover (but not the function $f_1$).

A first difficulty is that in the intersection $U_{j_1} \cap U_{j_2}$
(with $j_1 \ne j_2$) the same copy of $w_0^N$ 
should not be used by the two open sets.
This is easily addressed by taking $N$ sufficiently large.
More precisely, so that each set $U_j$ has a copy of $w_0^N$
we perform the previous construction not with $N$ but with $\tilde N > JN$.
The set $U_j$ then has the $j$-th copy of $w_0^N$ inside 
the copy of $w_0^{\tilde N}$.
From now on we assume that this first difficulty
has been taken care of.

A second difficulty is that copies of $w_0^N$ should appear
in the itinerary of $f_1(p)$ in the same order as the indices $j$.
Permuting the indices $j$ may not solve the problem,
but taking $N$ large, followed by a refinement of the covering, does.
Indeed, assume $N > J$.
Define continuous functions
$\tilde t_{j,-}, \tilde t_{j,+}: U_j \to [0,1]$ such that,
for $p \in U_j$, $\tilde t_{j,-}(p) < \tilde t_{j,+}(p)$
and the itinerary of the arc
$f_1(p)|_{[\tilde t_{j,-}(p), \tilde t_{j,+}(p)]}$ is $w_0^N$.
Set $\tilde t_j = (\tilde t_{j,-} + \tilde t_{j,+})/2: U_j \to [0,1]$.
The functions $\tilde t_j$ can be defined so as to be extendable to $K$;
furthermore, we may assume that for any $p$
there exists at most one pair $(j_0,j_1)$ with $j_0 < j_1$
and $\tilde t_{j_0}(p) = \tilde t_{j_1}(p)$.
Let
$\tau_j(p) = \card\{j' \;|\; \tilde t_{j'}(p) \le \tilde t_j(p) \} \in [1,J]$.
Let $U_{j,j'} \subseteq U_j$ be an open set such that
$p \in U_j$ and $\tau_j(p) = j'$ imply $p \in U_{j,j'}$.
The sets can be chosen so that if $j_0 \ne j_1$ then
$U_{j_0,j'}$ and $U_{j_1,j'}$ are disjoint.
Assign to $U_{j,j'}$ the $j'$-th copy of $w_0$ in 
$w_0^N = \iti(f_1(p)|_{[\tilde t_{j,-}(p), \tilde t_{j,+}(p)]})$.
More precisely, define functions
$\tilde t_{j,j',-}, \tilde t_{j,j',+}: U_{j,j'} \to [0,1]$
with $t_{j,-} < t_{j,j',-} < t_{j,j',+} < t_{j,+}$
and such that
$\iti(f_1(p)|_{[\tilde t_{j,-}(p), \tilde t_{j,j',-}(p)]}) = w_0^{j'-1}$,
$\iti(f_1(p)|_{[\tilde t_{j,j',-}(p), \tilde t_{j,j',+}(p)]}) = w_0$ and
$\iti(f_1(p)|_{[\tilde t_{j,j',+}(p), \tilde t_{j,+}(p)]}) = w_0^{N-j'}$.
The functions can be chosen so that
$p \in U_{j_0,j_0'} \cap U_{j_1,j_1'}$ and $j_0' < j_1'$ imply
$t_{j_0,j_0',+}(p) < t_{j_1,j_1',-}(p)$.
Relabel the non empty sets $U_{j,j'}$ in increasing order of $j'$,
allowing us to define the functions $t_{j,\pm 2}$.
This completes the discussion of $f_1$.

\smallskip

We now discuss the construction of $f_2$
and of the homotopy from $f_1$ to $f_2$.
For $p \in U_j$, 
the arc $f_1(p)|_{[t_{j,-2}(p),t_{j,+2}(p)]}$ has itinerary $w_0$.
For $f_2$, we want that arc to contain a closed curve,
a copy of $\gamma_0$ (up to projective transformation).
We show how to achieve this by a homotopy,
working for one value of $j$ at a time.
It will be noticed that the construction for $j = j_2$
does not spoil the previously performed construction for $j = j_1 < j_2$.
For this, take families of subsets
$U_{j,0} \subset W_{j,0} \subset U_{j,1} \subset W_{j,1} \subset U_j$
such that $U_{j,\ast}$ are open sets,
$W_{j,\ast}$ are compact manifolds with boundary
and the family $(U_{j,0})_{1 \le j \le J}$ covers $K$.
As usual, after this step we drop the extra index
and write $U_j$ instead of $U_{j,0}$.

Given $j$, assume $f_{1+\frac{j-1}{J}}$ constructed with the property
$f_{1+\frac{j-1}{J}}(p)(t_{j',-1}(p)) = f_{1+\frac{j-1}{J}}(p)(t_{j',+1}(p))$
for $p \in U_{j',0}$ and $j' < j$.
We want to construct a homotopic function $f_{1+\frac{j}{J}}$
with the same property for all $j' \le j$.
Let $q_j \in \Quat_{n+1}$ such that
$f_1(p)(t_{j,\pm 2}(p)) \in \Bru_{q_j \acute\eta}$.
Let $t_{j,\pm \frac12}(p)$ be the times corresponding
to the first and last letter in the arc
$f_1(p)|_{[t_{j,-2}(p), t_{j,+2}(p)]}$;
notice that, from Theorem~1 
in \cite{Goulart-Saldanha1},
these are continuously defined for $p \in U_j$.
Define functions
$t_{j,\pm 1}(p) = (2t_{j,\pm \frac12}(p)+t_{j,\pm 2}(p))/3$
and
$t_{j,\pm \frac32}(p) = (t_{j,\pm \frac12}(p)+2t_{j,\pm 2}(p))/3$
so that $t_{j,-2} < t_{j,-\frac32} < t_{j,-1} < t_{j,-\frac12} <
t_{j,+\frac12} < t_{j,+1} < t_{j,+\frac32} < t_{j,+2}$.
Draw a family of convex arcs from $q_j$ to $f_1(p)(t_{j,-\frac32}(p))$
(with domain $[0,t_{j,-\frac32}(p)]$)
and from $f_1(p)(t_{j,+\frac32}(p))$ to $q_j \hat\eta$.
(with domain $[t_{j,+\frac32}(p),1]$).
Multiplication by $q_j^{-1}$ defines a continuous map
$g_{j,0}: W_{j,1} \to \cL_n[w_0] \subset \cL_n(\hat\eta)$
with $q_j g_{j,0}(p)(t) = f_1(p)(t)$.
Since $\cL_n[w_0]$ is a contractible Hilbert manifold,
there exists a homotopy from $g_{j,0}$ to $g_{j,1}: W_{j,1} \to \cL_n[w_0]$
with the following properties:
\begin{enumerate}
\item{The restrictions $g_{j,0}|_{\partial W_{j,1}}$ and
$g_{j,1}|_{\partial W_{j,1}}$ are equal.}
\item{Up to reparametrization,
the restriction $g_{j,1}|_{W_{j,0}}$ is constant
equal to $\gamma_\star$ (the curve introduced in Lemma \ref{lemma:loopword}).
The reparamatrization takes $t_{j,-1}(p)$ and $t_{j,+1}(p)$ to 
$1/3$ and $2/3$ so that
$g_{j,1}(p)(t_{j,-1}(p)) = g_{j,1}(p)(t_{j,+1}(p)) = \acute\eta$.}
\item{For all $p \in W_{j,1}$ and all $s \in [0,1]$,
we have
$g_{j,s}(p)(t_{j,-2}(p)) \in \Bru_{\acute\eta}$
and
$g_{j,s}(p)(t_{j,+2}(p)) \in \Bru_{\acute\eta}$.}
\end{enumerate}
Here of course $g_{j,s}(p) = H(s,p)$,
where $H$ is the homotopy from $g_{j,0}$ to $g_{j,1}$.
The careful reader will have noticed that
we did not pay attention to differentiability
of curves in $\cL_n$ at glueing points.
Curves can be smoothened at such points,
as has been amply discussed in several papers.

The maps $g_{j,\ast}$ and the homotopy above
guide us to contruct $f_{1+\frac{j}{J}}$
and the homotopy with $f_{1+\frac{j-1}{J}}$.
First, consider the arcs obtained by restricting
$f_{1+\frac{j-1}{J}}(p)$ to
$[0,t_{j,-\frac32}(p)]$.
The corresponding arc in $f_{1+\frac{j-1+s}{J}}(p)$
is obtained by a projective transformation
taking
$f_{1+\frac{j-1}{J}}(p)(t_{j,-\frac32}(p)) \in \Bru_{q_j \acute\eta}$ to
$q_j g_{j,s}(p)(t_{j,-\frac32}(p)) \in \Bru_{q_j \acute\eta}$.
Consider now the restriction of 
$f_{1+\frac{j-1}{J}}(p)$ to
$[t_{j,+\frac32}(p),1]$:
the corresponding arc in $f_{1+\frac{j-1+s}{J}}(p)$
is obtained by a projective transformation
taking $f_{1+\frac{j-1}{J}}(p)(t_{j,+\frac32}(p)) \in \Bru_{q_j \acute\eta}$ to
$q_j g_{j,s}(p)(t_{j,+\frac32}(p)) \in \Bru_{q_j \acute\eta}$.
Finally, the restriction of $f_{1+\frac{j-1+s}{J}}(p)$
to $[t_{j,-\frac32}(p),t_{j,+\frac32}(p)]$
is $q_j g_{j,s}(p)$.
Again, smoothening procedures are implicit;
notice that they have no effect on itineraries
or on the condition $f_2(p)(t_{j,-1}(p)) = f_2(p)(t_{j,+1}(p))$.
This completes step $j$ in the construction of $f_2$
and the discussion of $f_2$.

Notice that the closed loop $f_2(p)|_{[t_{j,-1}(p),t_{j,+1}(p)]}$
is obtained from $\gamma_0$ at step $j$ in the construction of $f_2$,
but suffers projective transformations at later steps.
The group of projective transformations is contractible, however.
In order to pass from $f_2$ to $f_3$
we apply multiplication and projective transformations
to such closed loops,
completing the construction of $f_3$ and the proof.
\end{proof}




\section{Simple connectivity}
\label{sect:simplyconnected}

The aim of this section is to prove Theorem \ref{theo:simplyconnected}.
This is essentially equivalent to the following result.

\begin{prop}
\label{prop:simplyconnected}
Any continuous map $f_0: \Ss^1 \to \cD_n$ is loose.
\end{prop}

\begin{proof}
We may assume that $f_0$ assumes values in the $1$-complex of $\cD_n$.
From Lemma \ref{lemma:aaaaloose}
it suffices to prove that $f_0$ is homotopic to $f_1 = aaaa f_0$.
We construct an explicit homotopy between the maps $f_0$ and $f_1$.

Recall that in Example \ref{example:earlyaiH}
(consistently with Proposition \ref{prop:connected})
we constructed for every non-empty word $w \in \Word_n$
of dimension $0$ an explicit path $w^H$
in the $1$-skeleton of $\cD_n$ joining $w$ and $aaaaw$.
We first define:
\[ a^{H} = (a \xleftarrow{[ba]} bab
\xleftarrow{[ab]ab} \cdot \xrightarrow{aba[ba]b} \cdot \xrightarrow{ababab[ab]}
ababababa
\xleftarrow{a[ba]ababa} \cdot \xleftarrow{aaa[ba]a} aaaaa). \]
Define recursively
\[ a_{k+1}^{H} = (a_{k+1}\xleftarrow{[a_ka_{k+1}]} a_ka_{k+1}a_k
\xrightarrow{a_k^{H}a_{k+1}a_k} aaaaa_ka_{k+1}a_k
\xrightarrow{aaaa[a_ka_{k+1}]} aaaaa_{k+1}) \]
and $(a_kw)^{H} = a_k^{H}w$.
We now define $w^H$ for words of dimension $1$.


Let $w$ be a word of dimension $1$ with $\partial w = w_1 - w_0$,
where $w_0$ and $w_1$ are non-empty words of dimension $0$.
In order to complete the construction of the homotopy,
it suffices to construct $w^H: [0,1]^2 \to \cD_n$ 
with
$w^H|_{\{0\}\times [0,1]} = w$,
$w^H|_{\{1\}\times [0,1]} = aaaaw$
$w^H|_{[0,1] \times \{0\}} = w_0^H$,
$w^H|_{[0,1] \times \{1\}} = w_1^H$.
If the first letter of $w$ has dimension $0$ 
we may write $w = a_k \tilde w$ and $w^H = a_k^H \tilde w$.
We are left with the case when the first letter of $w$ has dimension $1$:
if $w = \sigma \tilde w$ we set $w^H = \sigma^H \tilde w$.
We are therefore left with the case $w = \sigma$, $\inv(\sigma) = 2$.

This is done on a case by case basis:
$[a_{k}a_{k+1}]$, $[a_{k+1}a_k]$ and $[a_ka_l]$ for $l > k+1$.

We first observe that the case $[a_{k}a_{k+1}]$ is rather trivial:
the square collapses to the segment in the definition of $a_{k+1}^{H}$.
We define $[ac]^{H}$ in Figure \ref{fig:acH}.

\begin{figure}[ht]
\def\svgwidth{12cm}
\centerline{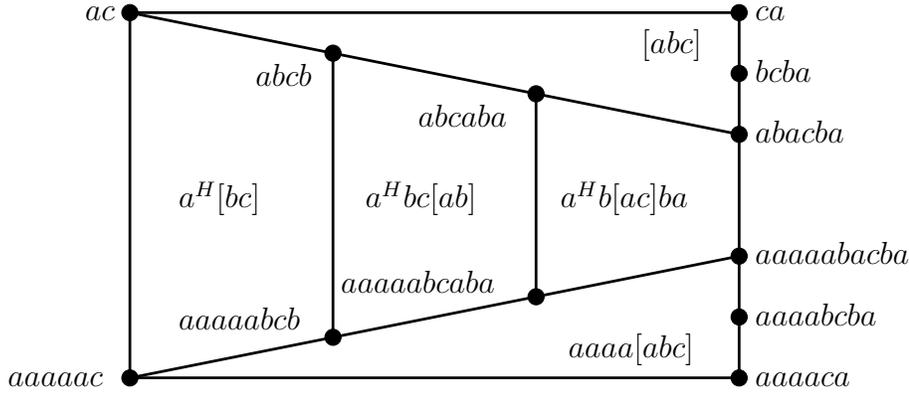}
\caption{The definition of $[ac]^{H}$.}
\label{fig:acH}
\end{figure}

The top and bottom hexagons
(they look like triangles, but they have six vertices each)
are $[abc]$ and $aaaa[abc]$, as indicated.
Each of the three central trapezoidal regions
is actually composed of six product cells
(since $a^{H}$ is not one edge, but six).
The definition of $[a_{k}a_{k+2}]^{H}$ is similar:
just substitute $a_k$, $a_{k+1}$ and $a_{k+2}$ for $a$, $b$ and $c$
in Figure \ref{fig:acH}
(a minor difference is that $a_{k}^{H}$ is actually $4+2k$ edges).
We define $[ad]^{H}$ in Figure \ref{fig:adH}.

\begin{figure}[ht]
\def\svgwidth{12cm}
\centerline{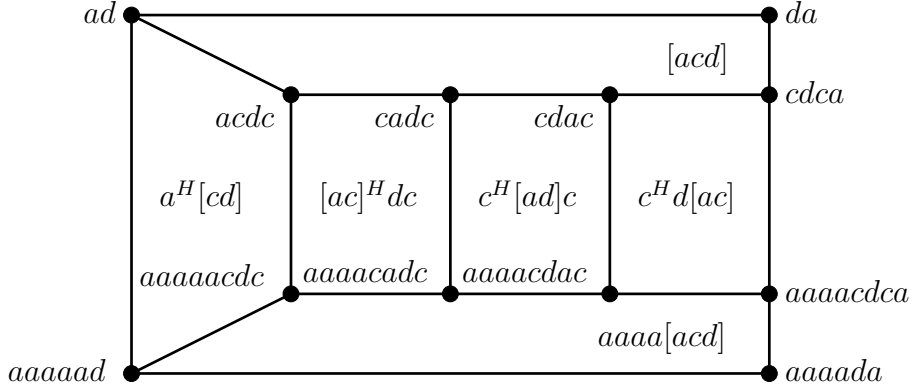}
\caption{The definition of $[ad]^{H}$.}
\label{fig:adH}
\end{figure}

Notice that the second of the four central trapezoidal regions
uses the previous construction of $[ac]^{H}$.
The other three consist of several product cells.
Again, the definition of $[a_{k}a_{k+3}]^{H}$ is similar.
The definition of $[a_{k}a_{k+4}]^{H}$ uses that of 
$[a_{k}a_{k+3}]^{H}$
and so on, recursively.

This takes care of the cases $[a_ka_l]$ for $l > k$;
we are left with the case $[a_{k+1}a_k]$.
We now describe $[cb]^{H}$ in Figure \ref{fig:cbH}.

\begin{figure}[ht]
\def\svgwidth{12cm}
\centerline{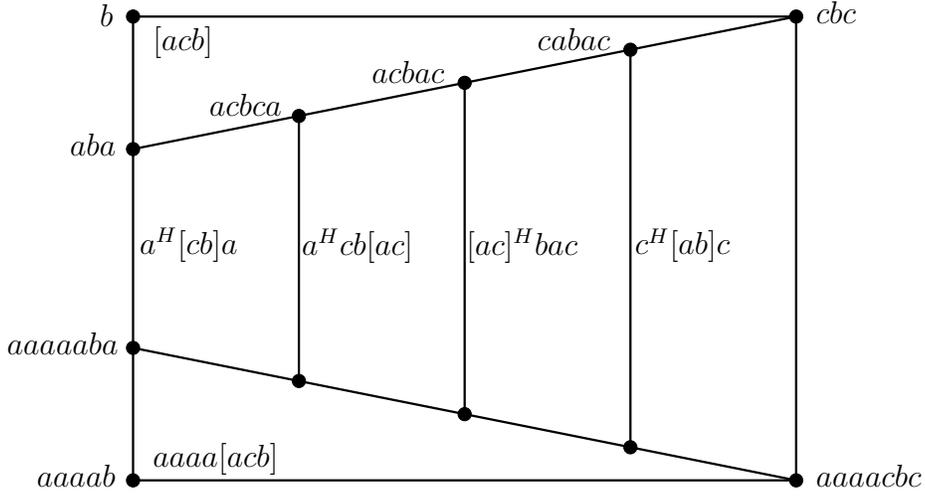}
\caption{The definition of $[cb]^{H}$.}
\label{fig:cbH}
\end{figure}

This definition uses a specific cell for $[acb]$;
a similar diagram works for the other cell.
The definition of $[a_{k+1}a_{k}]^{H}$ for $k > 2$ is similar,
just substitute
$a_{k-1}$, $a_{k}$ and $a_{k+1}$ for $a$, $b$ and $c$
(except for the initial $aaaa$ in the lower part of the diagram).

Finally, in order to define $[ba]^{H}$
we have to fill in the trapezoidal region,
indicated by (?) in Figure \ref{fig:baH};
notice that there is a partial collapse at the top.

\begin{figure}[ht]
\centerline{\includegraphics[width =8cm]{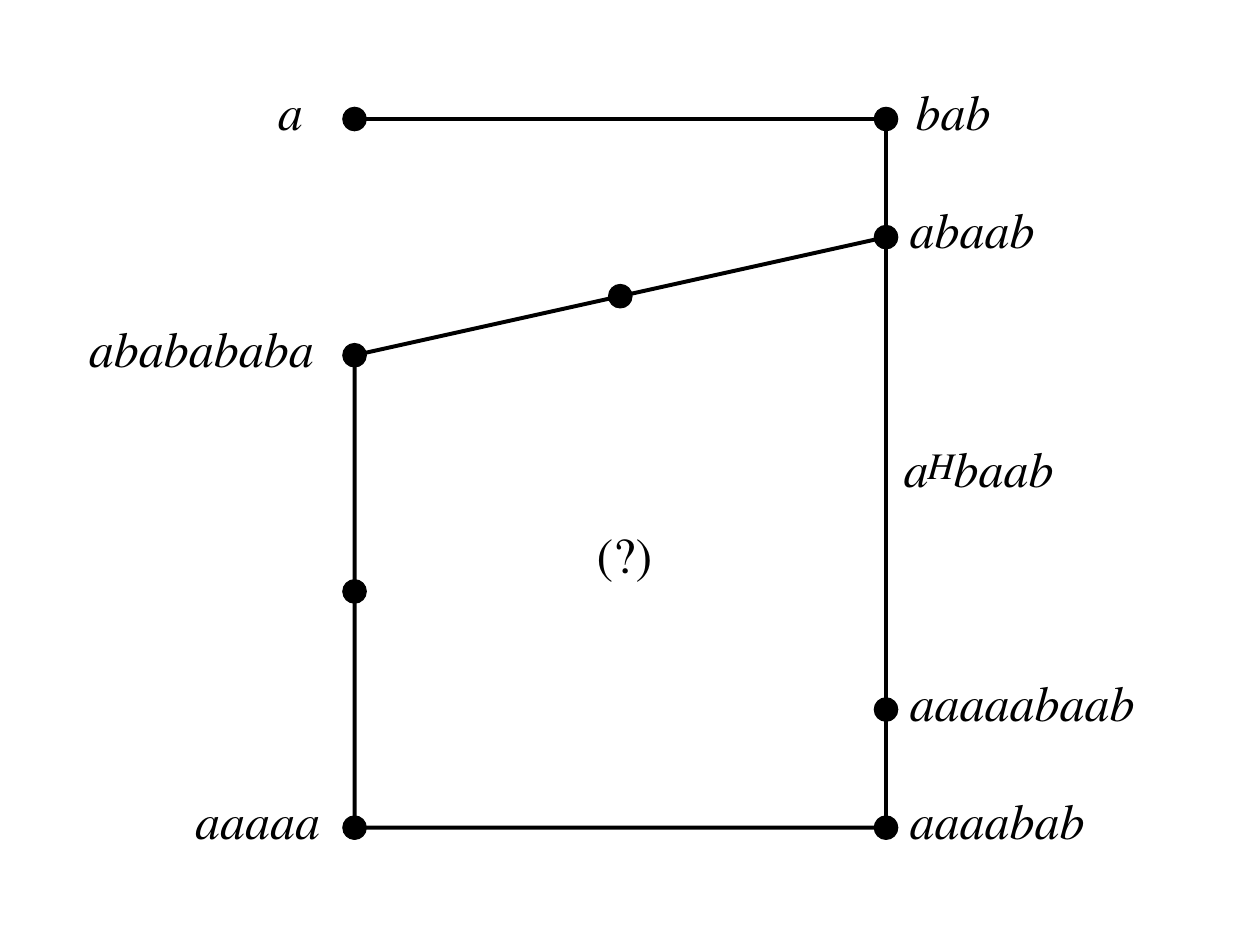}}
\caption{The (partial) definition of $[ba]^{H}$.}
\label{fig:baH}
\end{figure}

This can be done explicitly but is rather messy;
we prefer to use Lemma \ref{lemma:loose} and 
Proposition \ref{prop:ideal0isloose}.
The boundary of the trapezoidal region is a map
from $\Ss^1$ to $\cD_n[\Ideal_{(0)}]$ and therefore loose.
As a map from $\Ss^1$ to $\Omega\Spin_{n+1}$
it is homotopically trivial
since $\Omega\Spin_{n+1}$ is simply connected.
Thus, from Lemma \ref{lemma:loose},
the map is also homotopically trivial in $\cD_n$,
allowing us to fill the region, thus completing the
construction of $[ba]^H$ and the proof of the proposition.
\end{proof}


\bigbreak

\begin{proof}[Proof of Theorem \ref{theo:simplyconnected}]
From Proposition \ref{prop:simplyconnected},
any map from $\Ss^1$ to $\cD_n$ is loose.
As a map from $\Ss^1$ to $\Omega\Spin_{n+1}$
it is homotopically trivial
($\Omega\Spin_{n+1}$ being simply connected).
Thus, from Lemma \ref{lemma:loose},
any map from $\Ss^1$ to $\cD_n$ is homotopically trivial,
completing the proof of the theorem.
\end{proof}

\begin{rem}
\label{rem:baH}
It would be interesting to obtain an explicit solution
to complete Figure~\ref{fig:baH},
preferably with a small number of pieces.
\end{rem}

\section{The subcomplexes $\cY_{n,2} \subseteq \cD_n$}
\label{sect:Yn2}

We briefly recall some notation and results
from Section \ref{sect:permutations}.
We write $Y = S_{n+1} \smallsetminus S_{\PA}$
so that $\sigma \leftrightarrow \sigma'$ is an involution of $Y$:
for $\sigma \in Y$ take $k$ to be the smallest positive integer
with $k^\sigma \equiv (k+1)^\sigma \pmod 2$
and define $\sigma' = a_k \sigma$.
We call $\sigma \in Y$ low if $\sigma \blacktriangleleft \sigma'$
and high if $\sigma' \blacktriangleleft \sigma$.
As in Remark \ref{rem:Yk},
let $Y_k \subseteq Y$ be the set of permutations $\sigma$
which are either low with $\inv(\sigma) < k$
or high with $\inv(\sigma) \le k$.
Let $\Ideal_{Y_k} \subset \Word_n$ be the set of words
with \textit{at least} one letter in $Y_k$.

\begin{rem}
\label{rem:Yklower}
We saw in Examples \ref{example:IY} and \ref{example:IY2} 
that the subsets $\Ideal_Y, \Ideal_{Y_2} \subset \Word_n$
are lower subsets for both $\sqsubseteq$ and $\preceq$.
In this section we consider the lower set $\Ideal_{Y_2}$,
the open subset $\cL_n[\Ideal_{Y_2}] \subset \cL_n$
and the corresponding subcomplex
$\cY_{n,2} = \cD_n[\Ideal_{Y_2}] \subset \cD_n$.

We believe that for all $k > 2$ the subset
$\Ideal_{Y_k} \subseteq \Word_n$ is a lower set under $\preceq$,
but we shall neither prove nor use this claim.
Similarly, we believe (but do not prove)
that $\cL_n[\Ideal_{Y_k}] \subset \cL_n$
is an open subset.
Still, in the next section we will construct and study subcomplexes
$\cY_{n,k}$ corresponding to the sets $\Ideal_{Y_k}$.
Lemma \ref{lemma:YkisCW} below is good enough for our purposes.
\end{rem}

A subcomplex $\cX \subseteq \cD_n$ is \textit{nice} if
for every word $w \in \Word_n$
if $c_w \subseteq \cX$ then $c_{aaaaw} \subseteq \cX$;
the subcomplex $\cY_{n,2}$ is clearly nice.
Notice that if $\cX$ is nice and
$f: K \to \cX$ is a continuous map then
there exists a continuous map $aaaaf: K \to \cX$.
A nice subcomplex $\cX \subseteq \cD_n$ is \textit{loose} if
for every compact manifold $K$ and
every continuous map $f_0: K \to \cX$
there exists a homotopy $H: [0,1] \times K \to \cX$
from $f_0$ to $f_1 = aaaaf_0$.
More explicitly, we have
$H(0,s) = f_0(s)$ and $H(1,s) = aaaaf_0(s)$ for all $s \in K$.
Notice that we require the image of the homotopy $H$
to be included in $\cX$.
Contrast this with the definitions of
a weakly loose ideal and of a (strongly) loose ideal
in Section \ref{sect:looseideals}.
Clearly, if $\Ideal \subset \Word_n$ is a (strongly) loose ideal
then $\cD[\Ideal] \subset \cD_n$ is a loose subcomplex.
It follows from Remark \ref{rem:ideal0}
that the subcomplex $\cD[\Ideal_{[0]}] \subset \cD_n$
is not a loose subcomplex.

\begin{prop}
\label{prop:Yn2}
The ideal $\Ideal_{Y_2} \subset \Word_n$ is (strongly) loose.
\end{prop}

The previous proposition is equivalent to saying that
the subcomplex $\cY_{n,2} \subseteq \cD_n$ is loose.
Its proof relies strongly on
the proof of Proposition \ref{prop:ideal0isloose} above.
A few remarks will hopefully make the proofs easier to follow.


\begin{rem}
\label{rem:Hwarmup}
For every $k \in \nmesmo$,
we constructed in Example \ref{example:earlyaiH}
a path $a_k^H$ in the $1$-skeleton of $\cY_{n,2}$
from $a_k$ to $aaaaa_k$.
These paths were used
in the proofs of Propositions
\ref{prop:ideal0isloose} and \ref{prop:simplyconnected}.

Recall that $a' = [ba]$ and $a'_{k+1} = [a_ka_{k+1}]$:
in all cases, the first edge in $a_k^H$ is therefore $a'_k$.
The second vertex of $a'_k$ is $a_k^\star$:
\[ a^\star = bab, \quad b^\star = aba, \quad a_{k+1}^\star = a_ka_{k+1}a_k. \]
After that, every vertex and edge corresponds to a word
with at least one letter of dimension $0$.
Thus, all vertices and edges in $a_k^H$
are labeled by words which contain at least one letter from $Y_2$:
such words therefore belong to $\Ideal_{Y_2}$.
Thus, all paths are in $\cY_{n,2}$.

In the proof of Proposition \ref{prop:simplyconnected}
we also constructed $(a'_k)^H$,
a map from the square to $\cD_n$
which equals $a'_k$ at the top, $aaaaa'_k$ at the bottom,
$a_k^H$ at one side and $a_{k-1}^Ha_ka_{k-1}$ at the other side.
Since $a_k^H$ is constructed to be the concatenation 
of the other three sides, this map is rather trivial,
a mere stretching of the path $a_k^H$.
Again, all words  belong to $\Ideal_{Y_2}$
and therefore the map assumes values in $\cY_{n,2}$.

We completed this construction by considering $a' = [ba]$, 
with boundary points $a$ and $bab$.
In Figure \ref{fig:baH} we construct
(or at least show how to construct)
a map $(a')^H = [ba]^H$ from the square to $\cD_n$
which equals $a'$ at the top, $aaaaa'$ at the bottom,
$a^H$ at one side and $b^Hab$ at the other side.
The map also assumes values in $\cY_{n,2}$.
This is clear for the explicit part of the figure;
in the remaining part we may also assume that at least
one letter of dimension $0$ is present.
\end{rem}
'
\begin{rem}
\label{rem:HisinY2}
Recall from Remark \ref{rem:ideal0} the maps
$f_0, f_1: \Ss^0 \to \cD_n[\Ideal_{[0]}] \subset \cD_n$ with
$f_0(+1) = a$, $f_0(-1) = aaaaa$, $f_1 = aaaaf_0$.
As we have seen,
the maps $f_0$ and $f_1$ are homotopic in $\cD_n$
but not in $\cD_n[\Ideal_{[0]}]$.
Indeed, the only way to move away from the vertex $a$
is through the edge $[ba]$,
which joins it to the vertex $bab$.
Since $[ba] \in Y_2$,
it is clear that the obvious homotopy from $f_0$ to $f_1$
assumes values in $\cY_{n,2}$,
consistently with Proposition \ref{prop:Yn2}.

More generally, in the proof of Proposition \ref{prop:ideal0isloose}
we show that there exists a homotopy $H$
from a continuous map $f_0: K \to \cD_n[\Ideal_{[0]}]$ to $aaaaf_0$.
The proof shows how to construct $H$,
or at least parts of $H$:
we pass from $f_0$ to $f_1$, then to $f_2$ and $f_3$.
(The functions $f_\ast$
are as in the proof of Proposition \ref{prop:ideal0isloose}.)
It follows from Remark \ref{rem:ideal0} 
that in general
the image of the homotopy $H$ can not possibly be contained
in $\cL_n[\Ideal_{[0]}]$.
Indeed, in the passage from $f_0$ to $f_1$
we typically move from $(p_1,a_i,p_2)$ to $(p_1,a_i^\star,p_2)$
through $c_{w_1} \times c_{a'_i} \times c_{w_2}$.
Unless $w_1$ or $w_2$ contain letters of dimension $0$
(which we have no reason to expect)
then $w_1a'_iw_2 \notin \Ideal_{[0]}$.

On the other hand, from $f_1$ onwards,
the homotopy $H$ assumes values in $\cL_n[\Ideal_{[0]}]$.
Indeed, the curves then contain a large number of loops.

Also, the homotopy from $f_0$ to $f_1$
assumes values in $\cL_n[\Ideal_{Y_2}] \supset \cL_n[\Ideal_{[0]}]$.
Indeed, the paths $\delta_i$ assume values in $\cL_n[\Ideal_{Y_2}]$,
as we saw in Remark \ref{rem:Hwarmup}.

We thus see that, for $f_0$ as above, there exists a homotopy
$H: [0,1] \times K \to \cL_n[\Ideal_{Y_2}]$
from $f_0$ to  $aaaaf_0$.
\end{rem}

\begin{proof}[Proof of Proposition \ref{prop:Yn2}]
The idea is to follow the proof of
Proposition \ref{prop:ideal0isloose},
performing the necessary adaptations and
verifying that the homotopies assume values in $\cY_{n,2}$.
We will notice that part of the proof which requires
significant comments and adaptations is the passage
from $f_0: K \to \cD_n$ to $f_1$.
The functions $f_\ast$ are as in the proof of 
Proposition \ref{prop:ideal0isloose};
thus, curves in the image of $f_1$ have multiple loops.

Recall that we begin the proof of Proposition \ref{prop:ideal0isloose},
by taking a finite cover of $K$ by sets $U_j$.
In that proof we have that $p \in U_j$ implies
$f_0(p) \in c_{w(p)}$
with $w(p)$ of the form $w_{-}(p) a_{i_j} w_{+}(p)$.
In the present proof we instead have that
$p \in U_j$ implies
$f_0(p) \in c_{w(p)}$
with $w(p)$ of one of the three forms:
$w_{-}(p) a_{i_j} w_{+}(p)$,
$w_{-}(p) a'_{i_j} w_{+}(p)$ or
$w_{-}(p) a^{\star}_{i_j} w_{+}(p)$.
Construct maps $\delta_{i}: [0,1] \to \cD_n$
as in Equation \eqref{equation:delta}
(in the proof of Proposition \ref{prop:ideal0isloose})
satisfying 
$\delta_{i}(t) \in c_{a'_{i}}$ for $t \in (0,\frac12)$,
$\delta_{i}(\frac12) \in c_{a^{\star}_{i}}$
and assuming values in $\cY_{n,2}$.
Instead of having $f_0$ as in Equation \eqref{equation:f0},
we now have
\[ f_0(p) =
(p_1,\delta_{i_{j_1}}(s_{0,j_1}), \ldots,
\delta_{i_{j_k}}(s_{0,j_k}), p_{k+1}), \]
with $s_{0,j_\ast} = s_{0,j_\ast}(p) \in [0,\frac12]$.
Recall that in the construction of $f_1$
in proof of Proposition \ref{prop:ideal0isloose},
we define $s_{j_\ast} \in [0,1]$,
continuous functions of $p \in K$.
In the old proof, the functions $s_{j_\ast}$ are bumps,
equal to zero near the boundary of $U_{j_\ast}$;
it is crucial that for every $p$ there exists
at least one $j_\ast$ such that $p \in U_{j_\ast}$
and $s_{j_\ast}(p) = 1$.
Define new function $s_{j_\ast}$ by
$s_{j_\ast}(p) = \max\{s_{0,j_\ast}(p), s^{\old}_{j_\ast}(p)\}$
so that we have $s_{0,j_\ast}(p) \le s_{j_\ast}(p)$ (for all $p$).
The new functions  $s_{j_\ast}$ are also bumps,
equal to $s_{0,j_\ast}$ near the boundary of $U_{j_\ast}$;
the crucial property above still holds.
The function $f_1$ is now defined as in Equation \eqref{equation:f1},
using the new $s_{j_\ast}$:
\[
f_1(p) = 
(p_1,\delta_{i_{j_1}}(s_{j_1}), \ldots,
\delta_{i_{j_k}}(s_{j_k}), p_{k+1}). 
\]
Notice that (for all $p$) $f_1(p)$ has multiple loops,
as desired.
Finally, for the homotopy $H_1$ from $f_0$ to $f_1$,
instead of Equation \eqref{equation:H1}, we now write
\[
H_1(p,s) = 
(p_1,\delta_{i_{j_1}}(\med(s_{0,j_1},s,s_{j_1})), \ldots,
\delta_{i_{j_k}}(\med(s_{0,j_k},s,s_{j_k})), p_{k+1}); 
\]
here $\med$ denotes the median among three real numbers.
Notice that $f_0$, $f_1$ and $H_1$ all
assume values in $\cY_{n,2}$.

As discussed in Remark \ref{rem:HisinY2},
the remainder of the proof (or construction) preserves at least one 
letter of dimension $0$, so that the remaining homotopies
assume values in $\cY_{n,2}$, completing the proof.
\end{proof}


\section{Proof of Theorem \ref{theo:Y}}
\label{sect:Y}

Let $\Ideal \subseteq \Word_n$:
if $\Ideal$ is a lower set (for either $\sqsubseteq$ or $\preceq$)
then $\cD_n[\Ideal] \subset \cD_n$ is a subcomplex
(and therefore a closed subset).
Here, of course, $\cD_n[\Ideal]$ is the subcomplex
with cells $c_w$, $w \in \Ideal$.
On the other hand, at least in principle,
we do not need to prove that $\Ideal$ is a lower set
to deduce that $\cD_n[\Ideal]$ is a subcomplex.
Indeed, we presently define a weaker condition.

For any $w_0 \in \Word_n$, consider the glueing map
$g_{w_0}: \partial\DD^{\dim(w_0)} \to \cD_n$
(as constructed in Section \ref{sect:valid}).
The image of this map is a compact set
$B_{w_0} = g_{w_0}[\partial\DD^{\dim(w_0)}] \subset \cD_n$.
The construction of the CW complex $\cD_n$
can be performed in such a way that if $w_0, w_1$ are words
with $\dim(w_1) = \dim(w_0) - 1$ and $c_{w_1} \not\subseteq B_{w_0}$ 
then the cell $c_{w_1}$ is not necessary in order to glue $c_{w_0}$.
More precisely, we must then have $c_{w_1} \cap B_{w_0} \subseteq B_{w_1}$.
Indeed, if the glueing map $g_{w_0}$ touches $c_{w_1}$
in a non-surjective manner then we may deform $g_{w_0}$
so that $g_{w_0}$ does not touch the interior of $c_{w_1}$.
Notice that such a deformation
does not change the homotopy type of the complex.
From now on we assume that $\cD_n$ satisfies this condition.

The subset $\Ideal \subseteq \Word_n$ is \textit{subcomplex compatible} if
for all $w_0 \in \Ideal$ and
for all $w_1 \in \Word_n$ with $\dim(w_1) = \dim(w_0) - 1$
and $c_{w_1} \subseteq B_{w_0}$ we have $w_1 \in \Ideal$.
If $\Ideal$ is subcomplex compatible then
$\cD_n[\Ideal] \subseteq \cD_n$ is a subcomplex.
Also,
if $\Ideal$ is a lower set (for either $\sqsubseteq$ or $\preceq$)
then $\Ideal$ is subcomplex compatible.

\begin{rem}
\label{rem:subcomplexcompatible}
The concept of subcomplex compatible depends
on arbitrary and unspecified choices in the construction of $\cD_n$.
For instance, the following set is a lower set and therefore
subcomplex compatible:
\[ \Ideal_1 = \{ b, aba, cbc, acbca, acbac, cabac,
[ab], [cb], a[cb]a, c[ab]c, acb[ac], [ac]bac \}. \]
The set $\Ideal_2 = \Ideal_1 \cup \{[acb]\}$ is not a lower set
for either $\sqsubseteq$ or $\preceq$;
is it subcomplex compatible?
It depends on our choice of the cell $c_{[acb]}$
(or of its glueing map).
Two possibilities are shown in Figure \ref{fig:newacb}:
if we choose the left one then yes, $\Ideal_2$ is subcomplex compatible,
if we choose the right one then no.
\end{rem}


\begin{lemma}
\label{lemma:YkisCW}
Consider $k \ge 2$ and, for $\sigma_0 \in Y$,
let $B_{\sigma_0} \subset \cD_n$
be the image of its glueing map.
\begin{enumerate}
\item{If $\sigma_0 \in Y_k$ then
$B_{\sigma_0} \subseteq \cD_n[\Ideal_{Y_k}]$.}
\item{If $\sigma_0 \in Y_{k+1}$ is low and $\inv(\sigma_0) = k$ then
$B_{\sigma_0} \subseteq \cD_n[\Ideal_{Y_k}]$.}
\item{If $\sigma_0$ is high and $\inv(\sigma_0) = k+1$ then
$B_{\sigma_0} \subseteq \cD_n[\Ideal_{Y_k}] \cup c_{\sigma'_0}$.}
\end{enumerate}
In particular, $\Ideal_{Y_k} \subset \Word_n$ is subcomplex compatible.
\end{lemma}

\begin{proof}
Consider
$w_1 = \sigma_{1,1}\cdots\sigma_{1,\ell} \in \Word_n$
such that $\dim(w_1) = \dim(\sigma_0) - 1$
and $c_{w_1} \subseteq B_{\sigma_0}$.
We know from Example \ref{example:IY} that $w_1 \in Y$
so that there exists $j$ with $\sigma_{1,j} \in Y$.
If $\inv(\sigma_0) \le k$ then $\dim(w_1) \le k-2$ 
and therefore $\inv(\sigma_{1,j}) < k$
so that $\sigma_{1,j} \in Y_k$
and therefore $w_1 \in \Ideal_{Y_k}$.
If $\sigma_0$ is high and $\inv(\sigma_0) = k+1$ then
we may still have $\inv(\sigma_{1,j}) < k$,
which implies $w_1 \in \Ideal_{Y_k}$.
The other possibility is $\inv(\sigma_{1,j}) = k$
and therefore $\sigma_{1,j} \vartriangleleft \sigma_0$;
it then follows from computing dimensions that
all other letters in $w_1$ have dimension $0$.
If $\ell > 1$ there are
other letters $\sigma_{1,j'}$ with $\inv(\sigma_{1,j'}) = 1$
and therefore again $w_1 \in \Ideal_{Y_k}$.
If $\sigma_{1,j} \not\blacktriangleleft \sigma_0$
we have $\hat\sigma_{1,j} \ne \hat\sigma_0 = \hat w_1$
and therefore $\ell > 1$
(compare with Lemma \ref{lemma:triangleleft}).
If $\sigma_{1,j} \blacktriangleleft \sigma_0$
and $\sigma_{1,j} \ne \sigma'_0$ then (from Lemma \ref{lemma:vader})
we have that $\sigma_{1,j}$ is high 
and therefore $\sigma_{1,j} \in Y_k$
and therefore also in this case $w_1 \in \Ideal_{Y_k}$.
The last possible case is $j = \ell = 1$ and $\sigma_{1,j} = \sigma'_0$.
This completes the proof of the itemized claims.

In order to prove the final claim we need to consider
the more general case $w_0 = w_a \sigma_0 w_b$
and $w_1 = w_a \tilde w_1 w_b$.
But this case is analogous, except that the words $w_a$ and $w_b$
may occasionally help us by containing a letter in $Y_k$.
\end{proof}

We are interested in the subcomplexes
$\cY_{n,k} = \cD_n[\Ideal_{Y_k}]$.
For $m = \inv(\eta)$, we have $Y_m = Y$ and therefore
$\cY_{n,m} = \cD_n[\Ideal_Y]$.
Notice that from Theorem \ref{theo:newCW}
the CW complex $\cY_{n,m}$ is homotopy equivalent
to the Hilbert manifold $\cY_n$.

We are almost ready to prove Theorem \ref{theo:Y}:
let us first state a lemma.

\begin{lemma}
\label{lemma:Ynk}
For all $k \ge 2$,
the subcomplex $\cY_{n,k} \subseteq \cD_n$ is loose.
\end{lemma}

\begin{proof}
This proof is by induction on $k$,
the base case $k = 2$ being Proposition \ref{prop:Yn2} above.
Assume therefore that $k \ge 3$ and
that the subcomplex $\cY_{n,k-1} \subseteq \cD_n$ is known to be loose.
Consider a compact manifold $K$
and a continuous map $\alpha_0: K \to \cY_{n,k}$:
we claim that $\alpha_0$ is homotopic in $\cY_{n,k}$
to $\alpha_1: K \to \cY_{n,k-1}$.
Proving the claim completes the proof of the lemma.

Consider $\cY_{n,k}$ as a CW complex.
Let $\sigma \in Y_k \smallsetminus Y_{k-1}$ be a low permutation,
so that $\sigma' \in Y_k \smallsetminus Y_{k-1}$ be a high permutation,
$\inv(\sigma) = k-1$ and $\inv(\sigma') = k$. 
Proposition \ref{prop:blacktriangle} tells us that
there is a valid collapse starting from $\cY_{n,k}$
and removing the cells $c_{\sigma}$ and $c_{\sigma'}$:
\[ \cY_{n,k} \searrow
(\cY_{n,k} \smallsetminus ( c_{\sigma} \cup c_{\sigma'} ) ). \]
(We assume here only minimal knowledge of collapses;
see \cite{Cohen} or \cite{Forman} for much more on the subject.)
In other words, by applying a homotopy with values in $\cY_{n,k}$
to the map $\alpha_0$ we obtain a map which avoids
the cells $c_{\sigma}$ and $c_{\sigma'}$.
From Proposition \ref{prop:blacktriangle},
the new cells possibly touched by the new map are in $\cY_{n,k-1}$.

A similar collapse operation holds for pairs of the form
$c_{w_{-}\sigma w_{+}}$ and $c_{w_{-} \sigma' w_{+}}$.
Since $K$ is compact, a finite number of such collapses
obtains $\alpha_1$ with image contained in $\cY_{n,k-1}$.
\end{proof}

\begin{rem}
\label{rem:sigmaH}
The above proof can be rewritten as defining
maps $\sigma^H$ and $(\sigma')^H$ for
$\sigma \in Y_k \smallsetminus Y_{k-1}$ a low permutation.
For instance, consider $k = 3$, $\sigma = cb = [1342\cdots]$
and $\sigma' = acb = [3142\cdots]$.
The construction of $\sigma^H$ is illustrated in
Figure \ref{fig:cbH}.

More precisely, we take cells $\sigma'$ next to $\sigma$
and $aaaa\sigma'$ next to $aaaa\sigma$.
For another face $w$ of $\sigma'$, we have $w \in \Ideal_{Y_{k-1}}$
so that there is no difficulty in joining $w$ to $aaaaw$:
the map $w^H$ has already been constructed.
As with the pairs $(a_k,a'_k)$ ($k > 1$),
the construction of $\sigma^H$ gives us as a bonus $(\sigma')^H$.
\end{rem}

\begin{proof}[Proof of Theorem \ref{theo:Y}]
For $q \in \Quat_{n+1} \smallsetminus Z(\Quat_{n+1})$
we have $\cY_{n}(1;q) = \cL_{n}(1;q)$.
For $q \in Z(\Quat_{n+1})$,
the proper subset  $\cY_{n}(1;q) \subset \cL_{n}(1;q)$
is clearly open.
Apart from the contractible connected component of convex curves
(with itinerary equal to the empty word),
its closed complement
$\cM_{n}(1;q) = \cL_{n}(1;q) \smallsetminus \cY_{n}(1;q)$
is the union of the strata corresponding the words
formed by parity alternating permutations.
For $n \ne 3$, non trivial parity alternating permutation
has at least $3$ inversions;
for $n = 3$ at least $2$ inversions.
Thus $\cM_n(1;q) \smallsetminus \cL_{n,\conv}(1;q)$
is a closed subset of codimension at least $1$
(at least $2$ for $n \ne 3$).
This implies that $\cY_{n}(1;q)$ is open and dense in
$\cL_n({1;q}) \smallsetminus \cL_{n,\conv}(1;q)$,
as claimed.

We now prove that the inclusion
$\cY_{n}(1;q) \subset \Omega\Spin_{n+1}(1;q)$
is a weak homotopy equivalence.
Indeed, consider a continuous map
$\alpha_0: \Ss^k \to \Omega\Spin_{n+1}(1;q)$.
We already saw that the add-loop construction
obtains a homotopic map
$\alpha_1: \Ss^k \to \cL_n(1;q) \subset \Omega\Spin_{n+1}(1;q)$.
We may assume that after the add-loop construction,
every locally convex curve $\alpha_1(s)$ in the image of $\alpha_1$
begins with a non-convex closed arc.
By homotopy (more precisely, by connectivity of the non-convex
component of $\cL_n(1;1)$, as in Proposition \ref{prop:connected}),
this initial non-convex closed arc can be assumed
to have an itinerary starting with $aaaa$.
This implies that $\alpha_1$ has the form
$\alpha_1: \Ss^k \to \cY_n(1;q) \subset \Omega\Spin_{n+1}(1;q)$.

Similarly, consider $\alpha_0: \Ss^k \to \cY_n(1;q)$.
Assume there exists a $\beta_0: \DD^{k+1} \to \Omega\Spin_{n+1}(1;q)$,
$\beta_0|_{\Ss^k} = \alpha_0$.
By Lemma \ref{lemma:Ynk}, $\alpha_0$ is homotopic in $\cY_n(1;q)$
to $\alpha_1 = aaaaaaaa\alpha_0$.
From the add-loop construction,
$aaaa\alpha_0$ is homotopically trivial in $\cL_n(1;q)$,
that is, there exists
$\beta_2: \DD^{k+1} \to \cL_n(1;q)$ with
$\beta_2|_{\Ss^k} = aaaa\alpha_0$.
Define $\beta_1 = aaaa\beta_2$:
we have $\beta_1: \DD^{k+1} \to \cY_n(1;q)$ with
$\beta_1|_{\Ss^k} = \alpha_1$.
This completes the proof that the inclusion
$\cY_{n}(1;q) \subset \Omega\Spin_{n+1}(1;q)$
is a weak homotopy equivalence.
Since $\cY_{n}(1;q)$ is a Hilbert manifold,
the inclusion is also a homotopy equivalence \cite{Palais}.
\end{proof}

\section{Final Remarks}
\label{sect:finalremarks}

The present paper casts the foundations 
of a combinatorial approach
in the study of spaces of locally convex curves.
We proved a few results as a consequence
but we believe this approach can be used to prove far more.
We now state as a conjecture another result
we hope to prove using these methods in a forthcoming paper.
First, however, we review an older result.

Recall that $\Quat_3=\operatorname{Q}_8
=\{\pm1, \pm\mathbf{i}, \pm\mathbf{j}, \pm\mathbf{k}\}
\subset \Ss^3 \subset \HH$ 
where $\HH$ is the algebra of quaternions.
The main result in \cite{Saldanha3} classifies the spaces $\cL_2(q)$, 
$q\in\Quat_3$, into the following three weak homotopy types. 
For $q\neq\pm 1$, 
consistently with Corollary \ref{coro:notcenter},
we have 
$\cL_2(q)\approx\Omega\Ss^3$.
For $q \in \{\pm 1\}$ we have:
\begin{equation}
\label{equation:L2}
\cL_2(-1)\approx\Omega\Ss^3 \vee \Ss^0 \vee \Ss^4 \vee \Ss^{8} \vee \cdots, \qquad 
\cL_2(1)\approx\Omega\Ss^3 \vee \Ss^2 \vee \Ss^6 \vee \Ss^{10} \vee \cdots .
\end{equation}
For $n = 3$ we have
$\Spin_4 = \Ss^3 \times \Ss^3$ for $\Ss^3 \subset \HH$.
We also have that
$\Quat_4 \subset \operatorname{Q}_8\times\operatorname{Q}_8$
is generated by $(1,-1)$, $(\mathbf{i},\mathbf{i})$
and $(\mathbf{j},\mathbf{j})$.
This point of view is used in \cite{Alves, Alves-Saldanha} 
to obtain partial results concerning the homotopy type of $\cL_3$.

Theorem \ref{theo:Y} admits a minor improvement in the case $n=3$
which should be helpful.  Indeed, let 
\[ \tilde Y = Y \cup \{ ac, abc \} =
S_4 \smallsetminus \{ e, aba, bacb, bcb, cba, abacba = \eta \}. \]
Let $\tilde\cY_3 \subset \cL_3$ be the open set of curves 
whose itineraty admits at least one letter in $\tilde Y$;
we clearly have $\cY_3 \subset \tilde\cY_3 \subset \cL_3$.

\begin{theo}
\label{theo:tildeY}
For each $q \in \Quat_4$, the inclusion 
$\tilde\cY_3 \cap \cL_3(1;q) \subset \Omega\Spin_4(1;q)$
is a weak homotopy equivalence.
\end{theo}

\begin{proof}
Artificially define $[ac]' = [abc]$:
notice that $[ac] \blacktriangleleft [abc]$, as expected.
Proposition \ref{prop:blacktriangle} holds
for $\sigma_0 = [ac]$ and $\sigma_1 = [abc]$,
as has been observed in Example \ref{example:blacktriangle}
and can be verified from Figure \ref{fig:cD3}:
the cell $c_{[abc]}$ has a side $c_{[ac]}$;
the five other sides all include a letter of dimension $0$
(and are therefore in $\cY_{3,2}$).

The proof of the present result from Theorem \ref{theo:Y} (for $n=3$)
is now similar to a step in the proof of Lemma \ref{lemma:Ynk}:
perform the necessary collapses
for the pair $\sigma = [ac] \blacktriangleleft \sigma' = [abc]$. 
\end{proof}

The following theorem is the main result in \cite{Alves-Goulart-Saldanha}.

\begin{theo}
\label{theo:L3}
We have the following weak homotopy equivalences:
\begin{align*}
\cL_3(1;1) &\approx 
\Omega(\Ss^3 \times \Ss^3) \vee \Ss^4 \vee \Ss^8 \vee \Ss^8
\vee \Ss^{12} \vee \Ss^{12} \vee \Ss^{12} \vee \cdots, \\
\cL_3(1;-1) &\approx 
\Omega(\Ss^3 \times \Ss^3) \vee \Ss^2 \vee \Ss^6 \vee \Ss^6
\vee \Ss^{10} \vee \Ss^{10} \vee \Ss^{10} \vee \cdots, \\
\cL_3(1;-\hat a\hat c) &\approx 
\Omega(\Ss^3 \times \Ss^3) \vee \Ss^0 \vee \Ss^4 \vee \Ss^4
\vee \Ss^{8} \vee \Ss^{8} \vee \Ss^{8} \vee \cdots, \\
\cL_3(1;\hat a\hat c) &\approx 
\Omega(\Ss^3 \times \Ss^3) \vee \Ss^2 \vee \Ss^6 \vee \Ss^6
\vee \Ss^{10} \vee \Ss^{10} \vee \Ss^{10} \vee \cdots.
\end{align*}
The above bouquets include one copy of $\Ss^k$,
two copies of $\Ss^{(k+4)}$, \dots, $j+1$ copies of $\Ss^{(k+4j)}$, \dots,
and so on.
\end{theo}

Recall that, in the notation from \cite{Alves-Saldanha},
we have:
\begin{gather*}
\cL\Ss^3((+1,+1)) = \cL_3(1;1), \qquad
\cL\Ss^3((-1,-1)) = \cL_3(1;-1), \\
\cL\Ss^3((+1,-1)) = \cL_3(1;-\hat a\hat c), \qquad
\cL\Ss^3((-1,+1)) = \cL_3(1;\hat a\hat c).
\end{gather*}

It seems to be perhaps within reach but certainly harder
to use this combinatorial approach
to determine the homotopy type of
$\cL_n(1;q)$ for $n > 3$ and $q \in Z(\Quat_{n+1})$.
We hope to be able to prove
at least the following claim,
which should be contrasted with Corollary \ref{coro:notcenter}.

\begin{conj}
\label{conj:Lnlarge}
Consider $n > 3$ and $q \in Z(\Quat_{n+1})$.
Then $\cL_n(1;q)$ is
not homotopically equivalent to $\Omega\Spin_{n+1}$.
Moreover, if there are convex curves in $\cL_n(1;q)$
then the connected component $\cL_{n,\nconv}(1;q) \subset \cL_n(1;q)$
of non-convex curves is also
not homotopically equivalent to $\Omega\Spin_{n+1}$.
\end{conj}

\bigbreak

Another interesting aspect of the subject is its relation with 
the theory of differential operators, mentioned in the Introduction. 
A linear differential operator can be canonically associated 
to a nondegenerate curve $\gamma:[0,1]\to\Ss^n$ 
(see \cite{Khesin-Ovsienko}).
There is a Poisson structure in the space of the differential operators, given by the Adler-Gelfand-Dickey bracket 
\cite{Adler, Gelfand-Dickey1, Gelfand-Dickey2}. 
The identification above relates the spaces $\cL_n(1;z)$ with 
symplectic leaves of this strucuture \cite{Khesin-Ovsienko, Khesin-Shapiro1}. 
Notice that the spheres appearing in the bouquets in Equation 
\ref{equation:L2} and Theorem~\ref{theo:L3} are all even-dimensional.
We wonder whether this is fortuitous
or a manifestation of this symplectic structure;
this question is worth clarification.

\bibliography{gs}
\bibliographystyle{plain}

\footnotesize

\noindent
Victor Goulart \\
Departamento de Matem\'atica, UFES, \\
Av. Fernando Ferrari 514; Campus de Goiabeiras, Vit\'oria, ES 29075-910, Brazil. \\
\url{jose.g.nascimento@ufes.br}

\smallskip

\noindent
 Nicolau C. Saldanha \\
Departamento de Matem\'atica, PUC-Rio, \\
R. Marqu\^es de S. Vicente 255, Rio de Janeiro, RJ 22451-900, Brazil.  \\
\url{saldanha@puc-rio.br}

\end{document}